\documentclass[final,leqno,onefignum,onetabnum]{siamltex}
\usepackage[letterpaper,top=1.5in, bottom=1.5in, left=1in, right=1in]{geometry}

\usepackage{epsfig,epsf,fancybox}
\usepackage{amsmath}
\usepackage{mathrsfs}
\usepackage{amssymb}
\usepackage{amsfonts}
\usepackage{graphicx}
\usepackage{color}
\usepackage[linktocpage,colorlinks,linkcolor=blue,anchorcolor=blue,citecolor=magenta,urlcolor=cyan,hypertexnames=false]{hyperref}
\usepackage{boxedminipage}
\usepackage{stmaryrd}
\usepackage{multirow}
\usepackage{booktabs}
\usepackage[linesnumbered,ruled,vlined]{algorithm2e}\usepackage{algpseudocode}
\usepackage{cite}
\usepackage{float}

\usepackage[x11names]{xcolor}
\usepackage{tcolorbox}
\usepackage{braket,mathdots}
\usepackage{mathtools}
\usepackage[thinlines]{easytable}
\usepackage{array}

\usepackage{graphicx}
\usepackage{subcaption} 
\usepackage{mdframed}      
\usepackage{caption}

\allowdisplaybreaks[4]
\usepackage{mathtools}

\usepackage{mathtools}

\usepackage[normalem]{ulem} 

\newcommand{\bm}[1]{\boldsymbol{#1}}

\newcommand{\va}{{\mathbf{a}}}
\newcommand{\vb}{{\mathbf{b}}}

\newcommand{\vd}{{\mathbf{d}}}

\newcommand{\vh}{{\mathbf{h}}}

\newcommand{\vq}{{\mathbf{q}}}
\newcommand{\vr}{{\mathbf{r}}}
\newcommand{\vs}{{\mathbf{s}}}
\newcommand{\vt}{{\mathbf{t}}}
\newcommand{\vu}{{\mathbf{u}}}
\newcommand{\vv}{{\mathbf{v}}}
\newcommand{\vw}{{\mathbf{w}}}
\newcommand{\vx}{{\mathbf{x}}}
\newcommand{\vy}{{\mathbf{y}}}
\newcommand{\vz}{{\mathbf{z}}}

\newcommand{\vH}{{\mathbf{H}}}

\newcommand{\vJ}{{\mathbf{J}}}

\newcommand{\vL}{{\mathbf{L}}}

\newcommand{\RR}{\mathbb{R}} 
\newcommand{\vzero}{\mathbf{0}} 

\newcommand{\st}{\mbox{ s.t. }}

\DeclareMathOperator*{\argmin}{arg\,min} 

\newcommand{\bc}{\begin{center}}
\newcommand{\ec}{\end{center}}

\newcommand{\bdm}{\begin{displaymath}}
\newcommand{\edm}{\end{displaymath}}

\newcommand{\beq}{\begin{equation}}
\newcommand{\eeq}{\end{equation}}

\newcommand{\bfl}{\begin{flushleft}}
\newcommand{\efl}{\end{flushleft}}

\newcommand{\bt}{\begin{tabbing}}
\newcommand{\et}{\end{tabbing}}

\newcommand{\beqn}{\begin{eqnarray}}
\newcommand{\eeqn}{\end{eqnarray}}

\newcommand{\beqs}{\begin{align*}} 
\newcommand{\eeqs}{\end{align*}}  

\newtheorem{assumption}{Assumption}

\numberwithin{equation}{section}
\numberwithin{theorem}{section}

\newcommand{\yx}[1]{\textcolor{black}{#1}}

\colorlet{shadecolor}{gray!40}

\usepackage{tcolorbox}
\newtcbox{\alertinline}[1][red]
  {on line, arc = 0pt, outer arc = 0pt,
    colback = #1!20!white, colframe = #1!50!black,
    boxsep = 0pt, left = 1pt, right = 1pt, top = 1pt, bottom = 1pt,
    boxrule = 0pt, bottomrule = 1pt, toprule = 1pt}

\begin{document}
	
	\title{Alternating Direction Method of Multipliers for nonlinear constrained convex problems and applications to distributed resource allocation and constrained machine learning
}
	
	\author{Zhengjie Xiong\thanks{\{xiongz4, xuy21\}@rpi.edu, Department of Mathematical Sciences, Rensselaer Polytechnic Institute, Troy, NY} \and Yangyang Xu$^*$}
	
	\date{\today}
		\maketitle
	
	
		
	\begin{abstract}
We study a class of structured convex optimization problems, which have a two-block separable objective and nonlinear functional 
constraints as well as affine constraints that couple the two block variables. 
Such problems naturally arise from distributed resource allocation and constrained machine learning. To 
achieve high communication efficiency for the distributed applications, 
we propose a nonlinear alternating direction method of multipliers  (NL-ADMM) 
that preserves the classical splitting structure while accommodating general convex functional constraints. Unlike existing ADMM variants for nonconvex constrained problems, the proposed method does not require smoothness of the objective functions or differentiability of the constraint mapping, by leveraging convexity of the considered problem. We establish global convergence and an ergodic $\mathcal{O}(1/k)$ convergence rate of NL-ADMM by assuming the existence of a KKT solution. The results extend those of ADMM for linearly constrained convex problems. Numerical experiments are conducted on two representative distributed tasks. The results on numerous instances demonstrate that NL-ADMM can achieve (in many cases) 100x 
higher communication efficiency 
than the classic augmented Lagrangian method and nearly 2x higher than the Douglas-Rachford operator splitting method, 
making the new method well suited for large-scale distributed learning systems.
	\end{abstract}
	
	\vspace{0.3cm}
	
\section{Introduction}
Large-scale structured optimization problems often arise in distributed environments with separable objectives and globally coupled constraints. On solving such problems, local computation can be carried out efficiently at individual workers, while the enforcement of global feasibility typically requires repeated communication and synchronization, causing communication cost to dominate overall performance and making communication efficiency a central algorithmic consideration. Motivated by these considerations, we study a structured convex optimization problem in the form of 
\beq\label{(P)}
\min_{\vx\in\RR^{n_1},\vy\in\RR^{n_2}} f(\vx)+g(\vy), 
\quad \st \vh(\vx) \leq B\vy, A\vx+C\vy=\vd,
\eeq
where \yx{$f$ and $g$ are proper closed convex but possibly non-differentiable functions}, 
\yx{$\vh=[h_1; \ldots; h_{m_1}]: \RR^{n_1}\to \RR^{m_1}$ is a nonlinear map with each component function being convex} but not necessarily differentiable, $A\in\RR^{m_2\times n_1},B\in\RR^{m_1\times n_2},C\in\RR^{m_2\times n_2}$ are given matrices and $\vd\in\RR^{m_2}$ is a given vector. 

\subsection{Motivating Applications}
\label{sec:applications}
Many applications can be formulated into \eqref{(P)} such as $\ell_1$-problems \cite{yang2011alternating} and multi-class support vector machines \cite{xu2015alternating}. Our main motivation is to design communication-efficient algorithms for distributed problems with coupling constraints. Below, we give two representative applications.
\subsubsection{Distributed resource allocation}

Distributed resource allocation arises in many engineered systems, such as network optimization \cite{halabian2019distributed}, multi-agent coordination \cite{falsone2016distributed}, and power systems  \cite{chen2016distributed}, where multiple workers (or agents) make decisions under 
a joint global resource constraint.
Although many distributed resource allocation problems are written with affine equality constraints, practical operation typically does not require saturating the resource or using its full capacity, and such affine cases are naturally included as special instances of our more general inequality formulation. These settings motivate a model where 
$p$ workers choose local decision vectors $\{\vx_{j}\}_{j=1}^{p}$ to minimize aggregate cost while jointly adhering to a global resource limit, leading to the 
resource-constrained program
\beq\label{Resource_allocation}
\min_{\vx_{1},\dots, \vx_{p}}\sum_{j=1}^{p}f_{j}(\vx_{j}),\quad \text{s.t.} \quad \sum_{j=1}^{p}h_{j}(\vx_{j})\leq 0.
\eeq
Here, each local objective function $f_j(\vx_j)$ models the individual cost or utility associated with worker $j$, while the constraint function $h_j(\vx_j)$ quantifies worker $j$’s contribution to the shared global resource constraint.

To obtain a form compatible with distributed optimization schemes, we rewrite the {coupling} constraint by using a slack variable $\vy\in \mathbb{R}^{p}$ such that $\sum_{j=1}^{p}\vy_{j}=0$ and requiring $h_{j}(\vx_{j})\leq \vy_{j}, \forall\, j$.
Let $\vx=[\vx_1;\ldots;\vx_p]$ and define
\begin{equation*}
f(\vx)=\sum_{j=1}^p f_j(\vx_j),\quad
\vh(\vx)=\begin{bmatrix}h_1(\vx_1)\\\vdots\\h_p(\vx_p)\end{bmatrix},\quad
g(\vy)=\iota_{\mathcal{Y}}(\vy), \quad \mathcal{Y}=\left\{\vy: \sum_{j=1}^{p} \vy_j = 0\right\}.  
\end{equation*}
Then problem \eqref{Resource_allocation} can be equivalently formulated to 
\beq\label{eq:resource_allocation}
\min_{\vx,\vy}f(\vx)+g(\vy),\quad \text{s.t.} \quad \vh(\vx)\leq \vy,
\eeq
which is a special case 
of 
\eqref{(P)} with  $A=\mathbf{0}$, $B=I$, $C=\mathbf{0}$, and $\vd=\mathbf{0}$.

\subsubsection{Distributed machine learning}
Many constrained learning tasks, such as fairness-aware empirical risk minimization \cite{donini2018empirical} and Neyman–Pearson classification \cite{rigollet2011neyman},  require minimizing a loss function while ensuring that a constraint-measuring statistic (e.g., false positive rate, group fairness violation) does not exceed a prescribed threshold. 
These tasks can be formulated as 
the following constrained empirical risk minimization:
\begin{equation}\label{eq:ml_base}
    \min_{\vx_{0}} \frac{1}{N}\sum_{i=1}^{N} f(\vx_{0}, \bm{\bm{\xi}}_{i}^{1}),
    \quad \text{s.t.} \quad
    \frac{1}{M}\sum_{i=1}^{M} h(\vx_{0}, \bm{\bm{\xi}}_{i}^{2}) \leq r,
\end{equation}
where the loss term $f(\vx_{0},\bm{\xi}_{i}^{1})$ measures prediction error  on 
a sample $\bm{\xi}_{i}^{1}$ (e.g., logistic or squared loss or hinge loss), $h(\vx_{0}, \bm{\xi}_{i}^{2})$ measures constraint-related statistic (e.g., violation score) on 
a sample $\bm{\xi}_{i}^{2}$ and $r>0$ denotes a given threshold.
Suppose the data is too large and cannot fit into one computing unit. Then to solve \eqref{eq:ml_base}, 
we partition the 
data samples and distribute them across 
$p$ workers. Let the partition be 
$\{1,\cdots,N\} = \bigcup_{j=1}^{p}\mathcal{A}_j,\  
\{1,\cdots, M\}= \bigcup_{j=1}^{p}\mathcal{B}_j,$    
where $\mathcal{A}_j$ and $\mathcal{B}_j$ represent the local index subsets assigned to 
{worker} $j$. Then problem \eqref{eq:ml_base} can be rewritten as
\begin{equation*}\label{eq:ml_sum}
    \min_{\vx_{0}}
    \frac{1}{N}\sum_{j=1}^{p}\sum_{i\in\mathcal{A}_j} f(\vx_{0}, \bm{\xi}_i^{1}),
    \quad \text{s.t.} \quad
    \frac{1}{M}\sum_{j=1}^{p}\sum_{i\in\mathcal{B}_j} h(\vx_{0}, \bm{\xi}_i^{2}) \leq r.
\end{equation*}
Introducing local copies $\vx_j$ of the global variable $\vx_0$, 
we obtain the equivalent distributed form
\begin{equation}\label{eq:ml_dist}
\begin{aligned}
    \min_{\vx_0, \vx_1, \ldots, \vx_p}
    &\quad \frac{1}{N}\sum_{j=1}^{p}\sum_{i\in\mathcal{A}_j} f(\vx_j, \bm{\xi}_i^{1}),\ 
    \text{ s.t. }\ 
    \frac{1}{M}\sum_{j=1}^{p}\sum_{i\in\mathcal{B}_j} h(\vx_j, \bm{\xi}_i^{2}) \leq r,\ 
    \vx_j = \vx_0, \, j = 1, \ldots, p.
\end{aligned}
\end{equation}
For each {worker} 
$j$, define  the local objective and constraint functions by
\beq\label{eq:consensus_setup}
f_j(\vx_j) = \frac{1}{N}\sum_{i\in\mathcal{A}_j} f(\vx_j, \bm{\xi}_i^{1}),
\quad
h_j(\vx_j) = \frac{1}{M}\sum_{i\in\mathcal{B}_j} h(\vx_j, \bm{\xi}_i^{2}) - \frac{r}{p}.
\eeq
Then \eqref{eq:ml_dist} can be rewritten as 
\begin{equation}\label{eq:ml_local_form}
\begin{aligned}
    \min_{\vx_0, \vx_1,\dots,\vx_p} 
    \sum_{j=1}^{p} f_j(\vx_j),\ 
    \text{ s.t. }\ 
    \sum_{j=1}^{p} h_j(\vx_j) \leq 0,\ 
    \vx_j = \vx_0, \, j = 1,\dots,p.
\end{aligned}
\end{equation}
Again, introduce slack variables $\vy\in \mathbb{R}^{p}$ with $\sum_{j=1}^p \vy_j=0$. Finally, 
we 
let $\vx = [\vx_1;\dots;\vx_p]$ and 
$\tilde{\vy} = [\vy; \vx_0]$. 
Define
\begin{equation*}
f(\vx)=\sum_{j=1}^p f_j(\vx_j),\quad
\vh(\vx)=\begin{bmatrix}h_1(\vx_1)\\\vdots\\h_p(\vx_p)\end{bmatrix},\quad
g(\tilde{\vy})=\iota_{\mathcal{Y}}(\vy), \quad \mathcal{Y}=\left\{\vy: \sum_{j=1}^{p} \vy_j = 0\right\}.    
\end{equation*}
Then \eqref{eq:ml_local_form} is equivalent to 
\begin{equation}
\label{eq:consensus}
\min_{\vx,\tilde{\vy} } f(\vx)+g(\tilde{\vy} ),\quad
\text{s.t.}\quad \vh(\vx)\le \vy,\quad \vx-(\mathbf{1}\otimes I) \vx_0=\vd,
\end{equation}
which is again a special case of \eqref{(P)} with
$A=I$, $B=[\,I\ \ \mathbf{0}\,]$,
$C=[\,\mathbf{0}\ \ -\mathbf{1}\otimes I\,]$ and $\vd=\mathbf{0}$.


\subsection{Algorithm}

\begin{algorithm}[t]
\captionsetup{labelfont=bf}
\caption{Nonlinear Alternating Direction Method of Multipliers (NL-ADMM) for \eqref{(P)}}
\label{alg:nl-admm}
\begin{mdframed}[leftline=false,rightline=false,topline=false,bottomline=false,
                 linewidth=1pt,innertopmargin=.35\baselineskip,innerbottommargin=.35\baselineskip]
\textbf{Initialization:} choose \(\vy^{0}\), \(\vu_{1}^{0}\!\ge\!\mathbf{0}\), \(\vu_{2}^{0}\), $\beta_{1},\beta_{2}>0$, $\gamma_{1},\gamma_{2}\in(0,\frac{1+\sqrt{5}}{2})$
\medskip

\For{\(k=0,1,2,\ldots\)}{
\begin{subequations}\label{eq:NLADMM-steps}
\begin{align}
\vx^{k+1}&=\argmin_{\vx}f(\vx)+\frac{\beta_{1}}{2}\big\| [\vh(\vx)-B\vy^{k}+\vu_{1}^{k}]_{+}\big\|^{2}+\frac{\beta_{2}}{2}\big\| A\vx+C\vy^{k}-\vd+\vu_{2}^{k}\big\|^{2},\label{x_update}\\
\vs^{k+1}&=[-\vh(\vx^{k+1})+B\vy^{k}-\vu_{1}^{k}]_{+},\label{s_update}\\
\vy^{k+1}&=\argmin_{\vy}g(\vy)+\frac{\beta_{1}}{2}\big\| \vh(\vx^{k+1})+\vs^{k+1}-B\vy+\vu_{1}^{k}\big\|^{2}+\frac{\beta_{2}}{2}\big\| A\vx^{k+1}+ C\vy-\vd+\vu_{2}^{k}\big\|^{2},\label{y_update}\\
\vu_{1}^{k+1}&=\vu_{1}^{k}+\gamma_{1} (\vh(\vx^{k+1})+\vs^{k+1}-B\vy^{k+1}),\label{u1_update-0}\\
\vu_{2}^{k+1}&=\vu_{2}^{k}+\gamma_{2} (A\vx^{k+1}+C\vy^{k+1}-\vd). \label{u2_update-0}
\end{align}
\end{subequations}
\If{ \text{a stopping condition is satisfied} }{ Return \((\vx^{k+1},\vy^{k+1},\vu_{1}^{k+1},\vu_{2}^{k+1}).\)}}
\end{mdframed}
\end{algorithm}

A classical approach for solving~\eqref{(P)} is the augmented Lagrangian method (ALM) \cite{rockafellar1976augmented}. In distributed settings, applying ALM to~\eqref{(P)} typically requires solving each augmented Lagrangian subproblem by an iterative method, 
which in turn involves repeated aggregation of global objective and constraint information across workers. As a consequence, multiple rounds of communication between the server and workers are often needed within each outer iteration, leading to high communication cost, especially for the distributed formulations~\eqref{Resource_allocation} and~\eqref{eq:ml_base}.

To address this issue, we propose a nonlinear alternating direction method of multipliers (NL-ADMM) for~\eqref{(P)}. By splitting the augmented Lagrangian and updating primal variables in a blockwise fashion, NL-ADMM avoids repeatedly solving globally coupled subproblems. When applied to the distributed formulations~\eqref{eq:resource_allocation} and~\eqref{eq:consensus}, NL-ADMM requires only one communication round per iteration, resulting in significantly improved communication efficiency, as demonstrated in Section~\ref{sec:numerical}. Algorithm~\ref{alg:nl-admm} summarizes NL-ADMM. We allow separate penalty parameters for the nonlinear and linear constraints, as well as optional over-relaxation in the dual updates. 
Though an operator splitting method can be applied to~\eqref{eq:resource_allocation} and~\eqref{eq:consensus} and potentially achieve higher communication efficiency than the ALM, we observe that NL-ADMM performs significantly superior over the well-known Douglas--Rachford Splitting (DRS) method \cite{eckstein1992douglas}.

\subsubsection{Distributed implementation} 
We 
describe how Algorithm~\ref{alg:nl-admm} can be applied to the two distributed formulations in Section~\ref{sec:applications}. We adopt a server-worker setting; see Figure~\ref{fig:distributed_setup} for an illustration. 
For the \emph{distributed resource allocation} \eqref{eq:resource_allocation}, Algorithm~\ref{alg:nl-admm} specializes to the updates below: 
\begin{subequations}
\begin{align}
    \vx_{j}^{k+1}&=\argmin_{\vx_{j}} f_{j}(\vx_{j})+\frac{\beta_{1}}{2}[h_{j}(\vx_{j})-\vy_{j}^{k}+\vu_{1,j}^{k}]_{+}^{2}\label{distributed_implementation_1_x},\\
    \vs_{j}^{k+1}&=[-h_{j}(\vx_{j}^{k+1})+\vy_{j}^{k}-\vu_{1,j}^{k}]_{+},\label{distributed_implementation_1_s}\\
    \vy^{k+1} &=
    \vv-\frac{1}{p}\mathbf{1}\mathbf{1}^{\top }\vv, \text{ with }\vv=\vh(\vx^{k+1})+\vs^{k+1}+\vu_{1}^{k},\label{distributed_implementation_1_y}\\
    \vu_{1,j}^{k+1} &= \vu_{1,j}^{k}+\gamma_{1}(h_{j}(\vx_{j}^{k+1})+\vs_{j}^{k+1}-\vy_{j}^{k+1}),\label{distributed_implementation_1_u}
\end{align}
\end{subequations}
where \eqref{distributed_implementation_1_x}, \eqref{distributed_implementation_1_s}, and \eqref{distributed_implementation_1_u} are performed for each $j=1,\ldots,p$.

\begin{figure}[t]
    \centering
    \includegraphics[width=0.5\textwidth]{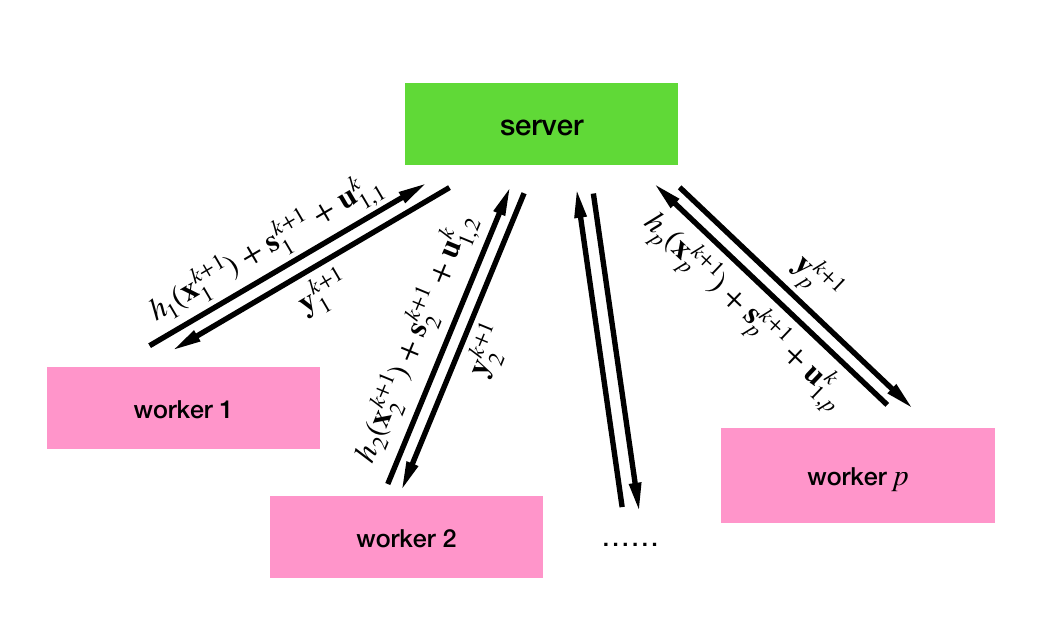}
    \caption{Distributed implementation of the proposed algorithm for the resource allocation model~\eqref{eq:resource_allocation}. At iteration $k$, each worker $j$ sends the locally computed quantity $h_j(\vx_j^{k+1})+\vs_j^{k+1}+\vu_{1,j}^{k}$ to the server, which aggregates these values and broadcasts the entries of the updated global variable $\vy^{k+1}$ to corresponding workers.}
    \label{fig:distributed_setup}
\end{figure}
The distributed iteration admits a precise sequential interpretation. 
At iteration $k$, we assume that $\vy_j^{k}$ and $\vu_{1,j}^{k}$ are available at worker $j$ (for $k=0$, $\vy_j^{0}$ and $\vu_{1,j}^{0}$ are given initial values). 
Each worker $j$ then updates its primal variable $\vx_j^{k+1}$ and slack variable $\vs_j^{k+1}$ according to \eqref{distributed_implementation_1_x} and \eqref{distributed_implementation_1_s}, and sends the quantity $h_j(\vx_j^{k+1})+\vs_j^{k+1}+\vu_{1,j}^{k}$ to the server. 
The server aggregates the returned values and updates the 
global variable $\vy^{k+1}$ based on \eqref{distributed_implementation_1_y}, whose entries are subsequently broadcast to corresponding workers. 
Upon receiving $\vy_{j}^{k+1}$, each worker updates its local dual variable $\vu_{1,j}^{k+1}$ using $(\vx_j^{k+1},\vs_j^{k+1},\vy_{j}^{k+1})$, which corresponds to \eqref{distributed_implementation_1_u}. 
Consequently, each iteration consists of a single aggregation step from the workers to the server and a single broadcast step from the server to the workers, and thus incurs only one communication round per iteration.

For the \emph{constrained machine learning model}~\eqref{eq:consensus}, the distributed implementation follows the same sequential pattern as above, with the addition of terms enforcing the consensus constraint $\vx_j=\vx_0$. In particular, the $\vx$- and $\vy$-updates include the corresponding quadratic penalty and dual terms associated with the linear consensus constraint, and an explicit update of the global variable $\vx_0$ is introduced, which reduces to an averaging step across workers at the server. Aside from these additional consensus-related terms, the structure of the local and global updates and the communication pattern remain unchanged.

\subsection{Contributions} 

\yx{Our contributions are three-fold and summarized below.
\begin{itemize}
\item First, we develop a nonlinear alternating direction method of multipliers (NL-ADMM) for a class of two-block structured convex optimization problems with nonlinear functional constraints. Unlike existing ADMM variants that are tailored to linear constraints or rely on smoothness or differentiability assumptions, the proposed framework accommodates general convex, possibly nonsmooth, nonlinear constraints while preserving the classical ADMM splitting structure. To the best of our knowledge, this work constitutes 
the first ADMM-type framework that systematically handles convex optimization problems with nonlinear functional constraints without resorting to linearization or smoothing techniques.
\item Second, under a standard KKT existence assumption, we establish global convergence of NL-ADMM and further derive an ergodic $\mathcal{O}(1/k)$ convergence rate. The analysis is carried out under minimal regularity conditions: it relies solely on convexity and does not require smoothness of the objective functions, differentiability or Lipschitz continuity of the nonlinear constraint mapping, nor rank conditions on the associated linear operators. These results extend classical ADMM convergence theory beyond the linear-constraint setting.
\item Finally, we demonstrate the practical effectiveness of NL-ADMM on representative distributed resource allocation and constrained machine learning problems. The numerical results show that NL-ADMM consistently exhibits more favorable convergence behavior than the classical augmented Lagrangian method and the Douglas–Rachford Splitting method, particularly in terms of communication efficiency across a range of distributed problem settings.
\end{itemize}
}

\subsection{Notations} 
We use $\lVert\cdot\rVert$ for the Euclidean norm. 
For any set $S\subseteq\mathbb{R}^{n}$, the $0$-$\infty$ indicator function $\iota_{S}:\mathbb{R}^{n}\to\mathbb{R}\cup\{\infty\}$ is defined by $\iota_{S}(\vx)=0$ if $\vx\in S$ and $+\infty$ otherwise.
For any vector $\vx$, we define $[\vx]_{+} = \max\{\vx, 0\}$. 
For any $\va,\vb\in\mathbb{R}^{n}$, we denote $\va\odot\vb$ as their element-wise product. The Kronecker product of matrices $A$ and $B$ is denoted by $A\otimes B$. We use $I$ for the  identity matrix, $\mathbf{1}$ for the  all-one vector, and $\mathbf{0}$ for a zero vector or a zero matrix. 
The subdifferential of a convex function $f$ at $\vx$ is denoted by $\partial f(\vx)$. \yx{For a vector function $\vh = [h_1; \ldots; h_m]$ with each $h_j$ being convex, we denote $\partial \vh(\vx) = \big\{[\vz_1, \ldots, \vz_m]^\top: \vz_j\in \partial h_j(\vx), \forall \, j\big\}$.}

\section{Related Work}
The ADMM, 
first introduced by Glowinski and Marrocco \cite{glowinski1975approximation}, is a fundamental algorithm for solving structured convex optimization problems \yx{of the form}
\begin{equation*}
    \min_{\vx,\vy} f(\vx)+g(\vy) \quad \st A\vx+C\vy=\vd,
\end{equation*}
where $f$ and $g$ are closed convex functions. 
Classical convergence of ADMM for two-block structured convex optimization dates back to the operator-splitting work of Gabay and Mercier \cite{gabay1976dual} and the monotone-operator analysis of Eckstein and Bertsekas \cite{eckstein1992douglas}, and an ergodic rate of $\mathcal{O}(1/k)$ is established by Monteiro \& Svaiter in \cite{monteiro2013iteration} and He \& Yuan in \cite{he2012-rate-drs}, where 
$k$ denotes the iteration number. A comprehensive overview of the ADMM and its applications can be found in \cite{boyd2011distributed}. Building on these foundations, a wide range of extensions have been developed \yx{for different structured problems}. 

\subsection{Linear constrained problems}
Under linear constraints, we consider two main extensions of the classical two-block ADMM framework: multiblock formulations and problems with nonconvex objectives. Building on two-block convergence results, multiblock extensions have been extensively studied, and it is now well understood that direct extensions may diverge even in convex settings, as shown by Chen \yx{et al.} 
\cite{chen2016direct}.  One approach to restoring convergence is to modify the algorithm itself: a prediction–correction variant of three-block ADMM is proposed in~\cite{he2018class}, analyzed through a variational inequality reformulation, and shown to enjoy global convergence with an ergodic  $\mathcal{O}(1/k)$ convergence rate. A complementary line of work studies conditions under which the original three-block ADMM remains convergent; in particular, when one block is smooth and strongly convex with its condition number lying in a prescribed interval, global convergence holds for any penalty parameter~\cite{lin2018global}. A similar result is shown in \cite[Theorem 2.1]{davis2017three} as a special case of a three-operator splitting method.


Three-block results naturally extend to the multiblock setting \cite{xu2018hybrid-Jac-GS, gao2019rpd-bcu, deng2017parallel, lin2015global, chen2019unified, he2012alternating, melo2017iteration, xu2019asynchronous-pdbcu}, where a key distinction lies between Gauss–Seidel–type schemes, which update blocks sequentially using the latest iterates, and Jacobi-type schemes, which update all blocks in parallel using only previous-iteration information. A unified symmetric Gauss–Seidel (sGS)–based proximal ADMM framework is developed in~\cite{chen2019unified}, integrating majorization techniques, indefinite proximal terms, inexact updates, and sGS decompositions, and establishing global convergence. In contrast, a parallel Jacobi-type multiblock ADMM is studied in~\cite{deng2017parallel}, where convergence is guaranteed either under near-orthogonality and full column-rank assumptions on the constraint matrices or through blockwise proximal regularization, yielding an $\mathcal{O}(1/k)$  convergence rate. A hybrid Gauss-Seidel and Jacobi update scheme is presented in \cite{xu2018hybrid-Jac-GS} and also achieves the $\mathcal{O}(1/k)$  convergence rate.

In addition, there is a substantial body of work on ADMM for nonconvex optimization with linear constraints \cite{barber2024convergence, wang2019global, hong2016convergence, xu2012alternating, hong2017distributed, themelis2020douglas}. Rather than attempting a comprehensive review, we briefly mention a few representative results. Hong, Luo, and Razaviyayn \cite{hong2016convergence} study multi-block ADMM for a family of affinely constrained nonconvex optimization problems, 
with a focus on 
consensus and sharing formulations, and establish convergence of the generated iterates to stationary solutions under appropriate regularity conditions. Wang, Yin, and Zeng \cite{wang2019global} extend the analysis of \cite{hong2016convergence} to a broader class of 
affinely constrained optimization problems with nonconvex and possibly nonsmooth objectives, establishing subsequential convergence to stationary points and global convergence under the Kurdyka-{\L}ojasiewicz (KL) condition. 

Taken together, the above works establish a mature theory for ADMM under linear constraints, encompassing both multiblock formulations and nonconvex objective settings. 
When there are nonlinear functional constraints, 
this existing ADMM theory does not apply. 

\subsection{Nonlinear constrained problems} 
In contrast to the well-developed theory for ADMM under linear constraints, the literature on nonlinear variants remains relatively limited.  We are not aware of any work on ADMM 
for convex problems involving nonlinear functional constraints without resorting to 
linearization. Aybat and Hamedani \cite{aybat2019distributed} propose a distributed inexact primal–dual algorithm (DPDA-D) for resource sharing over time-varying networks by reformulating the problem as a saddle-point problem and applying a linearized PDA scheme combined with approximate consensus. They established equivalence between classic PDA and preconditioned ADMM for problems with linear coupling operators and accordingly characterized their method as ADMM-like rather than a pure ADMM scheme. However, when the coupling operator is nonlinear, their method relies on first-order linearization and no such equivalence is established. In contrast, our method directly handles nonlinear constraints without linearization.

Most existing works on ADMM with nonlinear constraints focus on nonconvex problems with nonlinear equality constraints. Cohen, Hallak, and Teboulle~\cite{cohen2022dynamic} propose a dynamic linearized ADMM for problems with nonlinear equalities $F(\vx)=G\vy$, assuming smooth objective components with Lipschitz properties, a constraint mapping $F$ with a locally Lipschitz Jacobian, and a full row rank matrix $G$. Under these assumptions and a bounded adaptive sequence of proximal parameters, they establish convergence to critical points and, under a KL condition, global convergence of the entire sequence. Zhu, Zhao, and Zhang~\cite{zhu2024first} develop a first-order primal--dual augmented-Lagrangian scheme for nonconvex problems with nonlinear equalities $\phi(\vx)+B\vy=\mathbf{0}$, assuming smooth objective terms with Lipschitz gradients, a constraint mapping $\phi$ that is Lipschitz continuous, and a full row rank matrix $B$. Their method linearizes the augmented terms at each iteration and, under these assumptions, achieves an $\epsilon$-stationary point within $\mathcal{O}(\epsilon^{-2})$ iterations, with global and linear convergence ensured by an additional error-bound condition. 

More recent works follow similar lines. Sun and Sun~\cite{sun2024dual} analyze a scaled dual-descent ADMM for \yx{problems with} smooth nonlinear equality constraints $\vh(\vx)=\mathbf{0}$, requiring $\vh$ to be continuously differentiable with Lipschitz \yx{continuous} and uniformly bounded Jacobian, and the smooth objective to possess a Lipschitz \yx{continuous} gradient; under these conditions, they 
establish iteration complexity results to produce an $\epsilon$-stationary point. 
Hien and Papadimitriou~\cite{hien2024multiblock} propose a multiblock majorization-based ADMM for \yx{problems with} nonlinear 
constraints $\phi(\vx)+B\vy=\mathbf{0}$, under assumptions including block-separable subdifferentials, Lipschitz continuity of the gradient of the differentiable term in the objective, and full row rankness of $B$.  
Within this framework, they design surrogate functions that satisfy curvature or strong-convexity conditions. 
Under these settings, they establish subsequential convergence to critical points, $\mathcal{O}(\epsilon^{-2})$ complexity, and global convergence under a KL property. El~Bourkhissi and Necoara~\cite{el2025convergence} 
study a closely related problem with nonlinear equality constraints of the form $F(\vx)+G\vy=\mathbf{0}$, where the nonlinearity involves a single primal block $\vx$ rather than the multiblock setting 
considered in~\cite{hien2024multiblock}. They propose an inexact linearized ADMM achieving $\mathcal{O}(\epsilon^{-2})$ complexity together with full-sequence convergence under KL property.

 Although a nonlinear inequality constraint can be converted into an equality one by introducing a nonnegative slack variable, 
 the above-mentioned convergence results 
 do not apply to our considered problem \eqref{(P)}. In particular, \cite{cohen2022dynamic,hien2024multiblock,el2025convergence,zhu2024first} require the objective component associated with the variable $\vy$ to have 
 a Lipschitz continuous gradient, which enables smoothness-based linearization or majorization of the augmented Lagrangian. 
 However, in \eqref{(P)}, both $f$ and $g$ can be nondifferentiable. 
Moreover, \cite{cohen2022dynamic,sun2024dual,el2025convergence} require the nonlinear constraint mapping to have a Lipschitz continuous 
Jacobian, 
either globally or on a compact set. 
This condition is used to establish outer-level descent.  
However, in our considered problem \eqref{(P)}, 
the nonlinear constraint mapping $\vh$ is not necessarily differentiable. In contrast to existing works, our convergence analysis does not invoke any differentiability assumptions on $\vh$. 
Instead, the proposed method leverages convexity and blockwise separability while preserving the classical ADMM splitting structure. 

\section{Convergence Analysis}

{In this section, we analyze the convergence of Algorithm~\ref{alg:nl-admm} and establish its convergence rate. We make the following standard assumption.}
\begin{assumption}[Existence of a KKT point]\label{assumption_2}
There exists a point $(\vx^{*},\vy^{*},\vu_{1}^{*},\vu_{2}^{*})$ such that the following KKT conditions hold: 
\beq\label{KKT condition}
\begin{aligned}
    &\mathbf{0}\in \partial f(\vx^{*})+
    \yx{\beta_1\partial \vh(\vx^{*})^{\top}\vu_{1}^{*}}+\beta_{2}A^{\top}\vu_{2}^{*},\\
    &\mathbf{0}\in \partial g(\vy^{*})-\beta_{1}B^{\top}\vu_{1}^{*}+\beta_{2}C^{\top}\vu_{2}^{*},\\
    &\vh(\vx^{*})\leq B\vy^{*},\ A\vx^{*}+C\vy^{*}=\vd,\ \vu_{1}^{*}\geq \mathbf{0}, \\
    &\vu_{1}^{*}\odot(\vh(\vx^{*})-B\vy^{*}) =\mathbf{0}.
\end{aligned}
\eeq
\end{assumption}{In the above assumption, $\beta_1 \vu_1^*$ and $\beta_2 \vu_2^*$ are the Lagrangian multipliers corresponding to the inequality constraint $\vh(\vx) \le B \vy$ and the equality constraint $A\vx + C\vy = \vd$, where $\beta_1$ and $\beta_2$ are the same as those in Algorithm~\ref{alg:nl-admm}. This is for the simplicity of our analysis, as we will see that the scaled dual iterates $\beta_1 \vu_1^k$ and $\beta_2 \vu_2^k$ converge to the optimal dual solution.}

\subsection{Reformulations}
{For ease of notation, we} let $\vz=[\vx; \vs]
$ and define
\begin{equation*}
\begin{aligned}
    &F(\vz):=f(\vx)+\iota_{+}(\vs),\\
    &\vH(\vz):=\vh(\vx)+\vs,\\
    &\vL(\vz):= A \vx - \vd. 
    \end{aligned}    
\end{equation*}
Then 
noting that \eqref{x_update} and \eqref{s_update} are equivalent to the following joint update: 
\begin{equation*}
(\vx^{k+1},\vs^{k+1})=\argmin_{\vs\geq 0,\, \vx}f(\vx)+\frac{\beta_{1}}{2}\big\| \vh(\vx)+\vs-B\vy^{k}+\vu_{1}^{k}\big\|^{2}+\frac{\beta_{2}}{2}\big\| A\vx+C\vy^{k}-\vd+\vu_{2}^{k}\big\|^{2},    
\end{equation*}
we can rewrite the proposed 
{updates in \eqref{x_update}-\eqref{u2_update-0}} to the following more compact manner:
\begin{subequations}\label{eq:NLADMM-steps-new}
\begin{align}
\vz^{k+1}&\in\argmin_{\vz}
   F(\vz)+\tfrac{\beta_{1}}{2}\big\|\vH(\vz)-B \vy^{k}+\vu_{1}^{k}\big\|^{2}
       +\tfrac{\beta_{2}}{2}\big\|\vL(\vz)+C \vy^{k}+\vu_{2}^{k}\big\|^{2}\label{z_update-2},\\
\vy^{k+1}&\in\argmin_{\vy}
   g(\vy)+\tfrac{\beta_{1}}{2}\big\|\vH(\vz^{k+1})-B \vy+\vu_{1}^{k}\big\|^{2}
        +\tfrac{\beta_{2}}{2}\big\|\vL(\vz^{k+1})+C \vy+\vu_{2}^{k}\big\|^{2} \label{y_update-2},\\
\vu_{1}^{k+1}&=\vu_{1}^{k}+\gamma_{1}\!\left(\vH(\vz^{k+1})-B \vy^{k+1}\right)\label{u1_update},\\
\vu_{2}^{k+1}&=\vu_{2}^{k}+\gamma_{2}\!\left(\vL(\vz^{k+1})+C \vy^{k+1}\right).\label{u2_update}
\end{align}
\end{subequations}

\subsection{{Convergence results}}
In order to prove the convergence and {establish the convergence }rate of Algorithm \ref{alg:nl-admm}, 
we first {prove a few lemmas.}

\begin{lemma}\label{core_ineq}
Under Assumption \ref{assumption_2}, let $\{\vy^{k},\vz^{k},\vu_{1}^{k},\vu_{2}^{k}\}_{k\in \mathbb{N}^{+}}$ be the sequence generated by Algorithm \ref{alg:nl-admm}. 
Then 
\yx{for any feasible $(\vy,\vz)$ and any $(\vu_{1},\vu_{2})$, it holds}
\beq\label{(9_0)}
\begin{aligned}
& F(\vz^{k+1}) + g(\vy^{k+1})
  - \big(F(\vz) + g(\vy)\big)
  + \beta_{1}\big\langle \vH(\vz^{k+1}) - B\vy^{k+1}, \vu_{1}\big\rangle 
  + \beta_{2}\big\langle \vL(\vz^{k+1}) + C\vy^{k+1}, \vu_{2}\big\rangle  \\
&\quad
+ \tfrac{\beta_{1}}{2\gamma_{1}}
  \Bigl(
      \lVert \vu_{1}^{k+1} - \vu_{1}^{k} \rVert^{2}
    + \lVert \vu_{1}^{k+1} - \vu_{1} \rVert^{2}
    - \lVert \vu_{1}^{k}   - \vu_{1} \rVert^{2}
  \Bigr) \\
&\quad
+ \tfrac{\beta_{2}}{2\gamma_{2}}
  \Bigl(
      \lVert \vu_{2}^{k+1} - \vu_{2}^{k} \rVert^{2}
    + \lVert \vu_{2}^{k+1} - \vu_{2} \rVert^{2}
    - \lVert \vu_{2}^{k}   - \vu_{2} \rVert^{2}
  \Bigr) \\
&\le
\tfrac{\beta_{1}}{2}
  \bigl(1 - \tfrac{1}{\gamma_{1}}\bigr)^{2}
  \lVert \vu_{1}^{k} - \vu_{1}^{k-1} \rVert^{2}
- \tfrac{\beta_{1}}{2}
  \Bigl(
      \lVert B\vy^{k+1} - B\vy \rVert^{2}
    - \lVert B\vy^{k}   - B\vy \rVert^{2}
  \Bigr) \\
&\quad
+ \tfrac{\beta_{1}}{\gamma_{1}}
  \bigl(1 - \tfrac{1}{\gamma_{1}}\bigr)
  \lVert \vu_{1}^{k+1} - \vu_{1}^{k} \rVert^{2} \\
&\quad
+ \tfrac{\beta_{2}}{2}
  \bigl(1 - \tfrac{1}{\gamma_{2}}\bigr)^{2}
  \lVert \vu_{2}^{k} - \vu_{2}^{k-1} \rVert^{2}
- \tfrac{\beta_{2}}{2}
  \Bigl(
      \lVert C\vy^{k+1} - C\vy \rVert^{2}
    - \lVert C\vy^{k}   - C\vy \rVert^{2}
  \Bigr) \\
&\quad
+ \tfrac{\beta_{2}}{\gamma_{2}}
  \bigl(1 - \tfrac{1}{\gamma_{2}}\bigr)
  \lVert \vu_{2}^{k+1} - \vu_{2}^{k} \rVert^{2}.
\end{aligned}
\eeq

\end{lemma}
\begin{proof}
By \yx{the} optimality condition \yx{for the problem in \eqref{z_update-2}} with respect to $\vz^{k+1}$, \yx{we have}
\begin{equation*}\label{z_opt}
\mathbf{0}
\in 
\partial F(\vz^{k+1})
+ \beta_{1} \partial{\vH}(\vz^{k+1})^{\top}
    \bigl(\vH(\vz^{k+1}) - B\vy^{k} + \vu_{1}^{k}\bigr)
+ \beta_{2} \partial{\vL}(\vz^{k+1})^{\top}
    \bigl(\vL(\vz^{k+1}) + C\vy^{k} + \vu_{2}^{k}\bigr).
\end{equation*}
By the above inclusion, there exist $G_{\vH}(\vz^{k+1})\in \partial \vH(\vz^{k+1})$ and $G_{\vL}(\vz^{k+1})\in \partial \vL(\vz^{k+1})$, \yx{such that
\begin{equation*}
\vv := -\beta_{1} G_{\vH}(\vz^{k+1})^{\top}\big(\vH(\vz^{k+1})-B\vy^{k}+\vu_{1}^{k}\big)-\beta_{2} G_{\vL}(\vz^{k+1})^{\top}\big(\vL(\vz^{k+1})+C\vy^{k}+\vu_{2}^{k}\big) \in \partial F(\vz^{k+1}).    
\end{equation*}
By the convexity of $f$ and $\iota_+(\cdot)$, it holds that}
\begin{equation}\label{F_cvx}
\langle \vz^{k+1}-\vz,\vv\rangle\geq F(\vz^{k+1})-F(\vz), \quad
\forall\, \vz.    
\end{equation}
\yx{In addition, note} that 
\begin{equation*}
\begin{aligned}
\vH(\vz^{k+1}) - B\vy^{k} + \vu_{1}^{k}
&= \vh(\vx^{k+1}) - B\vy^{k} + \vu_{1}^{k}
   + \bigl[-\vh(\vx^{k+1}) + B\vy^{k} - \vu_{1}^{k}\bigr]_{+}= \bigl[\vh(\vx^{k+1}) - B\vy^{k} + \vu_{1}^{k}\bigr]_{+}.
\end{aligned}
\end{equation*}
 \yx{Hence, it follows from the convexity of} $h_{j}$ for all $j$ that  
\beq\label{h_cvx}
\big\langle \vz^{k+1} - \vz,
    G_{\vH}(\vz^{k+1})^{\top}
    \bigl(\vH(\vz^{k+1}) - B\vy^{k} + \vu_{1}^{k}\bigr)
\big\rangle
\ge
\big\langle 
    \vH(\vz^{k+1}) - \vH(\vz),
    \vH(\vz^{k+1}) - B\vy^{k} + \vu_{1}^{k}
\big\rangle.
\eeq
Meanwhile, since $\vL(\vz)$ 
is affine, \yx{we have} 
\beq\label{L_affine}
\big\langle 
    \vz^{k+1} - \vz,
    G_{\vL}(\vz^{k+1})^{\top}
    \bigl(\vL(\vz^{k+1}) + C\vy^{k} + \vu_{2}^{k}\bigr)
\big\rangle
=
\big\langle
    \vL(\vz^{k+1}) - \vL(\vz),
    \vL(\vz^{k+1}) + C\vy^{k} + \vu_{2}^{k}
\big\rangle.
\eeq
Thus, 
\yx{combining \eqref{F_cvx}, \eqref{h_cvx}, and \eqref{L_affine},} we derive 
\beq
\begin{aligned}\label{(1)}
F(\vz^{k+1}) - F(\vz)
&+
\beta_{1}
\big\langle 
    \vH(\vz^{k+1}) - \vH(\vz),
    \vH(\vz^{k+1}) - B\vy^{k} + \vu_{1}^{k}
\big\rangle \\
&+
\beta_{2}
\big\langle 
    \vL(\vz^{k+1}) - \vL(\vz),
    \vL(\vz^{k+1}) + C\vy^{k} + \vu_{2}^{k}
\big\rangle
\le 0.
\end{aligned}
\eeq

By \eqref{u1_update}, we have 
\beq\label{eq:H-term}
\begin{aligned}
&\big\langle \vH(\vz^{k+1}) - \vH(\vz),
    \vH(\vz^{k+1}) - B\vy^{k} + \vu_{1}^{k} \big\rangle \\
=&\big\langle \vH(\vz^{k+1}) - \vH(\vz), \vu_{1}^{k+1} \big\rangle
 + \big\langle \vH(\vz^{k+1}) - \vH(\vz),
    \vH(\vz^{k+1}) - B\vy^{k} + \vu_{1}^{k} - \vu_{1}^{k+1} \big\rangle \\
=&\big\langle \vH(\vz^{k+1}) - \vH(\vz), \vu_{1}^{k+1} \big\rangle
 + \big\langle \vH(\vz^{k+1}) - \vH(\vz),
    (1-\gamma_{1})(\vH(\vz^{k+1}) - B\vy^{k+1})
    + B\vy^{k+1} - B\vy^{k} \big\rangle \\
=&\big\langle \vH(\vz^{k+1}) - \vH(\vz), \vu_{1}^{k+1} \big\rangle
 + \big\langle \vH(\vz^{k+1}) - \vH(\vz), B\vy^{k+1} - B\vy^{k} \big\rangle \\
&\quad
 + (1-\gamma_{1})\big\langle \vH(\vz^{k+1}) - \vH(\vz),
    \vH(\vz^{k+1}) - B\vy^{k+1} \big\rangle .
\end{aligned}
\eeq
Similarly, by \eqref{u2_update}, we have 
\beq\label{eq:L-term}
\begin{aligned}
&\big\langle \vL(\vz^{k+1}) - \vL(\vz),
    \vL(\vz^{k+1}) + C\vy^{k} + \vu_{2}^{k} \big\rangle \\
=&\big\langle \vL(\vz^{k+1}) - \vL(\vz), \vu_{2}^{k+1} \big\rangle
 + \big\langle \vL(\vz^{k+1}) - \vL(\vz),
    -C\vy^{k+1} + C\vy^{k} \big\rangle \\
&\quad
 + (1-\gamma_{2})\big\langle \vL(\vz^{k+1}) - \vL(\vz),
    \vL(\vz^{k+1}) + C\vy^{k+1} \big\rangle .
\end{aligned}
\eeq
Then \yx{plugging \eqref{eq:H-term} and \eqref{eq:L-term} into }
\eqref{(1)} \yx{gives} 
\beq\label{(2)}
\begin{aligned}
&F(\vz^{k+1}) - F(\vz)
 + \beta_{1}\big\langle \vH(\vz^{k+1}) - \vH(\vz), \vu_{1}^{k+1} \big\rangle
 + \beta_{2}\big\langle \vL(\vz^{k+1}) - \vL(\vz), \vu_{2}^{k+1} \big\rangle \\
\leq&-\beta_{1}\big\langle \vH(\vz^{k+1}) - \vH(\vz), B\vy^{k+1} - B\vy^{k} \big\rangle
       -\beta_{1}(1-\gamma_{1})\big\langle \vH(\vz^{k+1}) - \vH(\vz), \vH(\vz^{k+1}) - B\vy^{k+1} \big\rangle \\
&       -\beta_{2}\big\langle \vL(\vz^{k+1}) - \vL(\vz), -  C\vy^{k+1} + C\vy^{k} \big\rangle
       -\beta_{2}(1-\gamma_{2})\big\langle \vL(\vz^{k+1}) - \vL(\vz), \vL(\vz^{k+1}) + C\vy^{k+1} \big\rangle .
\end{aligned}
\eeq


\yx{Moreover, }by \yx{the} optimality condition \yx{for the problem in \eqref{y_update-2}} with respect to $\vy^{k+1}$, \yx{it holds}
\begin{equation*}
\mathbf{0} \in \partial g(\vy^{k+1})
 - \beta_{1} B^{\top}\big( \vH(\vz^{k+1}) - B\vy^{k+1} + \vu_{1}^{k} \big)
 + \beta_{2} C^{\top}\big( \vL(\vz^{k+1}) + C\vy^{k+1} + \vu_{2}^{k} \big).
\end{equation*}
Then by \yx{the} convexity of $g$, and \yx{using} \eqref{u1_update} and \eqref{u2_update}, \yx{we obtain for all $k\ge0$,}
\beq\label{(3)}
g(\vy^{k+1}) - g(\vy)
 - \beta_{1}\big\langle \tfrac{1}{\gamma_{1}}(\vu_{1}^{k+1} - \vu_{1}^{k}) + \vu_{1}^{k}, B\vy^{k+1} - B\vy \big\rangle
 - \beta_{2}\big\langle \tfrac{1}{\gamma_{2}}(\vu_{2}^{k+1} - \vu_{2}^{k}) + \vu_{2}^{k}, -C\vy^{k+1} + C\vy \big\rangle
 \leq 0.
\eeq
\yx{Now adding }
\eqref{(2)} and \eqref{(3)} \yx{yields}  
\beq 
\begin{aligned}\label{(4)}
&F(\vz^{k+1}) + g(\vy^{k+1}) - \big(F(\vz) + g(\vy)\big)
 + \beta_{1}\big\langle \vH(\vz^{k+1}) - \vH(\vz), \vu_{1}^{k+1} \big\rangle
 + \beta_{2}\big\langle \vL(\vz^{k+1}) - \vL(\vz), \vu_{2}^{k+1} \big\rangle \\
&
 - \beta_{1}\big\langle \tfrac{1}{\gamma_{1}}(\vu_{1}^{k+1} - \vu_{1}^{k}) + \vu_{1}^{k}, B\vy^{k+1} - B\vy \big\rangle
 - \beta_{2}\big\langle \tfrac{1}{\gamma_{2}}(\vu_{2}^{k+1} - \vu_{2}^{k}) + \vu_{2}^{k}, -C\vy^{k+1} + C\vy \big\rangle \\
\leq&
-\beta_{1}\big\langle \vH(\vz^{k+1}) - \vH(\vz), B\vy^{k+1} - B\vy^{k} \big\rangle
-\beta_{1}(1-\gamma_{1})\big\langle \vH(\vz^{k+1}) - \vH(\vz), \vH(\vz^{k+1}) - B\vy^{k+1} \big\rangle \\
&
-\beta_{2}\big\langle \vL(\vz^{k+1}) - \vL(\vz), -C\vy^{k+1} + C\vy^{k} \big\rangle
-\beta_{2}(1-\gamma_{2})\big\langle \vL(\vz^{k+1}) - \vL(\vz), \vL(\vz^{k+1}) + C\vy^{k+1} \big\rangle .
\end{aligned}
\eeq

\yx{Furthermore, by the feasibility of $(\vy,\vz)$, i.e.,} $\vH(\vz)-B\vy=\mathbf{0}$, $\vs\geq \mathbf{0}$ and $\vL(\vz)+C\vy=\mathbf{0}$,  then \eqref{(4)} implies
\begin{equation*}
\begin{aligned}
&F(\vz^{k+1}) + g(\vy^{k+1}) - \big(F(\vz) + g(\vy)\big) \\
&+\beta_{1}\big\langle \vH(\vz^{k+1}) - B\vy^{k+1}, \vu_{1}^{k+1} \big\rangle
 +\beta_{1}\Big(1-\tfrac{1}{\gamma_{1}}\Big)\big\langle \vu_{1}^{k+1} - \vu_{1}^{k}, B\vy^{k+1} - B\vy \big\rangle \\
&+\beta_{2}\big\langle \vL(\vz^{k+1}) + C\vy^{k+1}, \vu_{2}^{k+1} \big\rangle
 +\beta_{2}\Big(1-\tfrac{1}{\gamma_{2}}\Big)\big\langle \vu_{2}^{k+1} - \vu_{2}^{k}, -C\vy^{k+1} + C\vy \big\rangle \\
=&F(\vz^{k+1}) + g(\vy^{k+1}) - (F(\vz) + g(\vy)) \\
&+\beta_{1}\big\langle \vH(\vz^{k+1}) - B\vy^{k+1}, \vu_{1}^{k+1} \big\rangle
 +\beta_{1}\big\langle B\vy^{k+1} - B\vy, \vu_{1}^{k+1} \big\rangle \\
&-\beta_{1}\big\langle \tfrac{1}{\gamma_{1}}(\vu_{1}^{k+1} - \vu_{1}^{k}) + \vu_{1}^{k}, B\vy^{k+1} - B\vy \big\rangle \\
&+\beta_{2}\big\langle \vL(\vz^{k+1}) + C\vy^{k+1}, \vu_{2}^{k+1} \big\rangle
 +\beta_{2}\big\langle -C\vy^{k+1} + C\vy, \vu_{2}^{k+1} \big\rangle \\
&-\beta_{2}\big\langle \tfrac{1}{\gamma_{2}}(\vu_{2}^{k+1} - \vu_{2}^{k}) + \vu_{2}^{k},
        -C\vy^{k+1} + C\vy \big\rangle \\
\leq&
-\beta_{1}\big\langle \vH(\vz^{k+1}) - B\vy^{k+1}, B\vy^{k+1} - B\vy^{k} \big\rangle
 -\beta_{1}\big\langle B\vy^{k+1} - B\vy, B\vy^{k+1} - B\vy^{k} \big\rangle \\
&-\beta_{1}(1-\gamma_{1})\big\langle \vH(\vz^{k+1}) - \vH(\vz),
      \vH(\vz^{k+1}) - B\vy^{k+1} \big\rangle \\
&-\beta_{2}\big\langle \vL(\vz^{k+1}) + C\vy^{k+1},
      -C\vy^{k+1} + C\vy^{k} \big\rangle
 -\beta_{2}\big\langle -C\vy^{k+1} + C\vy,
      -C\vy^{k+1} + C\vy^{k} \big\rangle \\
&-\beta_{2}(1-\gamma_{2})\big\langle \vL(\vz^{k+1}) - \vL(\vz),
      \vL(\vz^{k+1}) + C\vy^{k+1} \big\rangle .
\end{aligned}
\end{equation*}
Applying \eqref{u1_update}  and \eqref{u2_update} to the above inequality and rearranging terms, we have for any $\vu_{1}$ and $ \vu_{2}$,
\beq
\begin{aligned}\label{(5)}
&F(\vz^{k+1}) + g(\vy^{k+1}) - \big(F(\vz) + g(\vy)\big)
 + \beta_{1}\big\langle \vH(\vz^{k+1}) - B\vy^{k+1}, \vu_{1} \big\rangle
 + \beta_{2}\big\langle \vL(\vz^{k+1}) + C\vy^{k+1}, \vu_{2} \big\rangle \\
&+ \tfrac{\beta_{1}}{\gamma_{1}}\big\langle \vu_{1}^{k+1} - \vu_{1}^{k}, \vu_{1}^{k+1} - \vu_{1} \big\rangle
 + \tfrac{\beta_{2}}{\gamma_{2}}\big\langle \vu_{2}^{k+1} - \vu_{2}^{k}, \vu_{2}^{k+1} - \vu_{2} \big\rangle \\
\leq&
-\tfrac{\beta_{1}}{\gamma_{1}}\big\langle \vu_{1}^{k+1} - \vu_{1}^{k}, B\vy^{k+1} - B\vy^{k} \big\rangle
 -\beta_{1}\big\langle B\vy^{k+1} - B\vy, B\vy^{k+1} - B\vy^{k} \big\rangle \\
&+\beta_{1}\big(1 - \tfrac{1}{\gamma_{1}}\big)\big\langle \vH(\vz^{k+1}) - \vH(\vz), \vu_{1}^{k+1} - \vu_{1}^{k} \big\rangle
 -\beta_{1}\big(1 - \tfrac{1}{\gamma_{1}}\big)\big\langle \vu_{1}^{k+1} - \vu_{1}^{k}, B\vy^{k+1} - B\vy \big\rangle \\
&-\tfrac{\beta_{2}}{\gamma_{2}}\big\langle \vu_{2}^{k+1} - \vu_{2}^{k}, -C\vy^{k+1} + C\vy^{k} \big\rangle
 -\beta_{2}\big\langle -C\vy^{k+1} + C\vy, -C\vy^{k+1} + C\vy^{k} \big\rangle \\
&+\beta_{2}\big(1 - \tfrac{1}{\gamma_{2}}\big)\big\langle \vL(\vz^{k+1}) - \vL(\vz), \vu_{2}^{k+1} - \vu_{2}^{k} \big\rangle
 -\beta_{2}\big(1 - \tfrac{1}{\gamma_{2}}\big)\big\langle \vu_{2}^{k+1} - \vu_{2}^{k}, -C\vy^{k+1} + C\vy \big\rangle \\
=&-\tfrac{\beta_{1}}{\gamma_{1}}\big\langle \vu_{1}^{k+1} - \vu_{1}^{k}, B\vy^{k+1} - B\vy^{k} \big\rangle
 -\beta_{1}\big\langle B\vy^{k+1} - B\vy, B\vy^{k+1} - B\vy^{k} \big\rangle \\
&+\tfrac{\beta_{1}}{\gamma_{1}}\big(1 - \tfrac{1}{\gamma_{1}}\big)\big\Vert \vu_{1}^{k+1} - \vu_{1}^{k} \big\Vert^{2}  \\
&-\tfrac{\beta_{2}}{\gamma_{2}}\big\langle \vu_{2}^{k+1} - \vu_{2}^{k}, -C\vy^{k+1} + C\vy^{k} \big\rangle
 -\beta_{2}\big\langle -C\vy^{k+1} + C\vy, -C\vy^{k+1} + C\vy^{k} \big\rangle\\
&+\tfrac{\beta_{2}}{\gamma_{2}}\big(1 - \tfrac{1}{\gamma_{2}}\big)\big\Vert \vu_{2}^{k+1} - \vu_{2}^{k} \big\Vert^{2}.
\end{aligned}
\eeq

Meanwhile, we have from \eqref{(3)} that \yx{for all $k\ge1$,}
\begin{equation*}
g(\vy^{k+1}) - g(\vy^{k})
 - \beta_{1}\big\langle \tfrac{1}{\gamma_{1}}(\vu_{1}^{k+1} - \vu_{1}^{k}) + \vu_{1}^{k},
    B\vy^{k+1} - B\vy^{k} \big\rangle
 - \beta_{2}\big\langle \tfrac{1}{\gamma_{2}}(\vu_{2}^{k+1} - \vu_{2}^{k}) + \vu_{2}^{k},
    -C\vy^{k+1} + C\vy^{k} \big\rangle
 \leq 0,
\end{equation*}
and
\begin{equation*}
g(\vy^{k}) - g(\vy^{k+1})
 - \beta_{1}\big\langle \tfrac{1}{\gamma_{1}}(\vu_{1}^{k} - \vu_{1}^{k-1}) + \vu_{1}^{k-1},
    B\vy^{k} - B\vy^{k+1} \big\rangle
 - \beta_{2}\big\langle \tfrac{1}{\gamma_{2}}(\vu_{2}^{k} - \vu_{2}^{k-1}) + \vu_{2}^{k-1},
    -C\vy^{k} + C\vy^{k+1} \big\rangle
 \leq 0.  
\end{equation*}
Adding the above two inequalities and using Young's inequality gives that for all $k\ge1$,
\beq
\begin{aligned}\label{(6)}
&-\tfrac{\beta_{1}}{\gamma_{1}}\big\langle \vu_{1}^{k+1} - \vu_{1}^{k}, B\vy^{k+1} - B\vy^{k}\big\rangle
 -\tfrac{\beta_{2}}{\gamma_{2}}\big\langle  \vu_{2}^{k+1} - \vu_{2}^{k}, -C\vy^{k+1} + C\vy^{k} \big\rangle \\
\leq&
\beta_{1}\big\langle B\vy^{k+1} - B\vy^{k}, (1 - \tfrac{1}{\gamma_{1}})(\vu_{1}^{k} - \vu_{1}^{k-1}) \big\rangle
+\beta_{2}\big\langle -C\vy^{k+1} + C\vy^{k}, (1 - \tfrac{1}{\gamma_{2}})(\vu_{2}^{k} - \vu_{2}^{k-1}) \big\rangle \\
\le & \tfrac{\beta_1}{2}\lVert B\vy^{k+1} - B\vy^{k} \rVert^{2}
     + \tfrac{\beta_1}{2}(1 - \tfrac{1}{\gamma_{1}})^{2}\lVert \vu_{1}^{k} - \vu_{1}^{k-1} \rVert^{2}  + \tfrac{\beta_2}{2}\lVert -C\vy^{k+1} + C\vy^{k} \rVert^{2}
     + \tfrac{\beta_2}{2}(1 - \tfrac{1}{\gamma_{2}})^{2}\lVert \vu_{2}^{k} - \vu_{2}^{k-1} \rVert^{2}     .
\end{aligned}
\eeq
Now plugging \eqref{(6)} 
into \eqref{(5)} and using the equality $\va^\top \vb = \frac{1}{2}\big(\|\va\|^2 + \|\vb\|^2 - \|\va-\vb\|^2\big)$, we obtain the desired result.
\end{proof}

The next lemma can be proved by using the convexity of $f,g$ and $h_j$ for each $j=1,\ldots,m_1$ and the KKT conditions, cf. \cite[Eqn.~(2.2)]{xu2021iteration}.
\begin{lemma}\label{KKT_lemma}
Under Assumption \ref{assumption_2}, let  $(\vx^{*},\vy^{*},\vu_{1}^{*},\vu_{2}^{*})$ be a point such that \eqref{KKT condition} hold, then for any $(\vx,\vy)$, 
it holds
\beq\label{KKT}
f(\vx)+g(\vy)-\big(f(\vx^{*})+g(\vy^{*})\big)+\beta_{1}\big\langle \vh(\vx)-B\vy,\vu_{1}^{*}\big\rangle +\beta_{2}\big\langle A\vx+C\vy-\vd,\vu_{2}^{*}\big\rangle \geq 0.\eeq
\end{lemma}

\yx{Now we are ready to show the convergence result to optimality.}
\begin{theorem}[Convergence to Optimality]\label{thm_convergence}
Under Assumption 
\ref{assumption_2},  let $\{\vy^{k},\vz^{k},\vu_{1}^{k},\vu_{2}^{k}\}_{k\in \mathbb{N}^{+}}$ be the sequence generated by Algorithm \ref{alg:nl-admm} with $\beta_{1},\beta_{2}>0 $ and $\gamma_{1},\gamma_{2}\in (0,\frac{1+\sqrt{5}}{2})$. 
Then 
\begin{align*}
&\lim_{k\to\infty}  \big\|[\vh(\vx^{k})-B\vy^{k}]_{+}\big\| + \big\| A\vx^{k}+C\vy^{k}-\vd\big\|= 0,  \\
&\lim_{k\to\infty}f(\vx^{k})+g(\vy^{k})-\big(f(\vx^{*})+g(\vy^{*})\big)= 0. 
\end{align*}
\end{theorem}
\begin{proof}
For each $k$, define the Lyapunov function as
\begin{equation*}
V^{k}(\vu_{1},\vu_{2},\vy):=\frac{\beta_{1}}{2\gamma_{1}}\lVert \vu_{1}^{k}-\vu_{1}\rVert^{2}+\frac{\beta_{2}}{2\gamma_{2}}\lVert \vu_{2}^{k}-\vu_{2}\rVert^{2}+\frac{\beta_{1}}{2}\lVert B\vy^{k}-B\vy\rVert^{2}+\frac{\beta_{2}}{2}\lVert -C\vy^{k}+C\vy\rVert^{2}.    
\end{equation*}
For $\vs,\vs^{k+1}\geq \mathbf{0}$, it holds $\iota_{+}(\vs^{k+1})=\iota_{+}(\vs)=0$, and when $\vu_{1}\geq \mathbf{0}$, it holds $\langle \vs^{k+1}, \vu_1\rangle \ge 0$. Hence, \eqref{(9_0)} implies that for any $\vu_{1}\geq \mathbf{0}$,
\beq\label{(10)}
\begin{aligned}
&f(\vx^{k+1}) + g(\vy^{k+1})
 - \bigl( f(\vx) + g(\vy) \bigr)
 + \beta_{1}\big\langle \vh(\vx^{k+1}) - B\vy^{k+1}, \vu_{1} \big\rangle
 + \beta_{2}\big\langle A\vx^{k+1} + C\vy^{k+1} - \vd, \vu_{2} \big\rangle \\
&+ \tfrac{\beta_{1}(2-\gamma_{1})}{2\gamma_{1}^{2}}\lVert \vu_{1}^{k+1} - \vu_{1}^{k} \rVert^{2}
 + \tfrac{\beta_{2}(2-\gamma_{2})}{2\gamma_{2}^{2}}\lVert \vu_{2}^{k+1} - \vu_{2}^{k} \rVert^{2}
 + V^{k+1}\yx{(\vu_{1},\vu_{2},\vy)} \\
\leq&
V^{k}\yx{(\vu_{1},\vu_{2},\vy)}
 + \tfrac{\beta_{1}}{2}\bigl(1 - \tfrac{1}{\gamma_{1}}\bigr)^{2}
   \lVert \vu_{1}^{k} - \vu_{1}^{k-1} \rVert^{2}
 + \tfrac{\beta_{2}}{2}\bigl(1 - \tfrac{1}{\gamma_{2}}\bigr)^{2}
   \lVert \vu_{2}^{k} - \vu_{2}^{k-1} \rVert^{2}.
\end{aligned}
\eeq
Letting $(\vx,\vy,\vu_1,\vu_2)=(\vx^{*},\vy^{*},\vu_{1}^{*},\vu_{2}^{*})$ in \eqref{(10)}, summing it from $k=1,\cdots, K$ for some $K\in\mathbb{N}^{+}$ and rearranging terms, we obtain
\beq\label{(10_0)}
\begin{aligned}
&\sum_{k=1}^{K} \Bigl(
    f(\vx^{k+1}) + g(\vy^{k+1}) - \bigl(f(\vx^{*}) + g(\vy^{*})\bigr)
    + \beta_{1}\big\langle \vh(\vx^{k+1}) - B\vy^{k+1}, \vu_{1}^{*} \big\rangle
    + \beta_{2}\big\langle A\vx^{k+1} + C\vy^{k+1} - \vd, \vu_{2}^{*} \big\rangle
\Bigr) \\
&+\beta_{1}\Bigl(\tfrac{2-\gamma_{1}}{2\gamma_{1}^{2}} - \tfrac{1}{2}\bigl(1-\tfrac{1}{\gamma_{1}}\bigr)^{2}\Bigr)
  \sum_{k=1}^{K-1}\lVert \vu_{1}^{k+1} - \vu_{1}^{k} \rVert^{2}
 +\beta_{2}\Bigl(\tfrac{2-\gamma_{2}}{2\gamma_{2}^{2}} - \tfrac{1}{2}\bigl(1-\tfrac{1}{\gamma_{2}}\bigr)^{2}\Bigr)
  \sum_{k=1}^{K-1}\lVert \vu_{2}^{k+1} - \vu_{2}^{k} \rVert^{2} \\
&+ \tfrac{\beta_{1}(2-\gamma_{1})}{2\gamma_{1}^{2}}\lVert \vu_{1}^{K+1} - \vu_{1}^{K} \rVert^{2}
 + \tfrac{\beta_{2}(2-\gamma_{2})}{2\gamma_{2}^{2}}\lVert \vu_{2}^{K+1} - \vu_{2}^{K} \rVert^{2}
 + V^{K+1}\yx{(\vu_{1}^{*},\vu_{2}^{*},\vy^{*})} \\
\leq\;&
V^{1}(\vu_{1}^{*},\vu_{2}^{*},\vy^{*})
 + \tfrac{\beta_{1}}{2}\bigl(1-\tfrac{1}{\gamma_{1}}\bigr)^{2}\lVert \vu_{1}^{1} - \vu_{1}^{0} \rVert^{2}
 + \tfrac{\beta_{2}}{2}\bigl(1-\tfrac{1}{\gamma_{2}}\bigr)^{2}\lVert \vu_{2}^{1} - \vu_{2}^{0} \rVert^{2}.
\end{aligned}
\eeq

Since $\gamma_{1},\gamma_{2}\in(0,\frac{1+\sqrt{5}}{2})$, it is straightforward to verify
\begin{equation*}
\frac{2-\gamma_{i}}{2\gamma_{i}^{2}}
 - \frac{1}{2}\Bigl(1 - \frac{1}{\gamma_{i}}\Bigr)^{2} > 0,
\qquad
\frac{2-\gamma_{i}}{2\gamma_{i}^{2}} > 0,
\quad
i = 1, 2.
\end{equation*}
Let
\begin{equation}\label{eq:def-c12}
c_{1} := \max_{i=1,2} \,\frac{\beta_{i}}{2}\Bigl(1 - \tfrac{1}{\gamma_{i}}\Bigr)^{2},
\qquad
c_{2} := \min_{i=1,2} \,\beta_{i}\Bigl(\frac{2-\gamma_{i}}{2\gamma_{i}^{2}}
             - \tfrac{1}{2}\Bigl(1 - \tfrac{1}{\gamma_{i}}\Bigr)^{2}\Bigr).
\end{equation}
Then from \eqref{KKT} and \eqref{(10_0)}, we derive
\begin{equation}\label{eq:def-c3}
\sum_{k=1}^{K-1}\bigl( \lVert \vu_{1}^{k+1} - \vu_{1}^{k} \rVert^{2}
                      + \lVert \vu_{2}^{k+1} - \vu_{2}^{k} \rVert^{2} \bigr)
\leq
\frac{
  V^{1}(\vu_{1}^{*},\vu_{2}^{*},\vy^{*})
  + c_{1}\bigl( \lVert \vu_{1}^{1} - \vu_{1}^{0} \rVert^{2}
               + \lVert \vu_{2}^{1} - \vu_{2}^{0} \rVert^{2} \bigr)
}{c_{2}} =: c_3.
\end{equation}
Letting $K\rightarrow \infty$, we get
\beq\label{u_converge}
\sum_{k=1}^{\infty}\lVert \vu_{1}^{k+1}-\vu_{1}^{k}\rVert^{2}+\lVert \vu_{2}^{k+1}-\vu_{2}^{k}\rVert^{2}\leq c_{3} < \infty,
\eeq
which together with \eqref{u1_update-0}, \eqref{u2_update-0}, and $\vs^{k+1}\geq \mathbf{0}$ implies that as $k\to \infty$, 
\beq\label{feasibility_convergence}
\big\|[\vh(\vx^{k+1})-B\vy^{k+1}]_{+}\big\|=\big\lVert \vh(\vx^{k+1}) + \vs^{k+1} - B\vy^{k+1} \big\rVert\rightarrow 0,\quad \big\|A\vx^{k+1}+C\vy^{k+1}-\vd\big\|\rightarrow 0. 
\eeq
This proves the convergence of constraint violation.

To show the convergence of the objective value, we derive from \eqref{(9_0)} that 
\begin{equation*}
\begin{aligned}
&f(\vx^{k+1}) + \iota_{+}(\vs^{k+1}) + g(\vy^{k+1})
 - \bigl( f(\vx) + \iota_{+}(\vs) + g(\vy) \bigr) \\
&+ \beta_{1}\big\langle \vh(\vx^{k+1}) + \vs^{k+1} - B\vy^{k+1}, \vu_{1} \big\rangle
 + \beta_{2}\big\langle A\vx^{k+1} + C\vy^{k+1} - \vd, \vu_{2} \big\rangle \\
&+ \tfrac{\beta_{1}(2-\gamma_{1})}{2\gamma_{1}^{2}}\lVert \vu_{1}^{k+1} - \vu_{1}^{k} \rVert^{2}
 + \tfrac{\beta_{2}(2-\gamma_{2})}{2\gamma_{2}^{2}}\lVert \vu_{2}^{k+1} - \vu_{2}^{k} \rVert^{2}
 + V^{k+1}\yx{(\vu_{1},\vu_{2},\vy)} \\
\leq
&V^{k}\yx{(\vu_{1},\vu_{2},\vy)}
 + \tfrac{\beta_{1}}{2}\bigl(1 - \tfrac{1}{\gamma_{1}}\bigr)^{2}\lVert \vu_{1}^{k} - \vu_{1}^{k-1} \rVert^{2}
 + \tfrac{\beta_{2}}{2}\bigl(1 - \tfrac{1}{\gamma_{2}}\bigr)^{2}\lVert \vu_{2}^{k} - \vu_{2}^{k-1} \rVert^{2}.
\end{aligned}  
\end{equation*}
Letting $(\vx, \vy, \vu_1, \vu_2)=(\vx^{*},\vy^{*},\vu_{1}^{*}, \vu_{2}^{*})$ and $\vs=\vs^{*}=B\vy^{*}-\vh(\vx^{*})\geq \mathbf{0}$ in the above inequality, we have
\begin{equation*}
\begin{aligned}
&f(\vx^{k+1}) 
+ g(\vy^{k+1})
 - \bigl( f(\vx^{*}) 
 + g(\vy^{*}) \bigr) \\
&+ \beta_{1}\big\langle \vh(\vx^{k+1}) + \vs^{k+1} - B\vy^{k+1}, \vu_{1}^{*} \big\rangle
 + \beta_{2}\big\langle A\vx^{k+1} + C\vy^{k+1} - \vd, \vu_{2}^{*} \big\rangle \\
\leq\,&
V^{k}(\vu_{1}^{*}, \vu_{2}^{*}, \vy^{*})
 - V^{k+1}(\vu_{1}^{*}, \vu_{2}^{*}, \vy^{*})
 + \tfrac{\beta_{1}}{2}\bigl(1 - \tfrac{1}{\gamma_{1}}\bigr)^{2}
   \lVert \vu_{1}^{k} - \vu_{1}^{k-1} \rVert^{2}
 + \tfrac{\beta_{2}}{2}\bigl(1 - \tfrac{1}{\gamma_{2}}\bigr)^{2}
   \lVert \vu_{2}^{k} - \vu_{2}^{k-1} \rVert^{2},
\end{aligned}    
\end{equation*}
which together with Cauchy-Schwarz inequality implies
\beq\label{(10_1)}
\begin{aligned}
&f(\vx^{k+1}) + g(\vy^{k+1}) - \bigl( f(\vx^{*}) + g(\vy^{*}) \bigr) \\
\leq&
\beta_{1}\lVert \vu_{1}^{*} \rVert \big\lVert \vh(\vx^{k+1}) + \vs^{k+1} - B\vy^{k+1} \big\rVert
+\beta_{2}\lVert \vu_{2}^{*} \rVert \big\lVert A\vx^{k+1} + C\vy^{k+1} - \vd \big\rVert \\
&+ V^{k}(\vu_{1}^{*}, \vu_{2}^{*}, \vy^{*})
  - V^{k+1}(\vu_{1}^{*}, \vu_{2}^{*}, \vy^{*})
  + \tfrac{\beta_{1}}{2}\bigl( 1 - \tfrac{1}{\gamma_{1}} \bigr)^{2}
    \lVert \vu_{1}^{k} - \vu_{1}^{k-1} \rVert^{2} + \tfrac{\beta_{2}}{2}\bigl( 1 - \tfrac{1}{\gamma_{2}} \bigr)^{2}
    \lVert \vu_{2}^{k} - \vu_{2}^{k-1} \rVert^{2}.
\end{aligned}
\eeq

From \cite[Theorem 1]{robbins1971convergence}, we have 
the convergence of the Lyapunov sequence $\{V^{k}(\vu_{1}^{*},\vu_{2}^{*},\vy^{*})\}_{k\in \mathbb{N}^{+}}$. 
Thus, \eqref{u_converge}, \eqref{feasibility_convergence}, and \eqref{(10_1)} indicate
\begin{equation*}
\limsup_{k\to\infty}f(\vx^{k+1})+g(\vy^{k+1})-\big(f(\vx^{*})+g(\vy^{*})\big)\leq 0.  
\end{equation*}
On the other hand, by 
\eqref{KKT}, it follows
\begin{equation*}
\begin{aligned}
f(\vx^{k+1})+g(\vy^{k+1})-\big(f(\vx^{*})+g(\vy^{*})\big)&\geq -\beta_{1}\big\langle \vh(\vx^{k+1})-B\vy^{k+1},\vu_{1}^{*}\big\rangle -\beta_{2}\big\langle A\vx^{k+1}+C\vy^{k+1}-\vd,\vu_{2}^{*}\big\rangle \\
&\geq -\beta_{1}\lVert \vu_{1}^{*}\rVert \big\| [\vh(\vx^{k+1})
-B\vy^{k+1}]_+\big\|-\beta_{2}\lVert \vu_{2}^{*}\rVert \big\| A\vx^{k+1}+C\vy^{k+1}-\vd\big\|,
\end{aligned}  
\end{equation*}
which together with \eqref{feasibility_convergence} indicates
$$\liminf_{k\to\infty}f(\vx^{k+1})+g(\vy^{k+1})-\big(f(\vx^{*})+g(\vy^{*})\big)\ge 0.$$
Hence, we have the convergence of the objective value and complete the proof.
\end{proof}

\begin{theorem}[Convergence Rate]
Under Assumption 
\ref{assumption_2},  let $\{\vy^{k},\vz^{k},\vu_{1}^{k},\vu_{2}^{k}\}_{k\in \mathbb{N}^{+}}$ be the sequence generated by Algorithm \ref{alg:nl-admm} with $\beta_{1},\beta_{2}>0 $ and $\gamma_{1},\gamma_{2}\in (0,\frac{1+\sqrt{5}}{2})$. 
Define
\begin{align*}
&\bar{\vx}^{K}:=\frac{1}{K}\sum_{k=1}^{K}\vx^{k+1}, \quad \bar{\vy}^{K}:=\frac{1}{K}\sum_{k=1}^{K}\vy^{k+1},\quad 
\bar{\vr}_{1}^{K}:=h(\bar{\vx}^{K})-B\bar{\vy}^{K}, \quad \bar{\vr}_{2}^{K}:=A\bar{\vx}^{K}+C\bar{\vy}^{K}-\vd;\\    
& \alpha_{1}(\vu_{1}) :=\frac{\beta_{1}}{2\gamma_{1}}\lVert\vu_{1}^{1}-\vu_{1}\rVert^{2}, \quad \alpha_{2}(\vu_{2}):=\frac{\beta_{2}}{2\gamma_{2}}\lVert\vu_{2}^{1}-\vu_{2}\rVert^{2};\\   
& c_{4}:= \frac{\beta_{1}}{2}\lVert B\vy^{1}-B\vy^{*}\rVert^{2}+\frac{\beta_{2}}{2}\lVert -C\vy^{1}+C\vy^{*}\rVert^{2}
+c_1 c_3,  
\end{align*}
\yx{where $c_1$ and $c_3$ are defined in \eqref{eq:def-c12} and \eqref{eq:def-c3}.}
Then for any $\rho_{1}>\lVert \vu_{1}^{*}\rVert, \rho_{2}>\lVert \vu_{2}^{*}\rVert$, it holds 
for any $K\geq 1$, 
\begin{equation*}
\lVert[\bar{\vr}_{1}^{K}]_{+}\rVert+\lVert\bar{\vr}_{2}^{K}\rVert\leq \frac{M(\rho_{1},\rho_{2})}{\min_{i=1,2}\beta_{i}\big(\rho_{i}-\lVert \vu_{i}^{*}\rVert\big)K}, 
\end{equation*}
and
\begin{equation*}
-\frac{\yx{\max_{i=1,2}\beta_{i}\lVert \vu_{i}^{*}\rVert}}{\min_{i=1,2}\beta_{i}\big(\rho_{i}-\lVert \vu_{i}^{*}\rVert\big)}\frac{M(\rho_{1},\rho_{2})}{K}\leq f(\bar{\vx}^{K})+g(\bar{\vy}^{K})-\big(f(\vx^{*})+g(\vy^{*})\big)\leq \frac{M(\rho_{1},\rho_{2})}{K},
\end{equation*}
where 
\begin{equation*}
M(\rho_{1},\rho_{2}):=\sup_{\lVert \tilde{\vu}_{1}\rVert\leq \rho_{1}}\alpha_{1}(\tilde{\vu}_{1})+\sup_{\lVert \tilde{\vu}_{2}\rVert\leq \rho_{2}}\alpha_{2}(\tilde{\vu}_{2})+c_{4} < \infty.
\end{equation*}
\end{theorem}
\begin{proof}
\yx{Letting $(\vx,\vy)=(\vx^{*},\vy^{*})$} in \eqref{(10)}, we have
\begin{equation*}
\begin{aligned}
&f(\vx^{k+1}) + g(\vy^{k+1})
 - \bigl( f(\vx^{*}) + g(\vy^{*}) \bigr)
 + \beta_{1}\big\langle \vh(\vx^{k+1}) - B\vy^{k+1}, \vu_{1} \big\rangle
 + \beta_{2}\big\langle A\vx^{k+1} + C\vy^{k+1} - \vd, \vu_{2} \big\rangle \\
\leq&
V^{k}( \vu_{1}, \vu_{2}, \vy^{*})
 - V^{k+1}(\vu_{1}, \vu_{2}, \vy^{*})
 + \tfrac{\beta_{1}}{2}\bigl(1 - \tfrac{1}{\gamma_{1}}\bigr)^{2}\lVert \vu_{1}^{k} - \vu_{1}^{k-1} \rVert^{2}
  + \tfrac{\beta_{2}}{2}\bigl(1 - \tfrac{1}{\gamma_{2}}\bigr)^{2}\lVert \vu_{2}^{k} - \vu_{2}^{k-1} \rVert^{2}. 
\end{aligned}
\end{equation*}
Summing up the above inequality from $k=1,\cdots K$, for some $K\in \mathbb{N}^{+}$, we derive
\begin{equation*}
\begin{aligned}
&\sum_{k=1}^{K}\Bigl(
    f(\vx^{k+1}) + g(\vy^{k+1})
    - \bigl( f(\vx^{*}) + g(\vy^{*}) \bigr)
    + \beta_{1}\big\langle \vh(\vx^{k+1}) - B\vy^{k+1}, \vu_{1} \big\rangle
    + \beta_{2}\big\langle A\vx^{k+1} + C\vy^{k+1} - \vd, \vu_{2} \big\rangle
\Bigr) \\
\leq&
\tfrac{\beta_{1}}{2\gamma_{1}}\lVert \vu_{1}^{1} - \vu_{1} \rVert^{2}
+\tfrac{\beta_{2}}{2\gamma_{2}}\lVert \vu_{2}^{1} - \vu_{2} \rVert^{2}
+\tfrac{\beta_{1}}{2}\lVert B\vy^{1} - B\vy^{*} \rVert^{2}
+\tfrac{\beta_{2}}{2}\lVert -C\vy^{1} + C\vy^{*} \rVert^{2} \\
&+ \sum_{k=1}^{K}\Bigl(
    \tfrac{\beta_{1}}{2}\bigl(1 - \tfrac{1}{\gamma_{1}}\bigr)^{2}
      \lVert \vu_{1}^{k} - \vu_{1}^{k-1} \rVert^{2}
  + \tfrac{\beta_{2}}{2}\bigl(1 - \tfrac{1}{\gamma_{2}}\bigr)^{2}
      \lVert \vu_{2}^{k} - \vu_{2}^{k-1} \rVert^{2}
\Bigr).
\end{aligned}
\end{equation*}
Dividing both sides of the above inequality by $K$ and using the definition of $c_4$, we have
\beq\label{(12)}
\begin{aligned}
\frac{1}{K}\sum_{k=1}^{K}\Bigl(
    f(\vx^{k+1}) + g(\vy^{k+1})
    - \bigl( f(\vx^{*}) + g(\vy^{*}) \bigr)
    + \beta_{1}\big\langle \vh(\vx^{k+1}) - B\vy^{k+1}, \vu_{1} \big\rangle \\
    + \beta_{2}\big\langle A\vx^{k+1} + C\vy^{k+1} - \vd, \vu_{2} \big\rangle
\Bigr) 
\leq \frac{\alpha_{1}(\vu_{1}) + \alpha_{2}(\vu_{2}) + c_{4}}{K}.
\end{aligned}
\eeq

Moreover, by the convexity of $f$ and $g$, it holds
\begin{equation*}
f(\bar{\vx}^{K})+g(\bar{\vy}^{K})\leq \frac{1}{K}\sum_{k=1}^{K}f(\vx^{k+1})+g(\vy^{k+1}).
\end{equation*}
Also, from the convexity of each component of $\vh$ and $\vu_{1}\geq \mathbf{0}$, it follows 
\begin{equation*}
\big\langle \bar{\vr}_{1}^{K},\vu_{1}\big\rangle =\big\langle \vh(\bar{\vx}^{K})-B\bar{\vy}^{K},\vu_{1}\big\rangle \leq \frac{1}{K}\sum_{k=1}^{K}\big\langle \vh(\vx^{k+1})-B\vy^{k+1},\vu_{1}\big\rangle .
\end{equation*}
Hence, \eqref{(12)} \yx{together with the above two inequalities and $\sum_{k=1}^K \left(A\vx^{k+1} + C\vy^{k+1} - \vd\right)= K\bar{\vr}_2^K$} gives
\begin{equation*}
f(\bar{\vx}^{K})+g(\bar{\vy}^{K})-(f(\vx^{*})+g(\vy^{*}))+\beta_{1}\big\langle \bar{\vr}_{1}^{K},\vu_{1}\big\rangle +\beta_{2}\big\langle \bar{\vr}_{2}^{K},\vu_{2}\big\rangle \leq \frac{\alpha_{1}(\vu_{1})+\alpha_{2}(\vu_{2})+c_{4}}{K},\forall\, \vu_1\ge\vzero, \vu_2.
\end{equation*}

Now let $\vu_{1}=\rho_{1}\frac{[\bar{\vr}_{1}^{K}]_{+}}{\lVert [\bar{\vr}_{1}^{K}]_{+}\rVert}$ and $\vu_{2}=\rho_{2}\frac{\bar{\vr}_{2}^{K}}{\lVert \bar{\vr}_{2}^{K}\rVert}$  
in the above inequality to have
\beq\label{(13)}
\begin{aligned}
&f(\bar{\vx}^{K})+g(\bar{\vy}^{K})-(f(\vx^{*})+g(\vy^{*}))+\beta_{1}\rho_{1}\lVert [\bar{\vr}_{1}^{K}]_{+}\rVert+\beta_{2}\rho_{2}\lVert \bar{\vr}_{2}^{K}\rVert\\
\leq &\frac{\alpha_{1}\left(\rho_{1}\frac{[\bar{\vr}_{1}^{K}]_{+}}{\lVert [\bar{\vr}_{1}^{K}]_{+}\rVert}\right)+\alpha_{2}\left(\rho_{2}\frac{\bar{\vr}_{2}^{K}}{\lVert \bar{\vr}_{2}^{K}\rVert}\right)+c_{4}}{K}
\le \frac{M(\rho_{1},\rho_{2})}{K}.
\end{aligned}
\eeq
Meanwhile, 
\yx{let $(\vx,\vy)=(\bar{\vx}^{K},\bar{\vy}^{K})$ in} \eqref{KKT} and use $\vu_{1}^{*}\geq 0$. We have
\begin{equation*}
\begin{aligned}
&f(\bar{\vx}^{K})+g(\bar{\vy}^{K})-(f(\vx^{*})+g(\vy^{*}))+\beta_{1}\big\langle [\bar{\vr}_{1}^{K}]_{+},\vu_{1}^{*}\big\rangle +\beta_{2}\big\langle \bar{\vr}_{2}^{K},\vu_{2}^{*}\big\rangle \\
\geq &f(\bar{\vx}^{K})+g(\bar{\vy}^{K})-(f(\vx^{*})+g(\vy^{*}))+\beta_{1}\big\langle \bar{\vr}_{1}^{K},\vu_{1}^{*}\big\rangle +\beta_{2}\big\langle \bar{\vr}_{2}^{K},\vu_{2}^{*}\big\rangle 
\geq 0.
\end{aligned}
\end{equation*}
Thus
\begin{equation*}
f(\bar{\vx}^{K})+g(\bar{\vy}^{K})-(f(\vx^{*})+g(\vy^{*}))\geq -\beta_{1}\lVert \vu_{1}^{*}\rVert\lVert[\bar{\vr}_{1}^{K}]_{+}\rVert-\beta_{2}\lVert \vu_{2}^{*}\rVert\lVert\bar{\vr}_{2}^{K}\rVert.
\end{equation*}
Together with \eqref{(13)}, we derive
\begin{equation*}
\lVert[\bar{\vr}_{1}^{K}]_{+}\rVert+\lVert\bar{\vr}_{2}^{K}\rVert\leq \frac{M(\rho_{1},\rho_{2})}{\min_{i=1,2}\beta_{i}(\rho_{i}-\lVert \vu_{i}^{*}\rVert)K},
\end{equation*}
and 
\begin{equation*}
\begin{aligned}
-\frac{\yx{\max_{i=1,2}\beta_{i}\lVert \vu_{i}^{*}\rVert}}{\min_{i=1,2}\beta_{i}(\rho_{i}-\lVert \vu_{i}^{*}\rVert)}\frac{M(\rho_{1},\rho_{2})}{K} \leq f(\bar{\vx}^{K})+g(\bar{\vy}^{K})-(f(\vx^{*})+g(\vy^{*})). 
\end{aligned}
\end{equation*}
This completes the proof.
\end{proof}


\section{Numerical Results} 
\label{sec:numerical}
In this section, we test the proposed NL-ADMM on the two applications 
described in Section \ref{sec:applications}: 
the distributed resource-allocation problem \eqref{eq:resource_allocation} and the constrained machine learning model \eqref{eq:consensus}. Through comparisons with standard baselines, we illustrate the effectiveness and communication efficiency of the proposed method across both tasks.

\subsection{Distributed resource allocation}
To generate concrete test instances of \eqref{eq:resource_allocation}, \yx{for each $j=1,\ldots,p$, we set} 
\begin{equation*}
f_{j}(\vx_{j})=\frac{1}{2}\vx_{j}^{\top}Q_{j}^{(f)}\vx_{j}+b_{j}^{(f)\top}\vx_{j}+\iota_{[l,u]^{d}}(\vx_{j}), \quad h_{j}(\vx_{j})=\frac{1}{2}\vx_{j}^{\top}Q_{j}^{(h)}\vx_{j}+b_{j}^{(h)\top}\vx_{j}+c_{j}.
\end{equation*}
in \eqref{Resource_allocation}.
\yx{Here,} $Q_{j}^{(f)}$ \yx{and} $Q_{j}^{(h)}$ are positive definite matrices in $\RR^{d\times d}$ generated as 
\begin{equation*}
Q_{j}^{(f)}=R_{j}^{(f)\top}R_{j}^{(f)}+\mu_{f}I, \quad Q_{j}^{(h)}=R_{j}^{(h)\top}R_{j}^{(h)}+\mu_{h}I,
\end{equation*}
with $R_{j}^{(f)}$ and $R_{j}^{(h)}$ drawn entry-wise from a standard normal distribution and normalized by their spectral norms, i.e., $R_{j}^{(\cdot)} \leftarrow R_{j}^{(\cdot)} / \lVert R_{j}^{(\cdot)} \rVert_{2}$. The vectors $b_{j}^{(f)}$ and $b_{j}^{(h)}$ are also generated entry-wise from a standard normal distribution. The constant $c_{j}$ is chosen as a negative random scalar to ensure Slater’s condition at the origin, and 
$\mu_{f}=10^{-2}$ and $\mu_{h}=10^{-4}$. 
In addition, we set 
$l=-5$ and $u=5$ for the box constraint. 


We 
evaluate the numerical performance of the proposed nonlinear ADMM 
\yx{by comparing} it with two 
baselines: the augmented Lagrangian method (ALM) and the Douglas–Rachford Splitting (DRS) method. 
ALM is applied directly to 
problem~\eqref{Resource_allocation}. 
Given a penalty parameter $\beta_{\mathrm{ALM}}>0$ and an initial multiplier
$u^{0}\ge 0$, the ALM iterates are defined as
\begin{subequations}\label{eq:ALM_scaled}
\begin{align}
\vx^{k+1}
&=\argmin_{\vx}\;
\sum_{j=1}^{p}f_{j}(\vx_{j})
+\frac{\beta_{\mathrm{ALM}}}{2}\left\|\left[\sum_{j=1}^{p}h_{j}(\vx_{j})+u^{k}\right]_{+}\right\|^{2}, \label{alm_primal}
\\
u^{k+1}
&=\left[u^{k}+\sum_{j=1}^{p}h_{j}(\vx^{k+1}_{j})\right]_{+}.
\label{alm_scaled_u}
\end{align}
\end{subequations}
The DRS method is applied to \eqref{eq:resource_allocation}, initialized with $\vw^{0}$ and uses parameters 
$\beta_{\mathrm{DRS}}>0$ and $\eta_{\mathrm{DRS}}\in(0,1]$. 
The resulting updates are given by
 \begin{subequations}\label{DRS_algorithm_1}
\begin{align}
\vu^{k}
&= \tfrac{1}{p}\mathbf{1}\mathbf{1}^{\top}\vw^{k},
\label{eq:drs-u} \\
\vq^{k}
&= 2\vu^{k}-\vw^{k},
\label{eq:drs-q} \\
\vx_{j}^{k}
&\in \argmin_{\vx_{j}}\;
f_{j}(\vx_{j})
+\tfrac{\beta_{\mathrm{DRS}}}{2}
\bigl(\,[h_{j}(\vx_{j})+\vq_{j}^{k}]_{+}\bigr)^{2},
\quad j=1,\dots,p,
\label{eq:drs-x} \\
\vt^{k}
&= [\vq^{k}+\vh(\vx^{k})]_{+},
\label{eq:drs-t} \\
\vw^{k+1}
&= \vw^{k}
+2\eta_{\mathrm{DRS}}\bigl(\vt^{k}-\vu^{k}\bigr),
\label{eq:drs-w}
\end{align}
\end{subequations} with derivation provided in the appendix.


 
The algorithmic parameters are fixed uniformly across the settings of $p\in\{2,5,10\}$, with the dimension of each block fixed to $d=500$. For the ALM updates, we 
set its penalty parameter to $\beta_{\mathrm{ALM}}=5\times 10^{-4}$. The NL-ADMM and DRS employ the penalty parameters $\beta_{1}=10^{-3}$ and $\beta_{\mathrm{DRS}}=10^{-3}$, respectively. Note that NL-ADMM does not involve a second penalty parameter $\beta_{2}$, since there is no equality constraint for problem \eqref{eq:resource_allocation}. For DRS, we fix the relaxation parameter $\eta_{\mathrm{DRS}}=0.5$. These parameters are tuned through a grid search over candidate values covering multiple orders of magnitude, with the goal of optimizing communication efficiency. 



All inner minimization subproblems arising from the ALM formulation, as well as from the proximal operators in NL-ADMM and DRS, are solved by using an accelerated proximal gradient (APG) method with \yx{line search}. 
The stopping tolerance for APG is set to $\texttt{subtol}=10^{-5}$, and the outer iterations of all three methods terminate once the primal, dual, and complementarity residuals fall below $\texttt{tol}=10^{-4}$. To obtain higher accuracy reference solutions for each generated instance, we additionally 
\yx{run} the ALM 
with 
smaller tolerances $(\texttt{subtol}, \texttt{tol}) = (10^{-6}, 10^{-5})$.



Convergence is reported with respect to the cumulative number of communication rounds, rather than raw iteration counts, in order to reflect the communication cost incurred during distributed execution. For NL-ADMM and DRS, each outer iteration consists of one broadcast from the server to all workers and one aggregation of their local responses, resulting in two communication steps per iteration and hence one communication round. All local 
subproblems are solved independently by APG at each worker and therefore incur no additional communication cost. 
For 
ALM,  
each \yx{global} function 
evaluation within the APG subsolver 
\yx{requires} one communication round \yx{between the server and workers}. 

Figure \ref{Fig:resource_allocation} reports the objective value gaps and KKT violations of NL-ADMM and DRS 
with respect to the cumulative number of communication rounds for a representative run. For 
each $p\in\{2,5,10\}$, NL-ADMM reduces both metrics \yx{faster, requiring substantially fewer communication rounds than DRS to achieve comparable accuracy}. 
\begin{figure}
  \centering
  \begin{subfigure}{0.23\textwidth}
    \centering
    \includegraphics[width=\linewidth]{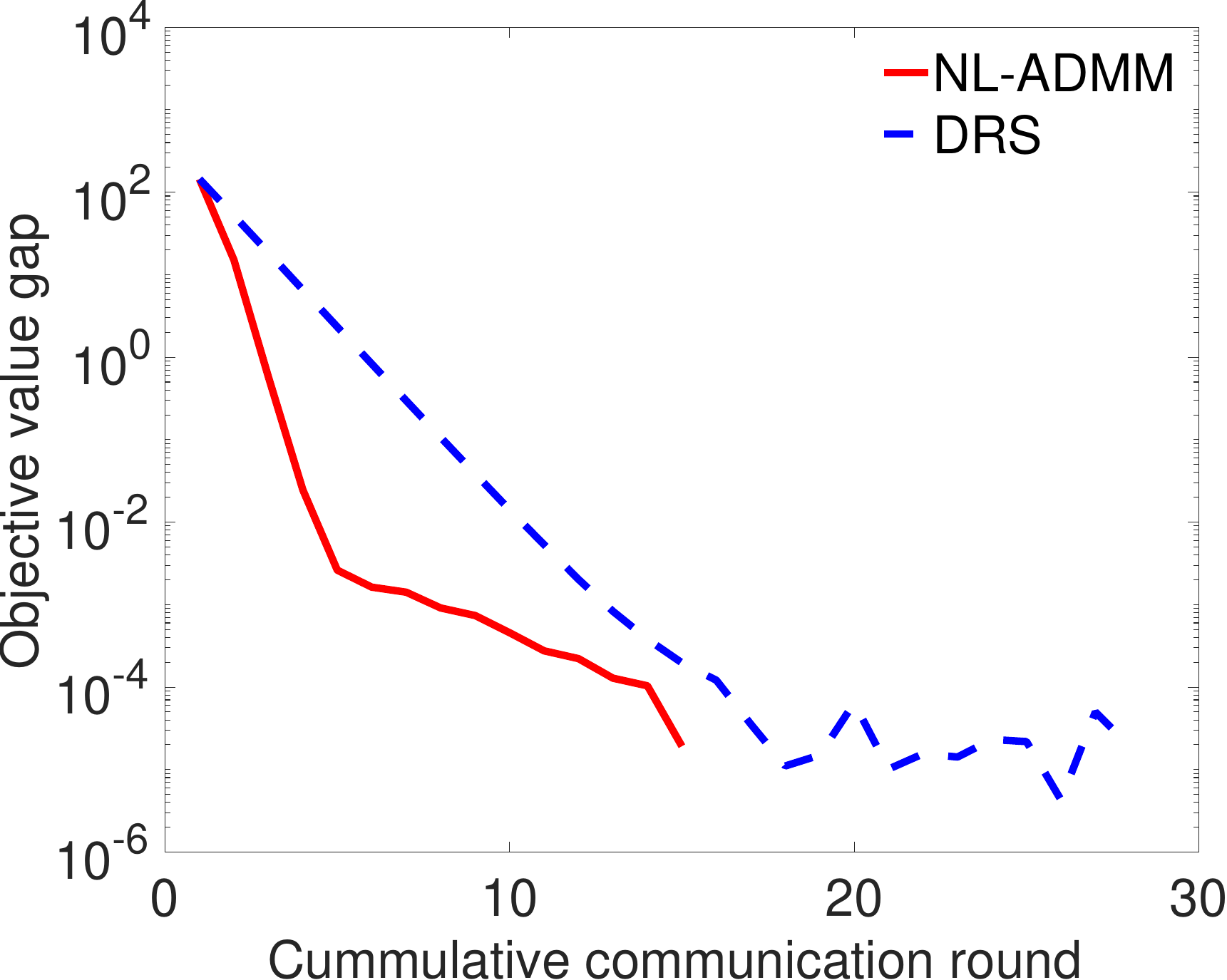}
    \label{fig:gap-p2-1}
  \end{subfigure}
  \begin{subfigure}{0.23\textwidth}
    \centering
    \includegraphics[width=\linewidth]{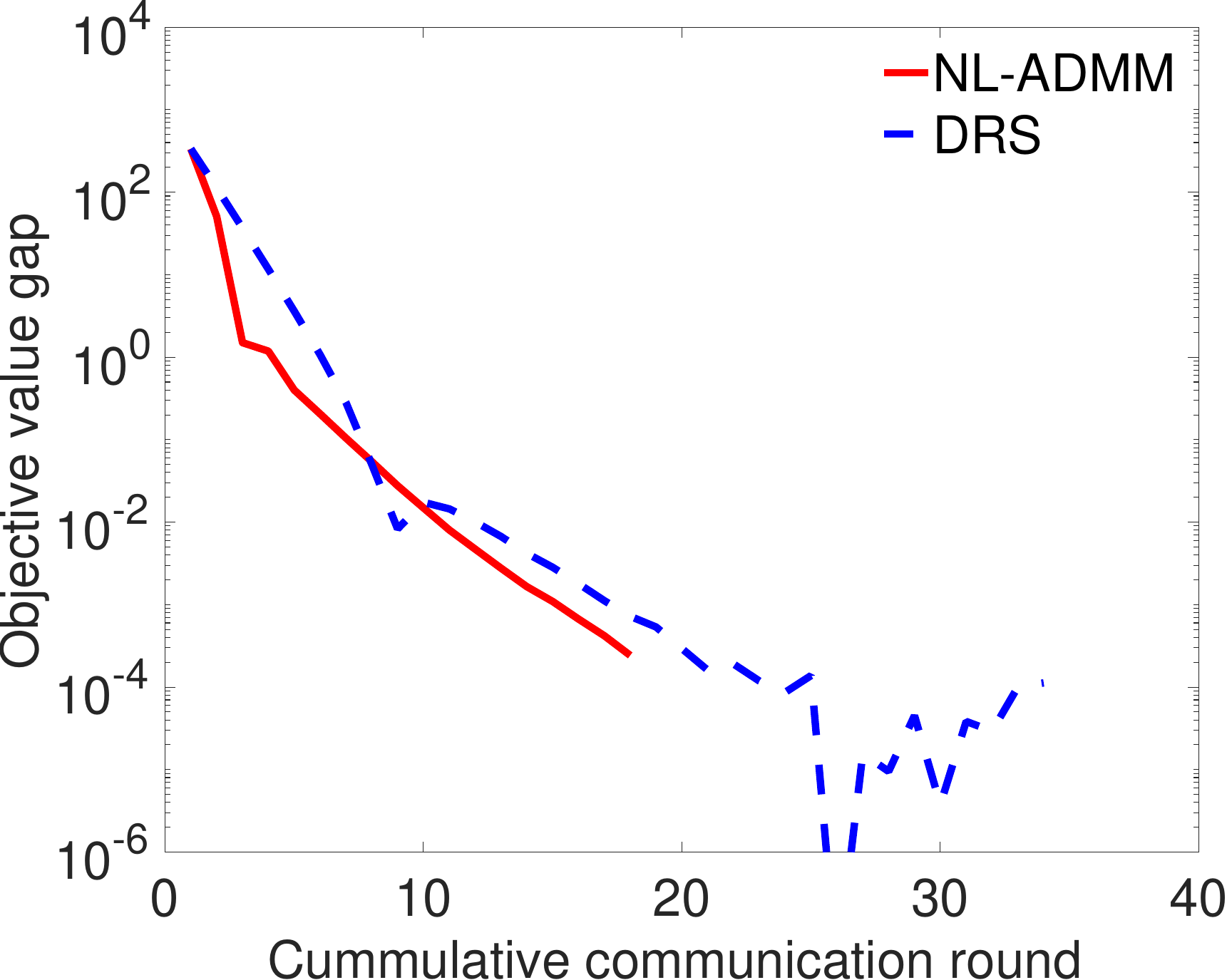}
    \label{fig:gap-p5-1}
  \end{subfigure}
  \begin{subfigure}{0.23\textwidth}
    \centering
    \includegraphics[width=\linewidth]{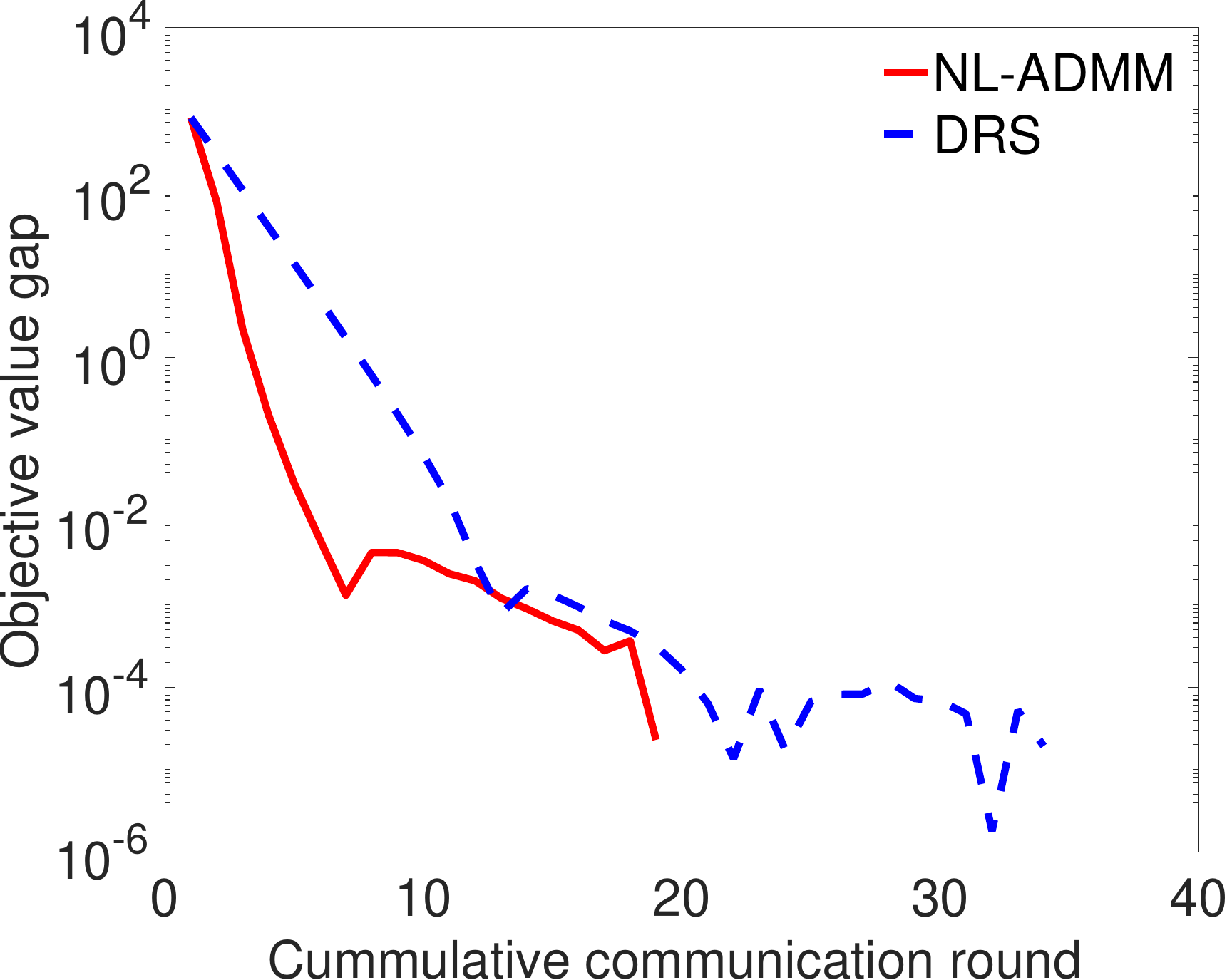}
    \label{fig:gap-p10-1}
  \end{subfigure}

  \vspace{0.4em}

  \centering
  \begin{subfigure}{0.23\textwidth}
    \centering
    \includegraphics[width=\linewidth]{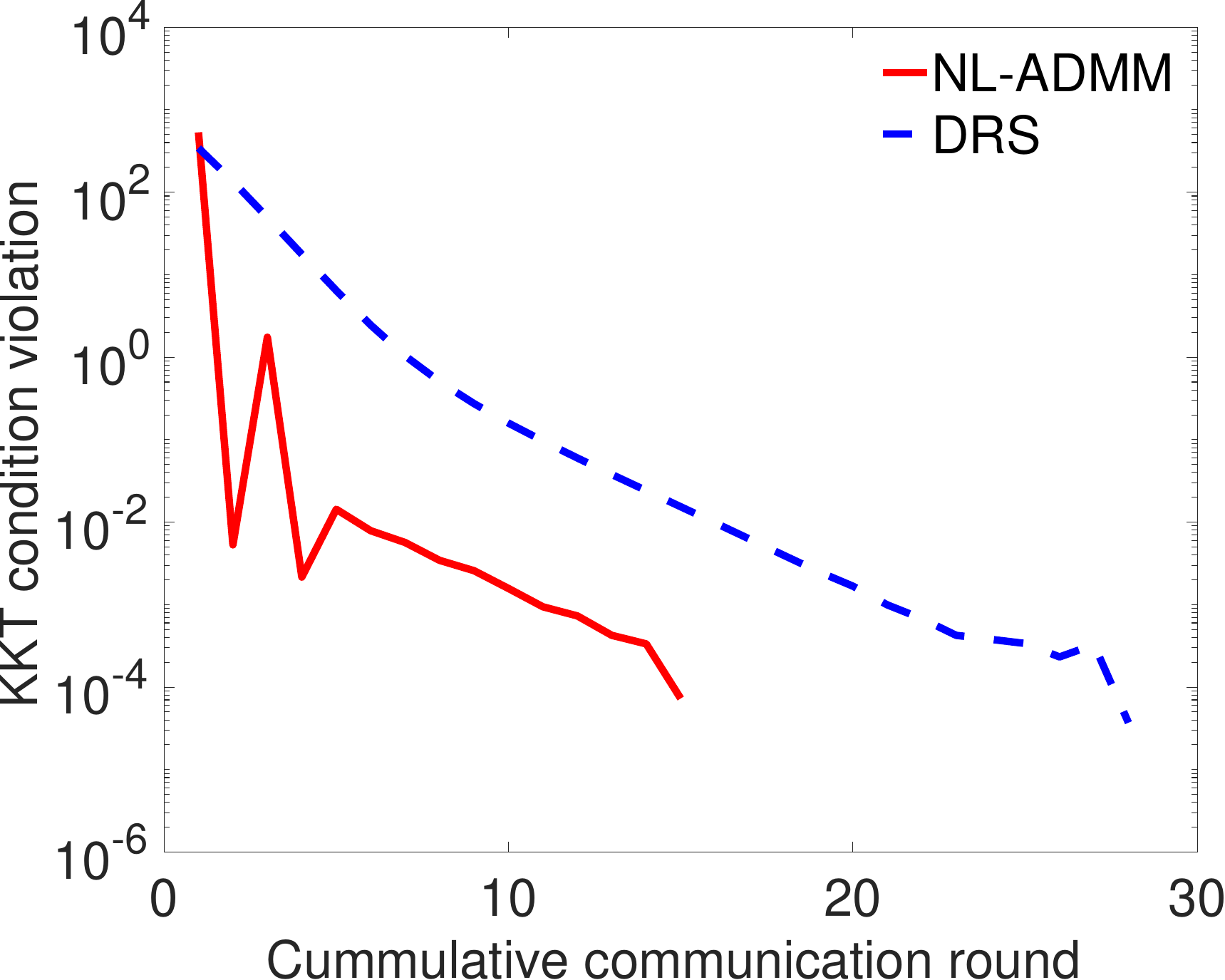}
    \caption{\(p=2\)}
    \label{fig:kkt-p2-1}
  \end{subfigure}
  \begin{subfigure}{0.23\textwidth}
    \centering
    \includegraphics[width=\linewidth]{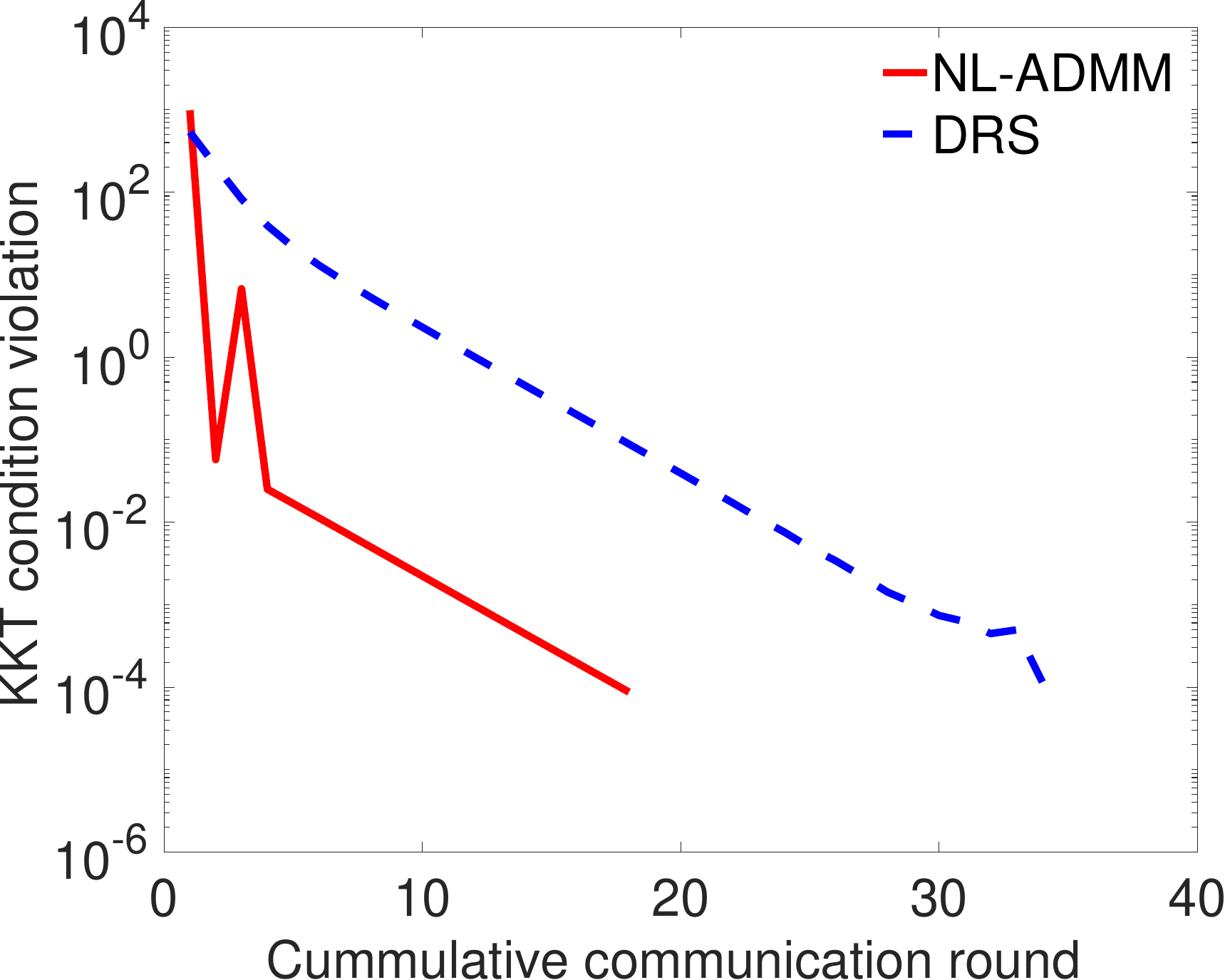}
    \caption{\(p=5\)}
    \label{fig:kkt-p5-1}
  \end{subfigure}
  \begin{subfigure}{0.23\textwidth}
    \centering
    \includegraphics[width=\linewidth]{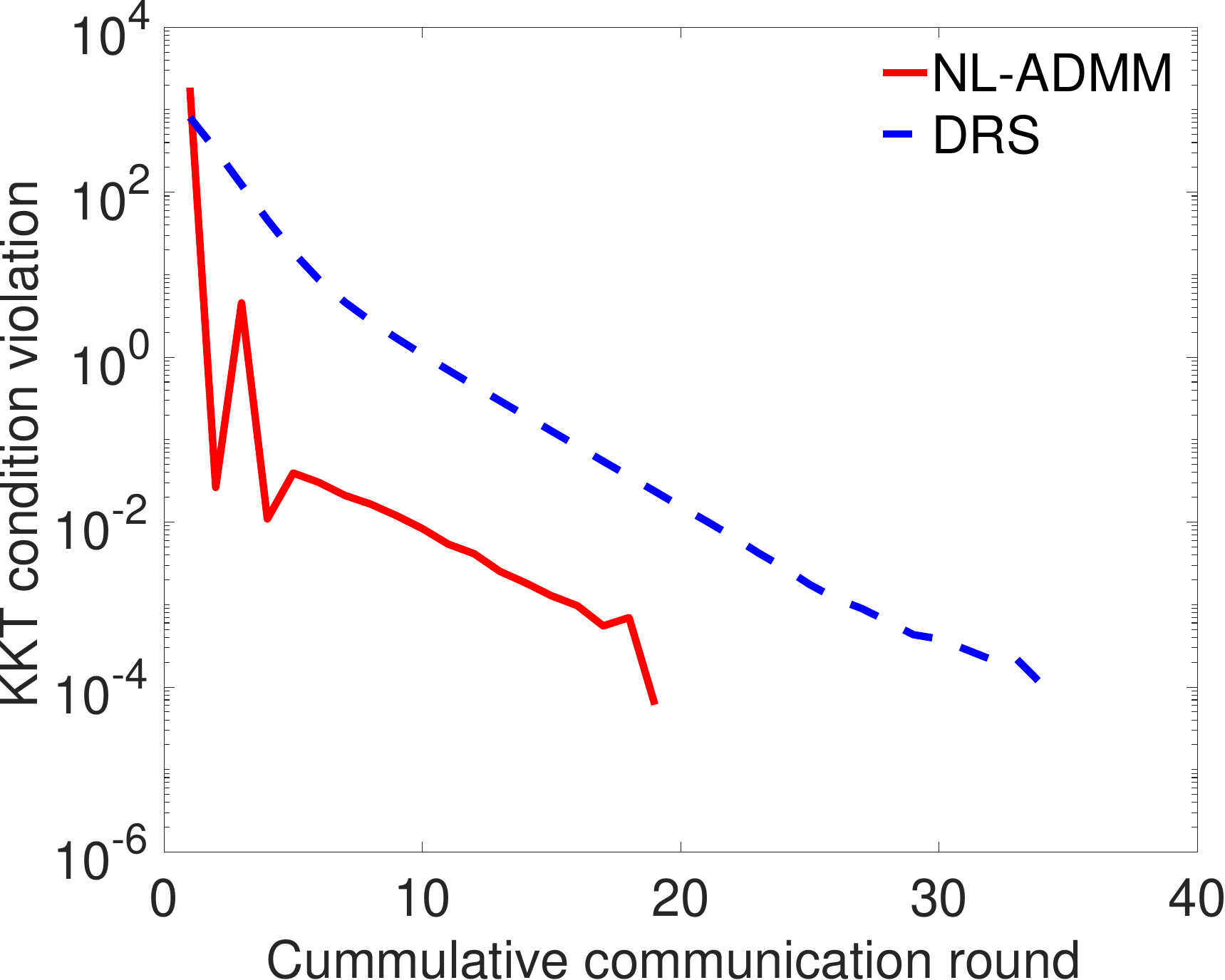}
    \caption{\(p=10\)}
    \label{fig:kkt-p10-1}
  \end{subfigure}

  \caption{Convergence behavior \yx{of the proposed NL-ADMM and the DRS baseline on solving instances of \eqref{eq:resource_allocation}}. 
  Top row: objective value gaps $|f(\vx)+g(\vy)-\big(f(\vx^{*})+g(\vy^{*})\big)|$. Bottom row: KKT violation, defined as the sum of primal, dual, and complementarity residuals. Columns correspond to \yx{different number $p$ of workers}. 
  The reference values $f(\vx^{*})+g(\vy^{*})$ are obtained from running ALM to a higher accuracy. 
  }
  \label{Fig:resource_allocation}
\end{figure}

\begin{table}[t]
\centering
\caption{Average communication rounds of the three compared methods and outer iterations of ALM over $10$ independent runs on instances of \eqref{eq:resource_allocation}.}
\label{tab:quad_comm}
\setlength{\tabcolsep}{10pt}
\begin{tabular}{c|ccc|c}
\toprule
$p$  & ALM (rounds) & DRS (rounds) & NL-ADMM (rounds) & ALM outer iterations \\
\midrule
2  & 1510.30 & 30.00 & 15.90 & 16.10 \\
5  & 1696.10 & 32.70 & 18.50 & 10.90 \\
10 & 2240.20 & 34.50 & 19.10 & 9.00  \\
\bottomrule
\end{tabular}
\end{table}

Table \ref{tab:quad_comm} reports the averages over 10 independently generated instances. NL-ADMM consistently requires the fewest communication rounds, with only a mild increase as $p$ grows. DRS also demonstrates reasonable scalability with respect to 
$p$, but at a noticeably higher communication cost. 
\yx{In} contrast, the 
ALM 
incurs substantially larger communication counts, 
even though the number of outer iterations remains small. Overall, NL-ADMM demonstrates favorable communication efficiency and scalability. The superior performance of NL-ADMM is attributed to the variable splitting formulation, as well as our new design of the algorithm. 

\subsection{Distributed machine learning}
\yx{In this subsection, we test the proposed algorithm on problem \eqref{eq:ml_base} or equivalently \eqref{eq:consensus}.} 
We 
set the functions in \eqref{eq:consensus_setup} as: 
\begin{equation*}
f_j(\vx_{j})
= \frac{1}{N}\sum_{i\in \mathcal{A}_{j}}\log\Big(1+\exp\big(-(\bm{\xi}^{1}_{i})^{\top}\vx_{j}\big)\Big)
+\frac{c}{2p}\|\vx_{j}\|_2^2,
\quad
h_j(\vx_{j})
= \frac{1}{M}\sum_{i\in \mathcal{B}_{j}}\log\Big(1+\exp\big((\bm{\xi}^{2}_{i})^{\top}\vx_{j}\big)\Big)-\frac{r}{p}.
\end{equation*}
Here, $c>0$ denotes the 
regularization parameter and 
is fixed to $c=10^{-2}$ in all experiments, and $r>0$ specifies the 
threshold of a violation score and 
is set to $r=0.5$.

\yx{Again, we compare the proposed NL-ADMM with ALM and DRS.} Three benchmark datasets from \cite{chang2011libsvm} are used for evaluation:
\texttt{medelon}, \texttt{australian\_scale}, and \texttt{a9a}. 
For problem \eqref{eq:ml_local_form}, the ALM formulation does not include the consensus constraint; consequently, the ALM updates reduce exactly to \eqref{alm_primal} and \eqref{alm_scaled_u}. DRS applied to \eqref{eq:consensus} is initialized with $\vw^{0}=(\vw_{\vu}^{0},\vw_{\vv}^{0})$ and parameters 
$\beta_{\mathrm{DRS}}>0$ and $\eta_{\mathrm{DRS}}\in(0,1]$, and it performs the updates:
\begin{subequations}
\label{DRS_algorithm_2}
\begin{align}
(\vu^{k},\vv^{k})
&=
\left(
\tfrac{1}{p}\mathbf{1}\mathbf{1}^{\top}\vw_{\vu}^{k},\;
\vw_{\vv}^{k}-(\mathbf{1}\otimes I)\bar{\vw}_{\vv}^{k}
\right), \ \text{with} \
\bar{\vw}_{\vv}^{k}
=
\tfrac{1}{p}\sum_{j=1}^{p}\vw_{\vv,j}^{k}, \label{eq:DRS-uv}\\
\vq^{k}
&=
(\vq_{\vu}^{k},\vq_{\vv}^{k})
=
\left(
2\vu^{k}-\vw_{\vu}^{k},\;
2\vv^{k}-\vw_{\vv}^{k}
\right),
\\
\vx_{j}^{k}
&\in
\argmin_{\vx_{j}}\;
f_{j}(\vx_{j})
+\tfrac{\beta_{\mathrm{DRS}}}{2}
\bigl(\,[h_{j}(\vx_{j})+\vq_{\vu,j}^{k}]_{+}\bigr)^{2}+\tfrac{\beta_{\mathrm{DRS}}}{2}\,
\bigl\|\vx_{j}+\vq_{\vv,j}^{k}\bigr\|^{2},
\quad j=1,\dots,p, \\
\vt^{k}
&=
(\vt_{\vu}^{k},\vt_{\vv}^{k})
=
\left(
[\vq_{\vu}^{k}+\vh(\vx^{k})]_{+},\;
\vq_{\vv}^{k}+\vx^{k}
\right),
 \\
\vw^{k+1}
&=
\vw^{k}
+2\eta_{\mathrm{DRS}}
\bigl(
\vt_{\vu}^{k}-\vu^{k},\;
\vt_{\vv}^{k}-\vv^{k}
\bigr).\label{eq:DRS-w-p2}
\end{align}
\end{subequations}
Detailed derivation is given in the appendix.

\yx{For each $p\in\{2,5,10\}$, we set the penalty parameter(s) to $\beta_{1}=1.00$ and $\beta_{2}=0.05$ for the NL-ADMM, $\beta_{\mathrm{ALM}}=0.50$ for ALM, and  $\beta_{\mathrm{DRS}}=0.10$ for DRS. In addition, the relaxation parameter for DRS is set to $\eta_{\mathrm{DRS}}=0.75$.} Similar to the previous experiment, to accommodate the heterogeneity and scaling of real-world datasets, these parameter values are selected via  grid search over candidate values spanning several orders of magnitude to achieve the best communication efficiency. 
The resulting choices are then fixed and used consistently across all datasets and problem instances. 

Following the same setup as in the previous experiment, all inner minimization subproblems from  ALM, NL-ADMM and DRS formulation, are solved by APG with line search, with tolerance $\texttt{subtol}=10^{-6}$. We monitor the primal, dual, and complementarity residuals to assess convergence, and all algorithms terminate once all of these residuals fall below a tolerance $\texttt{tol}=10^{-5}$. 
For benchmarking, we additionally compute a high-accuracy ALM solution for each dataset to serve as the reference optimal objective value, using tighter stopping criteria
$(\texttt{subtol},\texttt{tol})=(10^{-8},10^{-7})$, $(2\times10^{-9},2\times10^{-8})$, and $(10^{-8},10^{-7})$ for the \texttt{madelon}, \texttt{australian\_scale}, and \texttt{a9a} datasets, respectively.



As in the previous experiment, convergence is 
reported with respect to 
cumulative communication rounds. 
Each of NL-ADMM and DRS incurs one broadcast and one aggregation per iteration, which together constitute one communication round, whereas ALM \yx{requires one communication round for each APG step.} 
Figure~\ref{fig:madelon-2x3} shows the convergence curves of the three compared methods for the \texttt{madelon} dataset. We include the 
ALM trajectory to highlight the advantage of solving the variable splitting formulation by NL-ADMM and DRS. 
In Figures \ref{fig:australian_scale-2x3} and \ref{fig:a9a_scale-2x3} for the \texttt{australian\_scale} and \texttt{a9a} datasets, 
\yx{we omit the ALM curve to show the significant advantage of the proposed NL-ADMM over DRS. Clearly, the proposed method is much more communication-efficient than DRS.} 

\begin{figure}
  \centering
  \begin{subfigure}{0.23\textwidth}
    \centering
    \includegraphics[width=\linewidth]{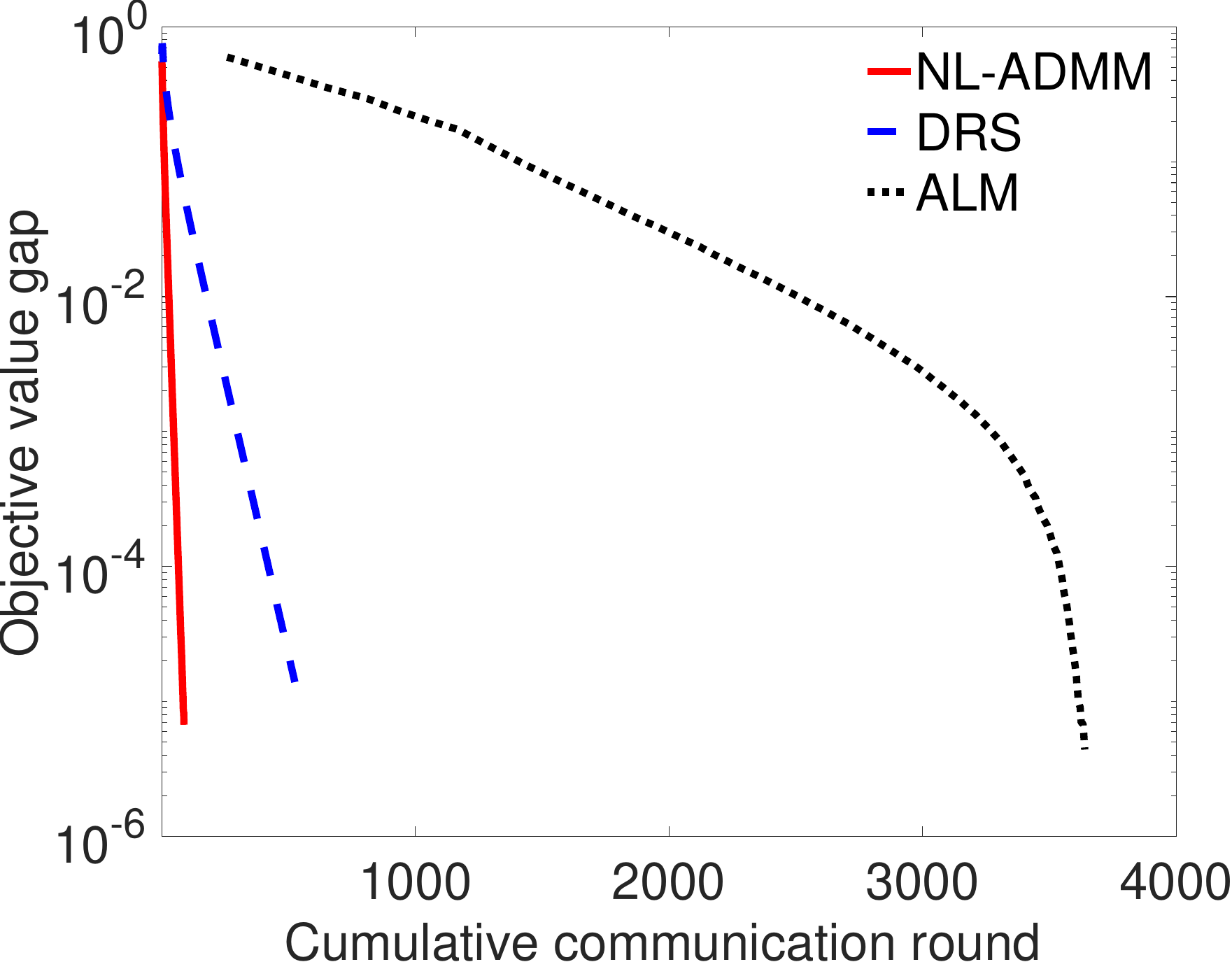}
    \label{fig:gap-p2}
  \end{subfigure}
  \begin{subfigure}{0.23\textwidth}
    \centering
    \includegraphics[width=\linewidth]{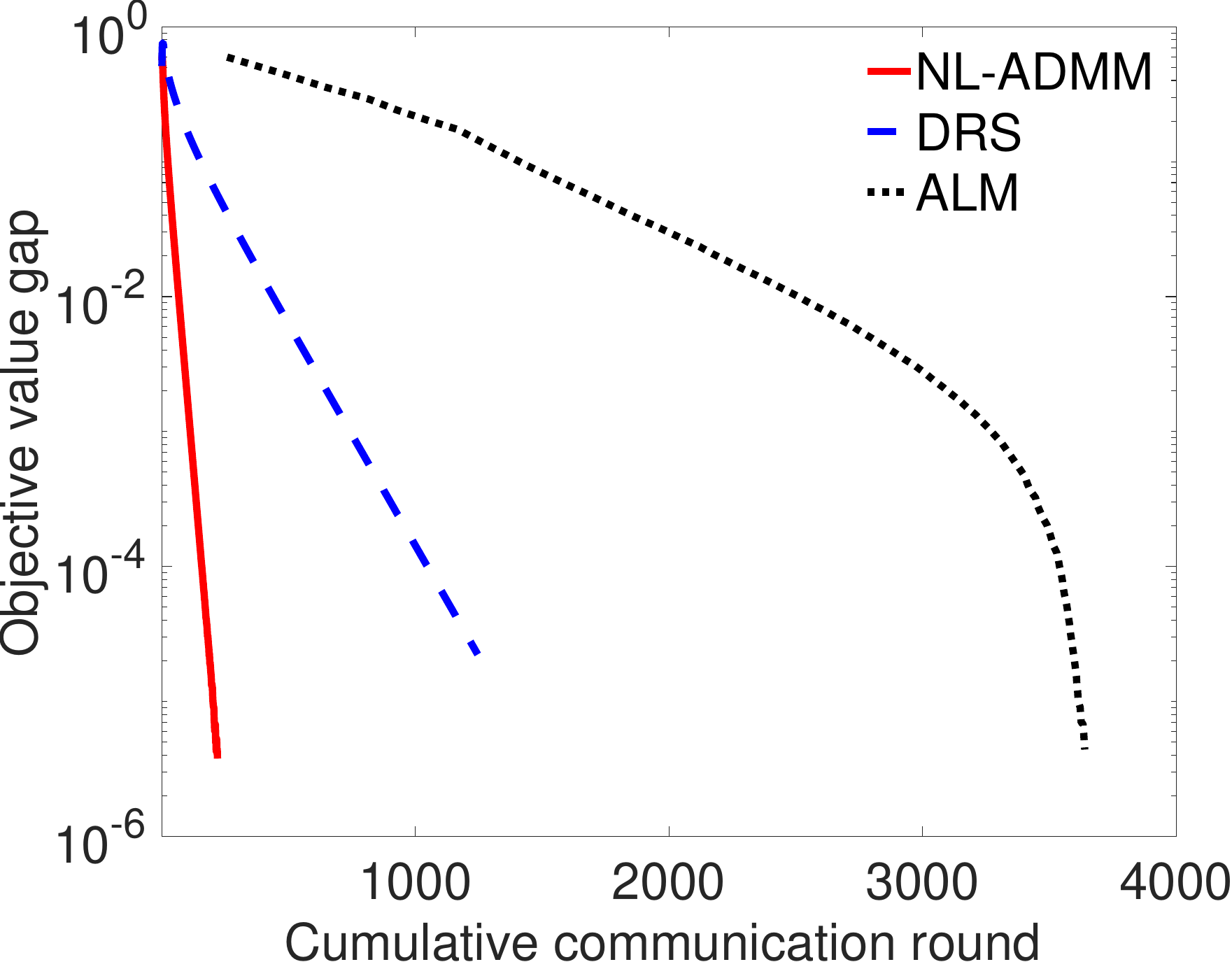}
    \label{fig:gap-p5}
  \end{subfigure}
  \begin{subfigure}{0.23\textwidth}
    \centering
    \includegraphics[width=\linewidth]{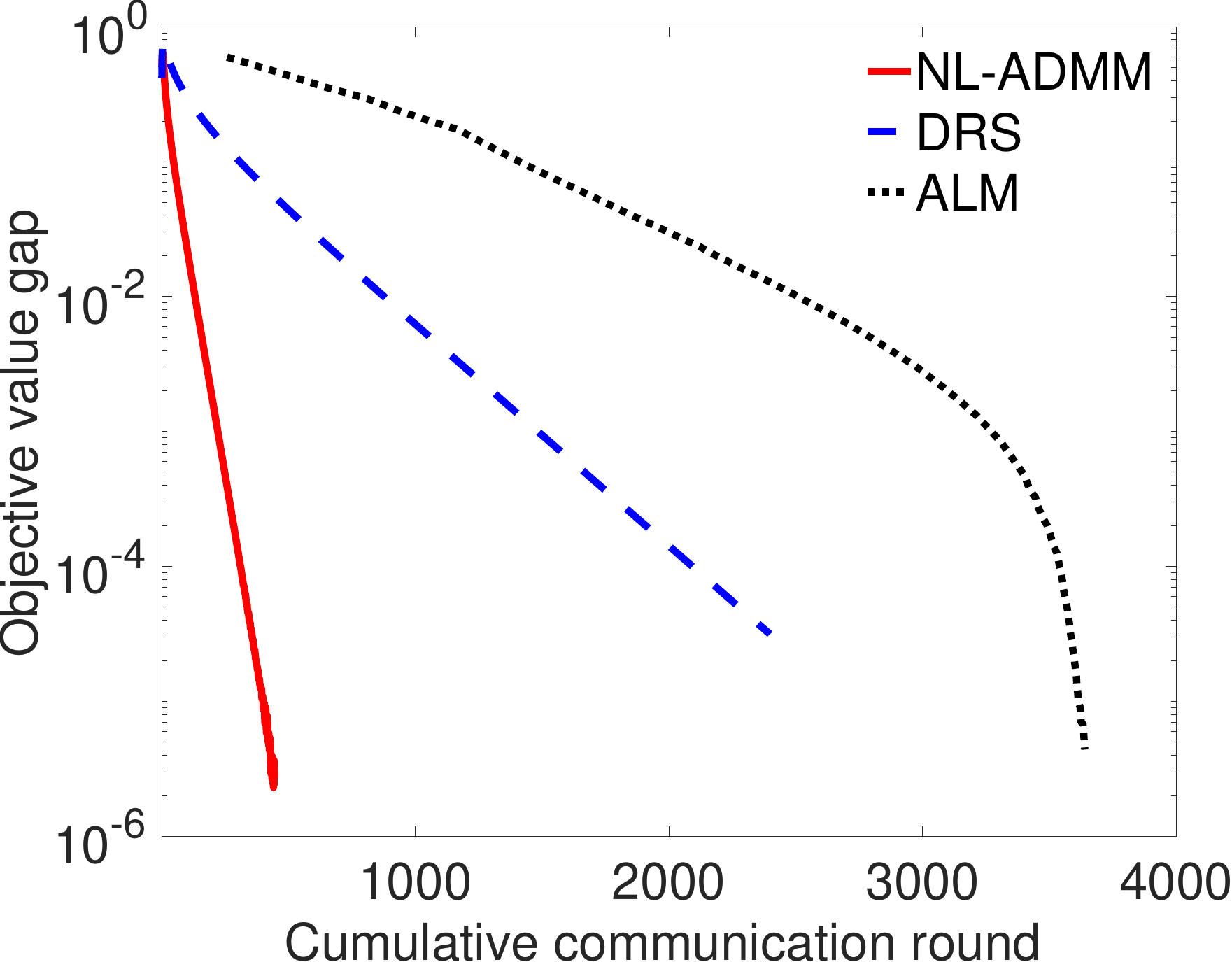}
    \label{fig:gap-p10}
  \end{subfigure}

  \vspace{0.4em}

  \centering
  \begin{subfigure}{0.23\textwidth}
    \centering
    \includegraphics[width=\linewidth]{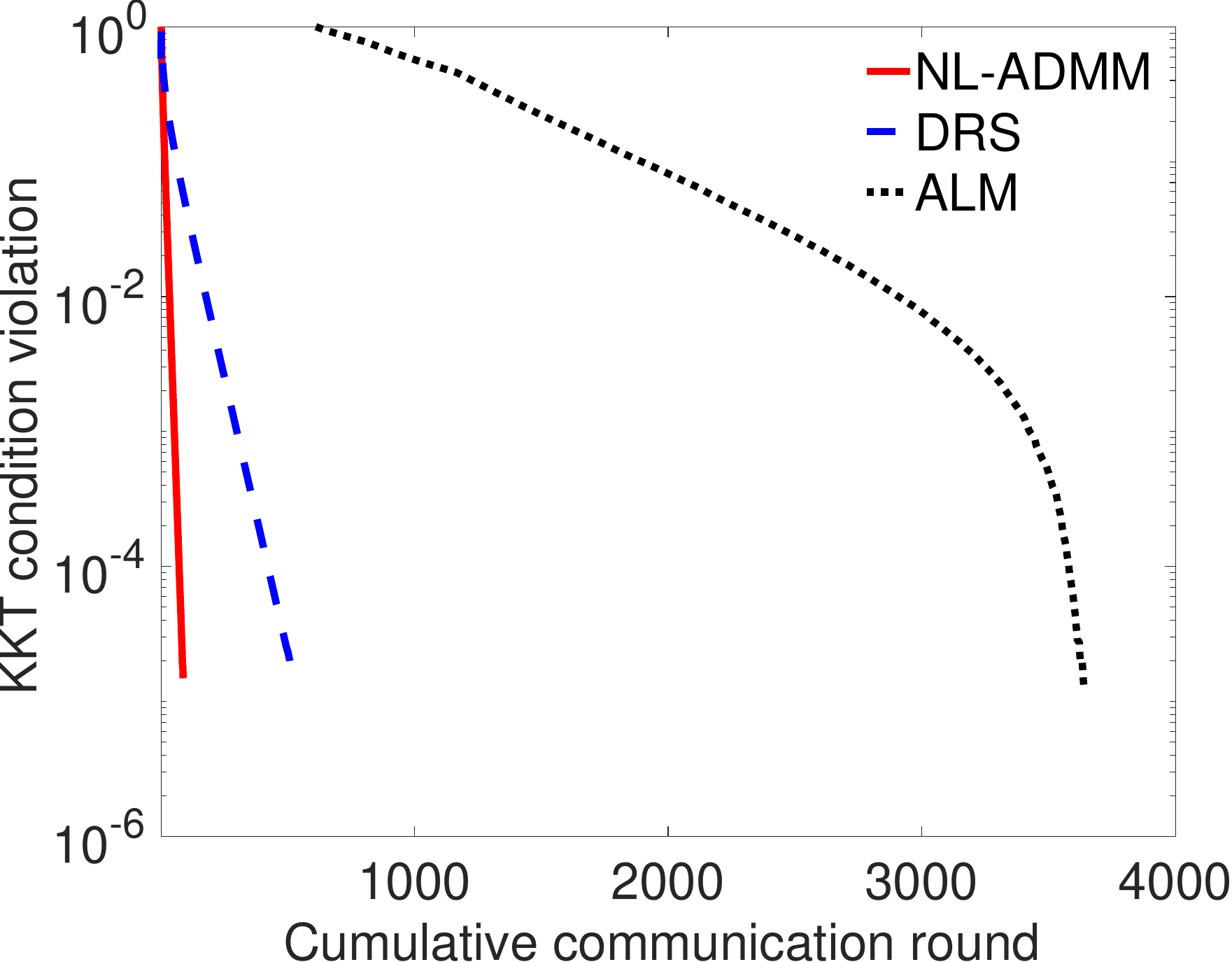}
    \caption{\(p=2\)}
    \label{fig:kkt-p2}
  \end{subfigure}
  \begin{subfigure}{0.23\textwidth}
    \centering
    \includegraphics[width=\linewidth]{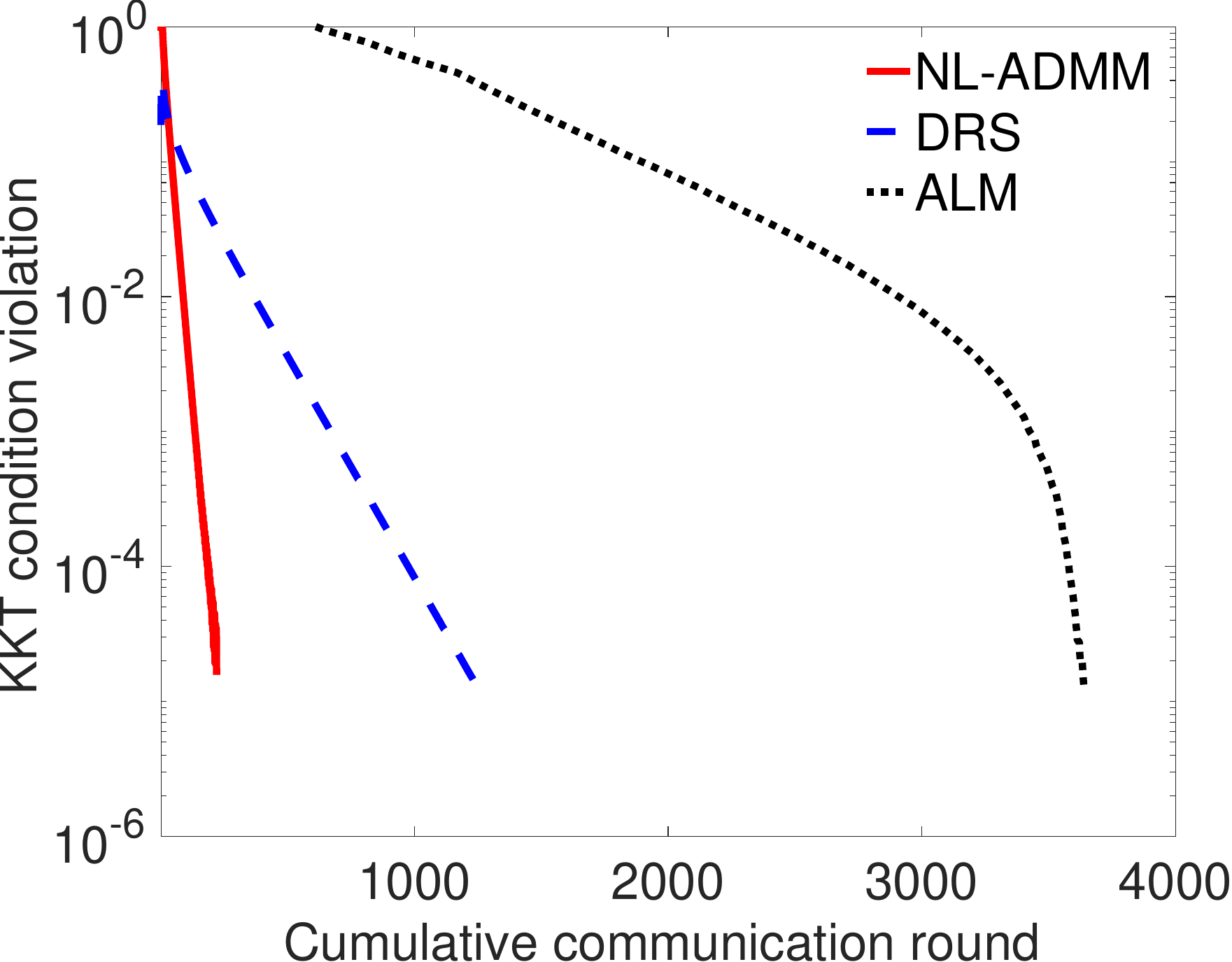}
    \caption{\(p=5\)}
    \label{fig:kkt-p5}
  \end{subfigure}
  \begin{subfigure}{0.23\textwidth}
    \centering
    \includegraphics[width=\linewidth]{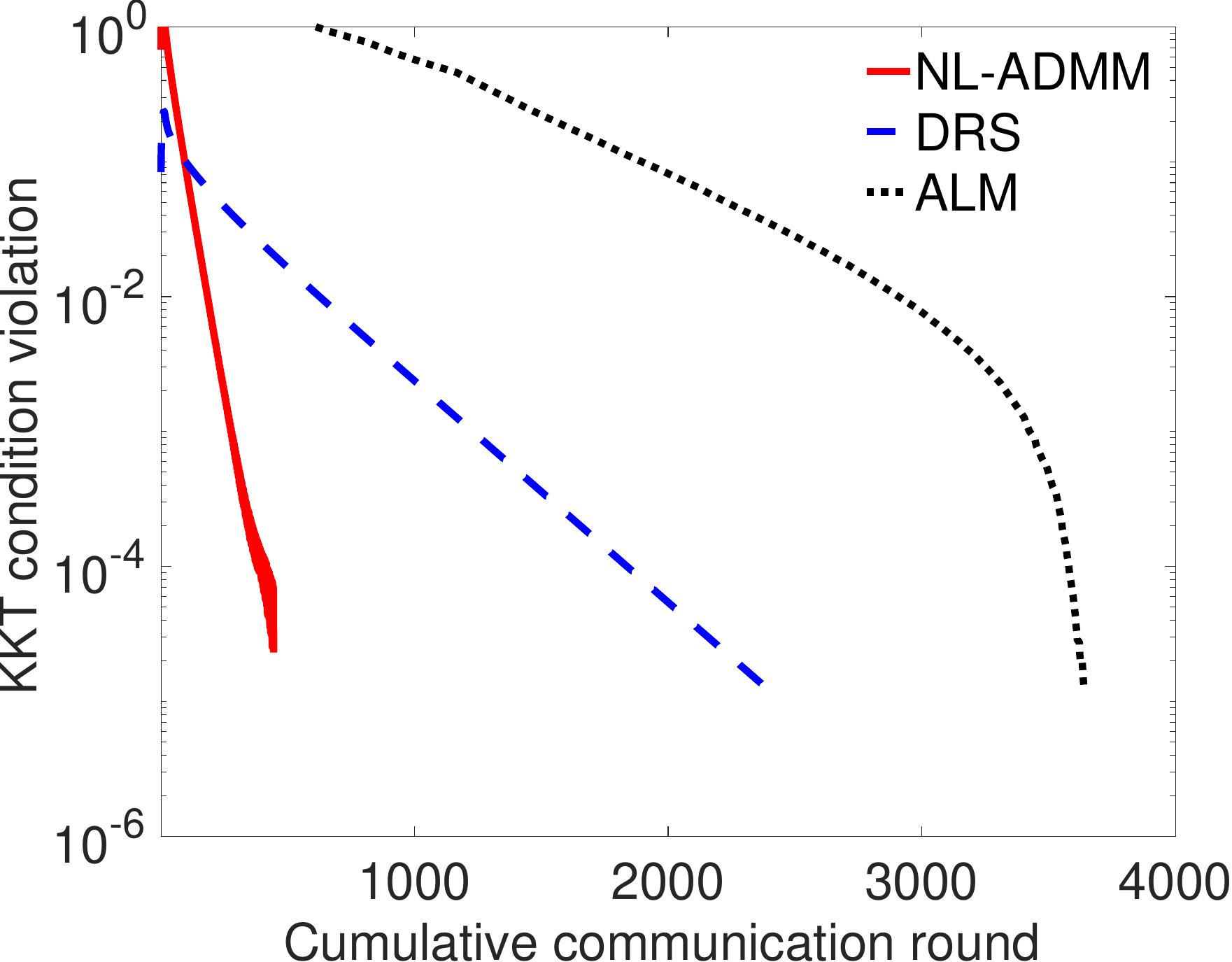}
    \caption{\(p=10\)}
    \label{fig:kkt-p10}
  \end{subfigure}

  \caption{Convergence behavior of three compared methods on solving instance of \eqref{eq:consensus} with the \texttt{madelon} dataset. 
  Top row: objective value gaps $|f(\vx)+g(\tilde{\vy})-\big(f(\vx^{*})+g(\tilde{\vy}^{*})\big)|$. Bottom row: KKT violation, defined as the sum of primal, dual, and complementarity residuals. Columns correspond to different number $p$ of workers. The reference values $f(\vx^{*})+g(\tilde{\vy}^{*})$ are obtained from running ALM to a higher accuracy. The displayed 
  ALM uses $94$ outer iterations.}
  \label{fig:madelon-2x3}
\end{figure}

\begin{figure}
    \centering
  \begin{subfigure}{0.23\textwidth}
    \includegraphics[width=\linewidth]{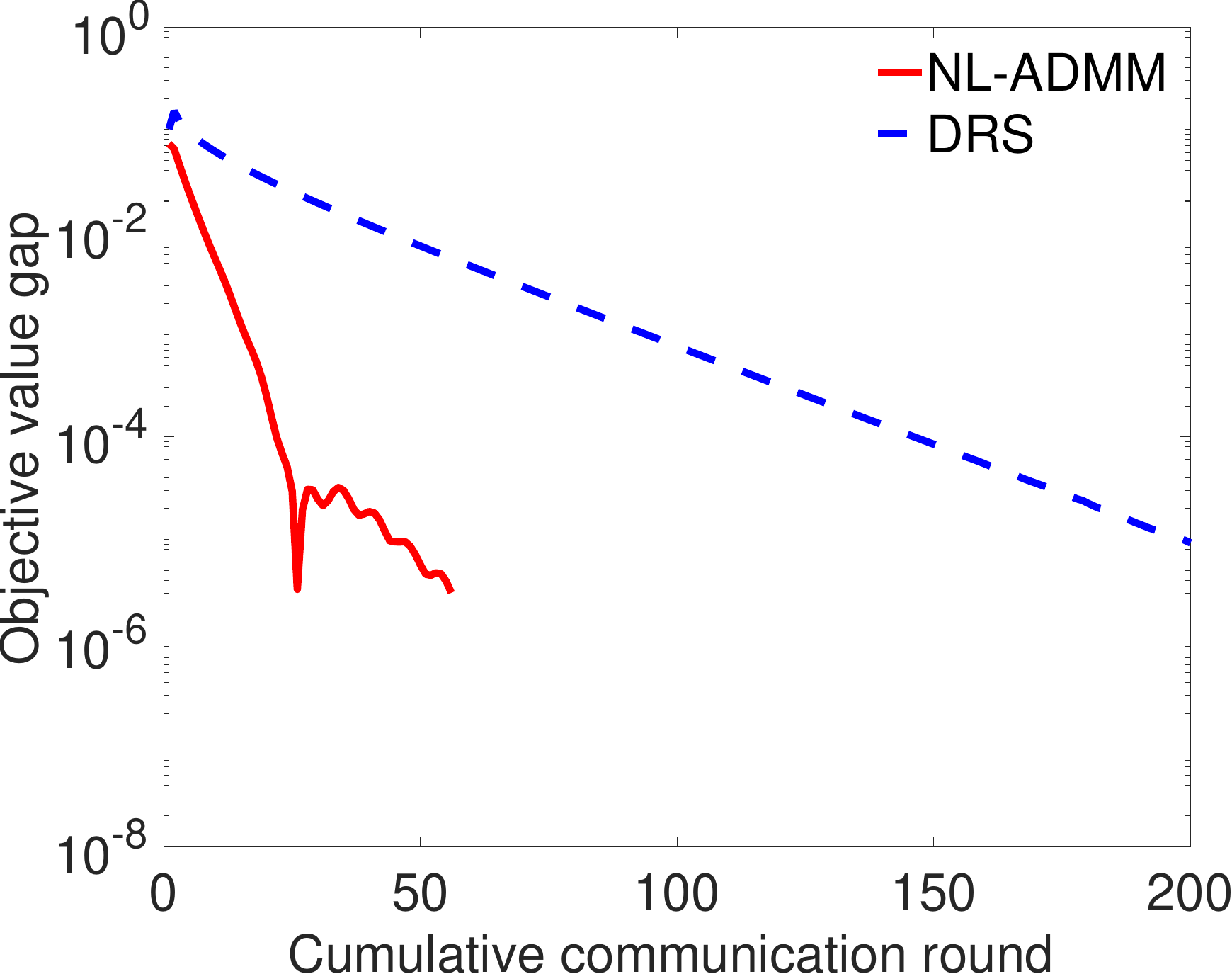}
    \label{fig:aus-gap-p2}
  \end{subfigure}
  \begin{subfigure}{0.23\textwidth}
    \centering
    \includegraphics[width=\linewidth]{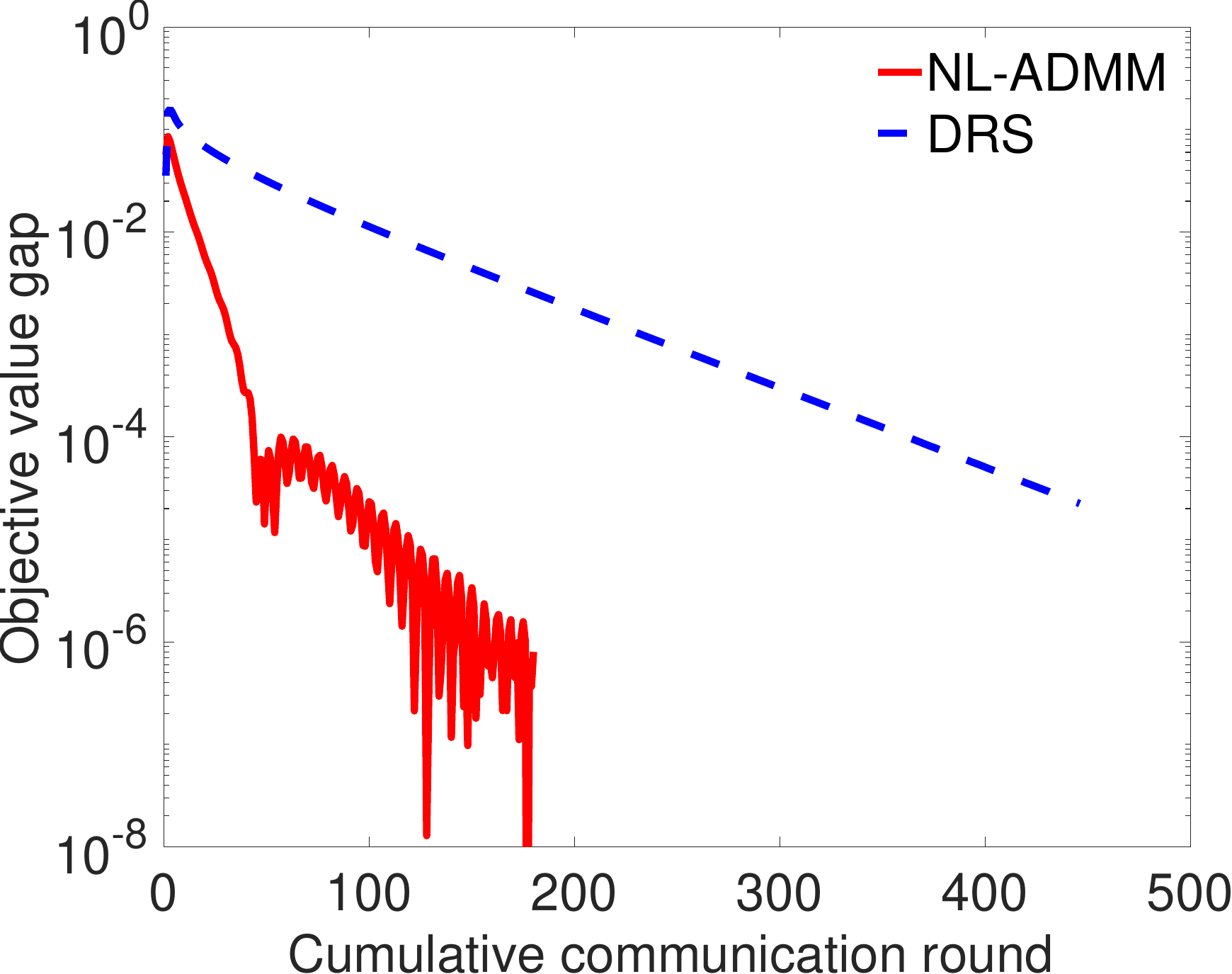}
    \label{fig:aus-gap-p5}
  \end{subfigure}
  \begin{subfigure}{0.23\textwidth}
    \centering
    \includegraphics[width=\linewidth]{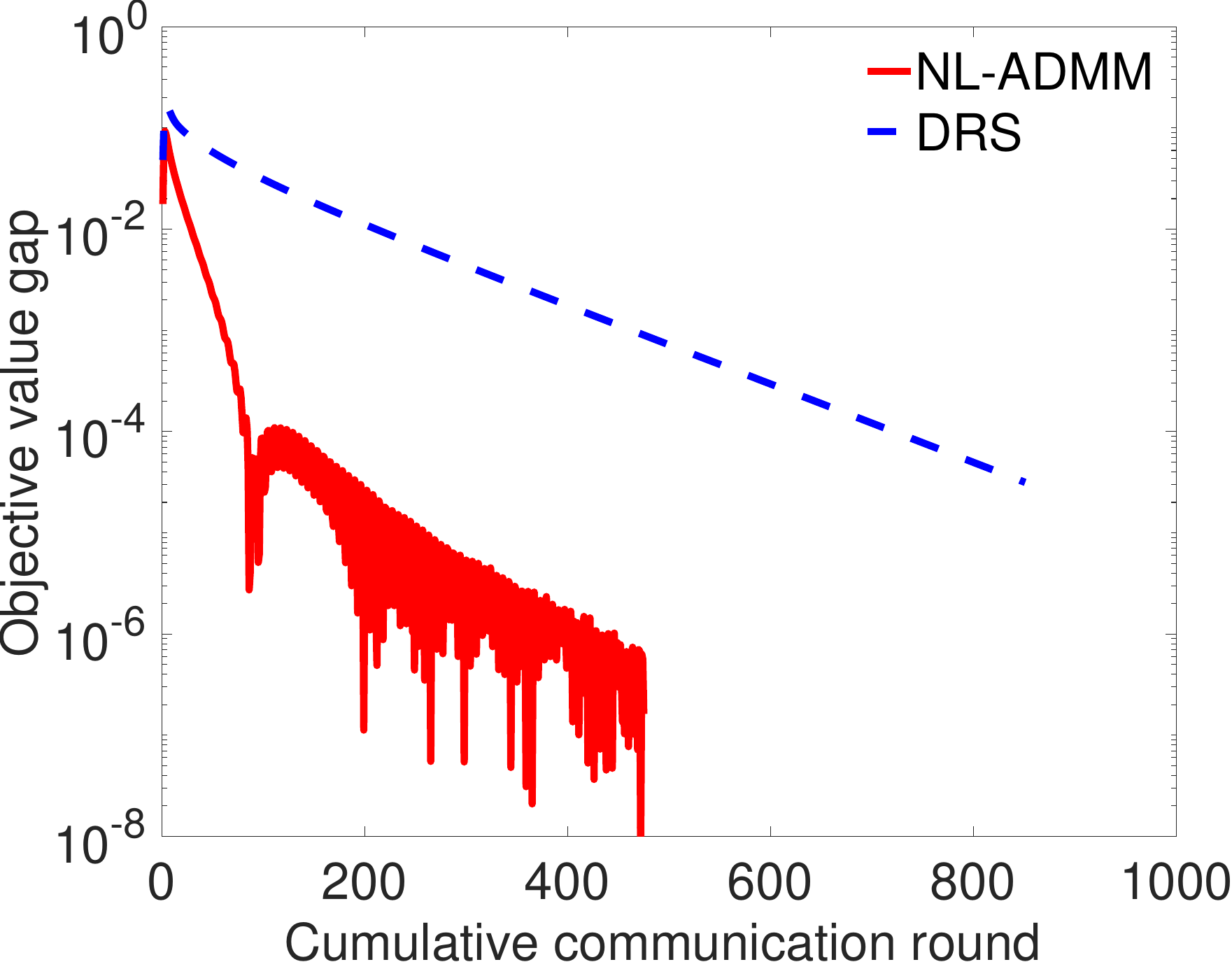}
    \label{fig:aus-gap-p10}
  \end{subfigure}

  \vspace{0.4em}

  \centering
  \begin{subfigure}{0.23\textwidth}
    \includegraphics[width=\linewidth]{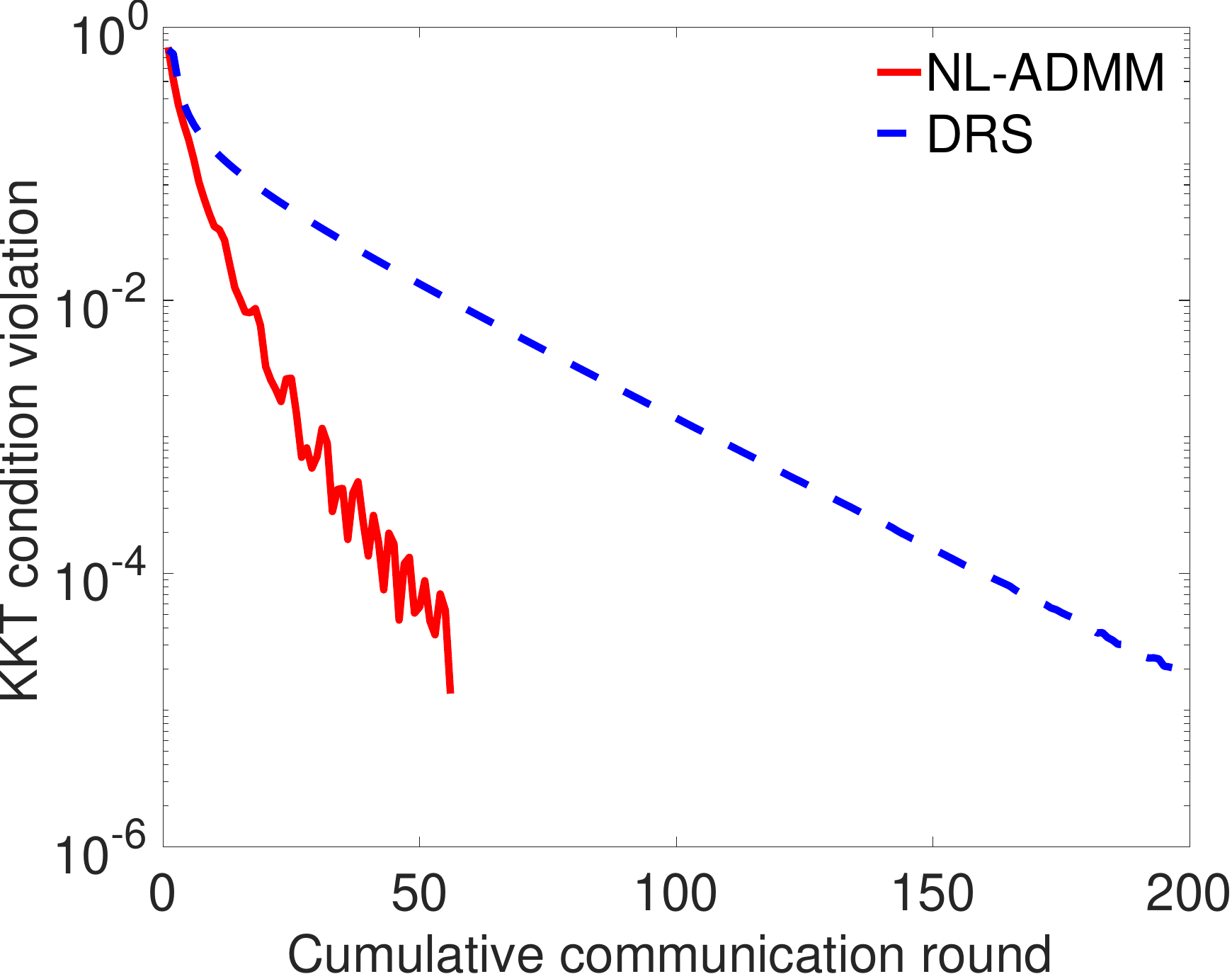}
    \caption{\(p=2\)}
    \label{fig:aus-kkt-p2}
  \end{subfigure}
  \begin{subfigure}{0.23\textwidth}
    \centering
    \includegraphics[width=\linewidth]{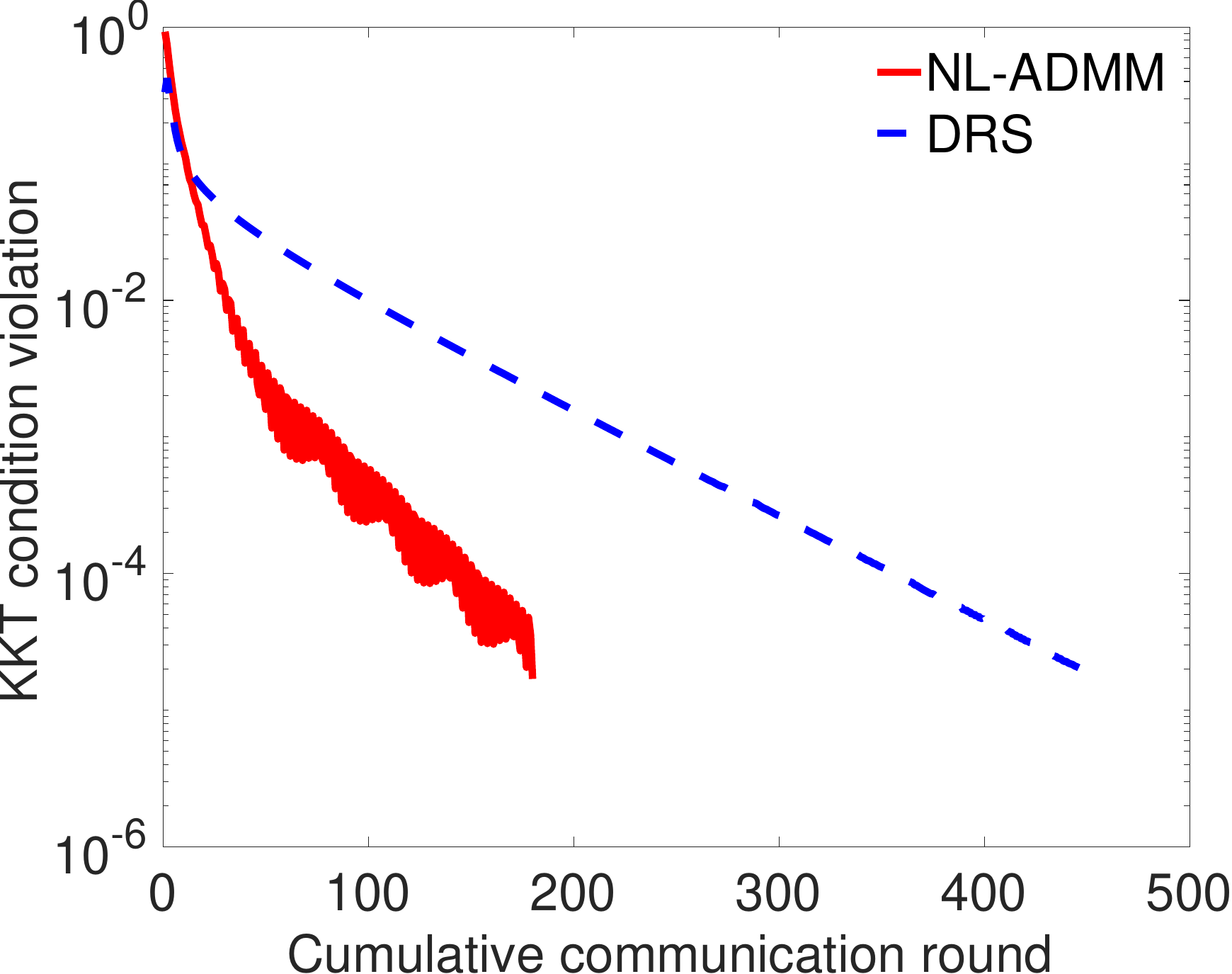}
    \caption{\(p=5\)}
    \label{fig:aus-kkt-p5}
  \end{subfigure}
  \begin{subfigure}{0.23\textwidth}
    \includegraphics[width=\linewidth]{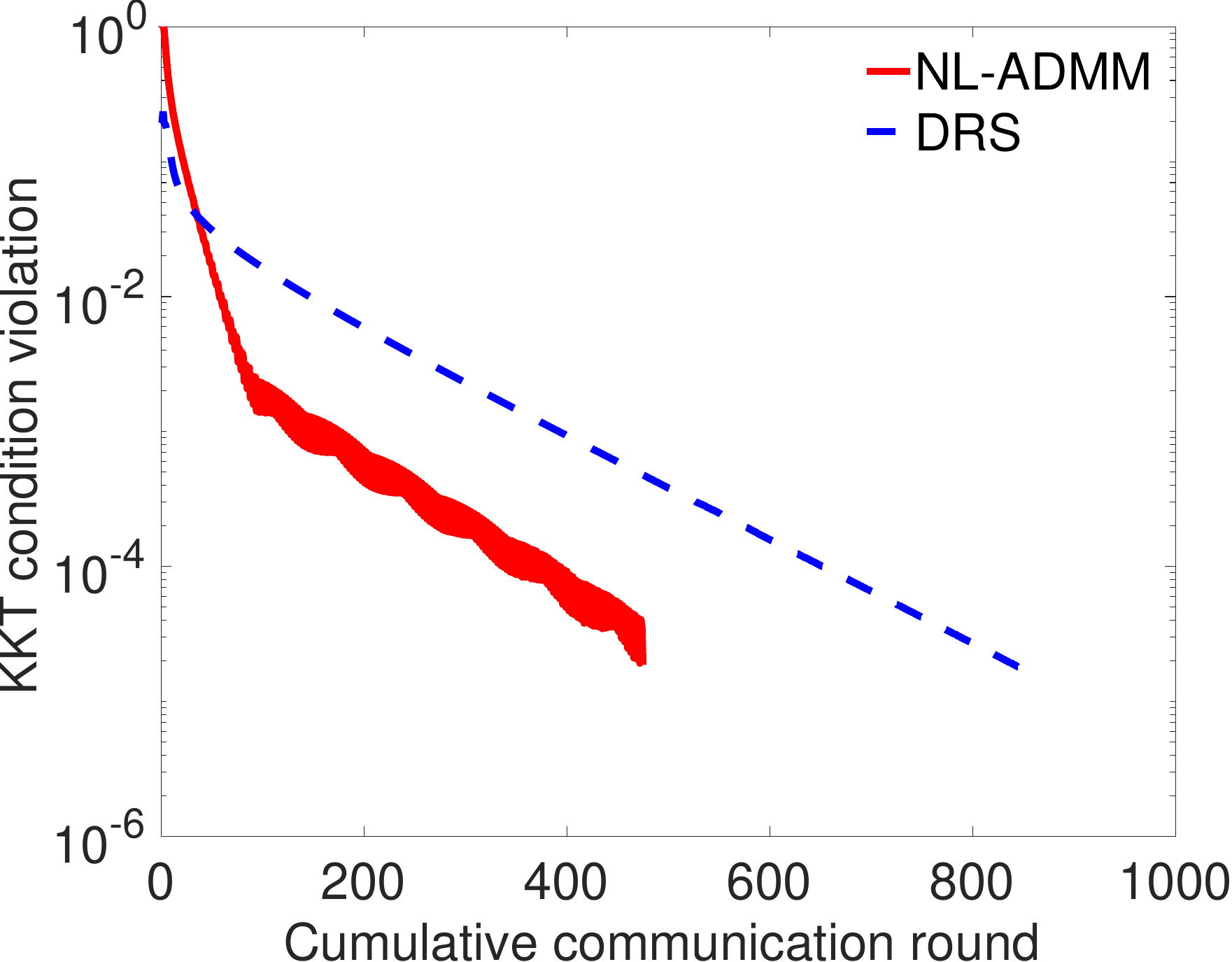}
    \caption{\(p=10\)}
    \label{fig:aus-kkt-p10}
  \end{subfigure}

  \caption{Convergence behavior of the proposed nonlinear ADMM and the DRS baseline on solving instance of \eqref{eq:consensus} with the \texttt{australian\_scale} dataset. 
  The top row shows objective value gaps $|f(\vx)+g(\tilde{\vy})-\big(f(\vx^{*})+g(\tilde{\vy}^{*})\big)|$, while the bottom row reports KKT violation defined as the sum of primal, dual, and complementarity residuals. Columns correspond to different number $p$ of workers. 
  Reference values $f(\vx^{*})+g(\tilde{\vy}^{*})$ are obtained from a high-accuracy ALM solution. The moderate-accuracy ALM run uses $39$ outer iterations and $4415$ communication rounds.}
  \label{fig:australian_scale-2x3}
\end{figure}
\begin{figure}
    \centering
  \begin{subfigure}{0.23\textwidth}
    \includegraphics[width=\linewidth]{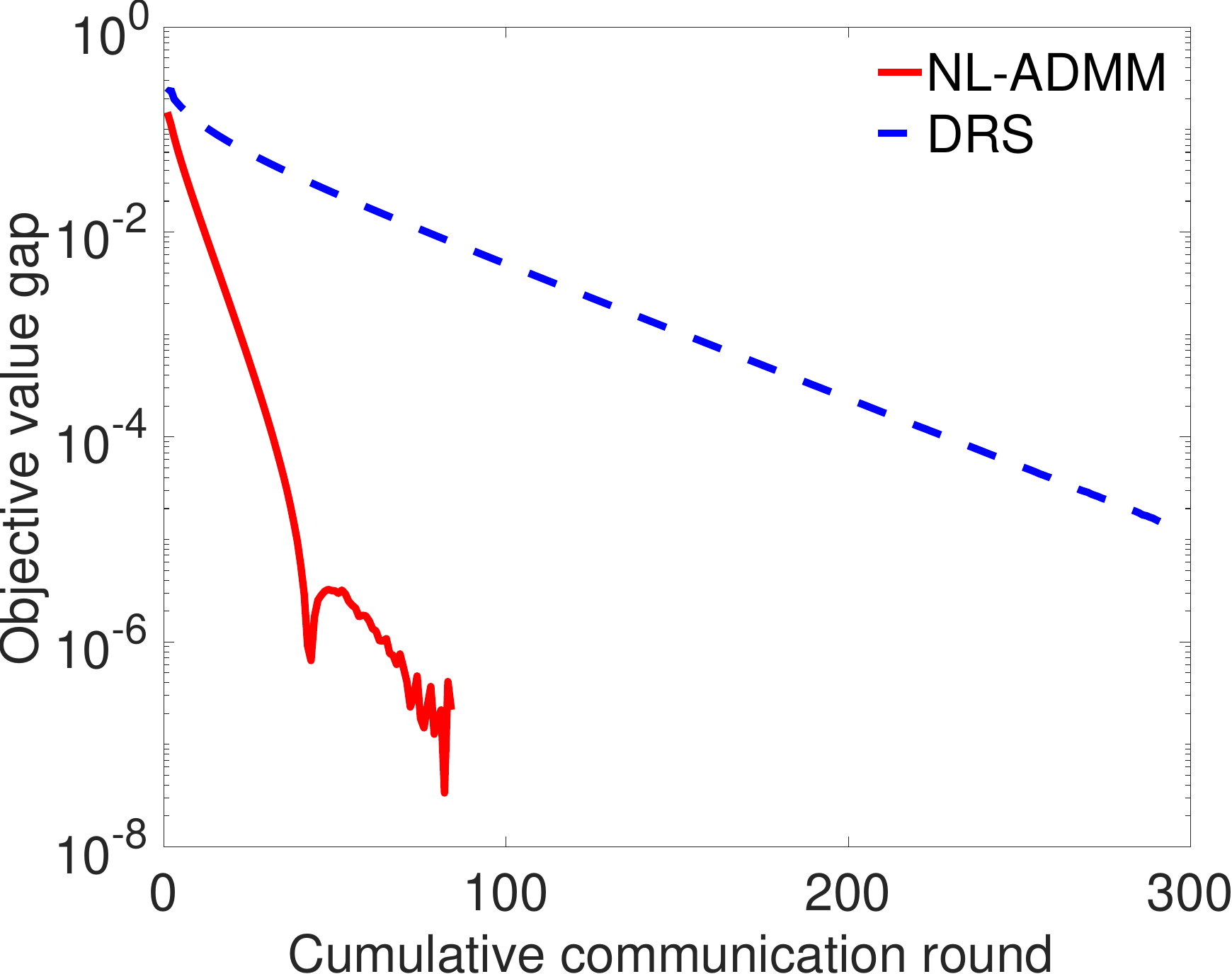}
    \label{fig:a9a-gap-p2}
  \end{subfigure}
  \begin{subfigure}{0.23\textwidth}
    \centering
    \includegraphics[width=\linewidth]{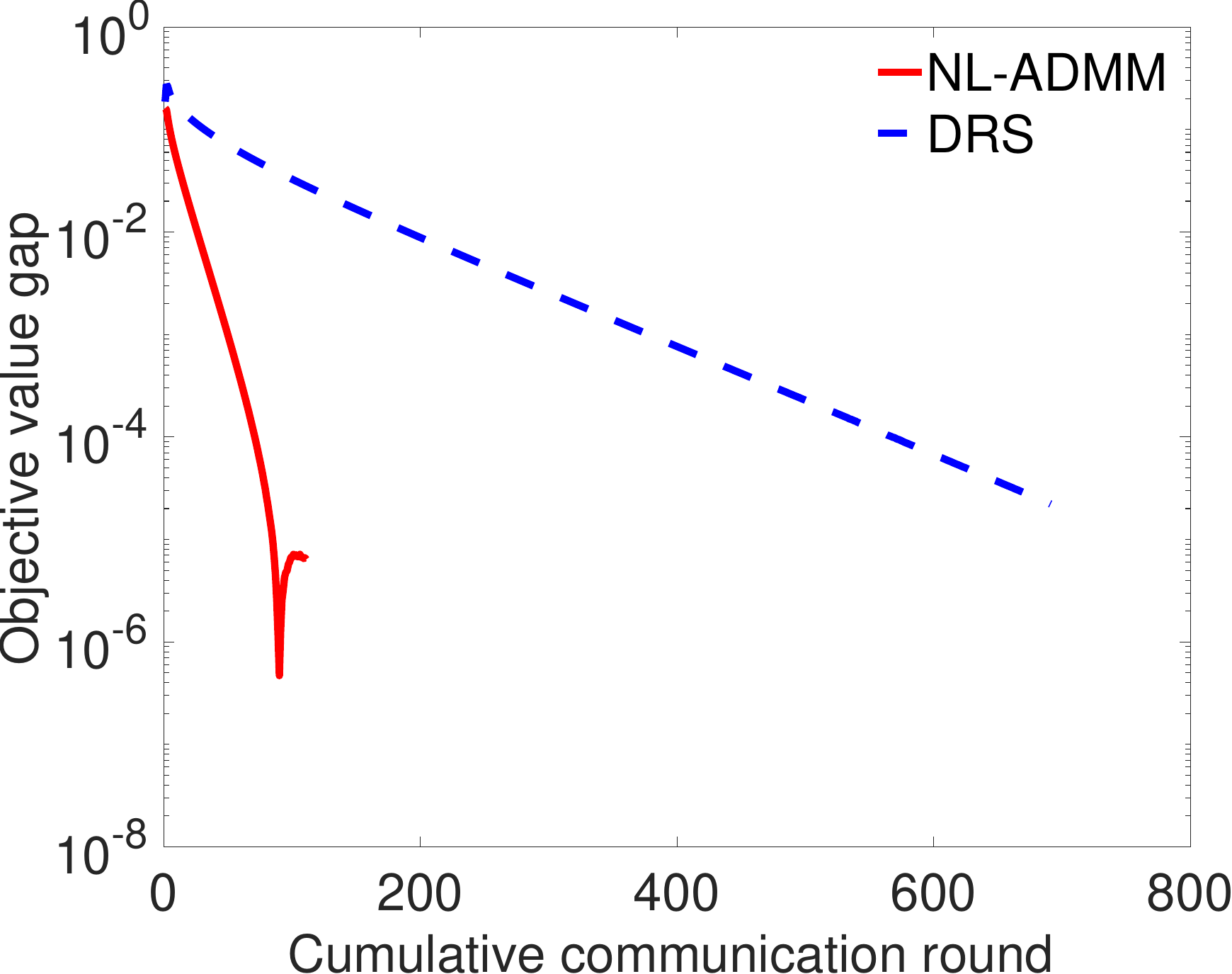}
    \label{fig:a9a-gap-p5}
  \end{subfigure}
  \begin{subfigure}{0.23\textwidth}
    \centering
    \includegraphics[width=\linewidth]{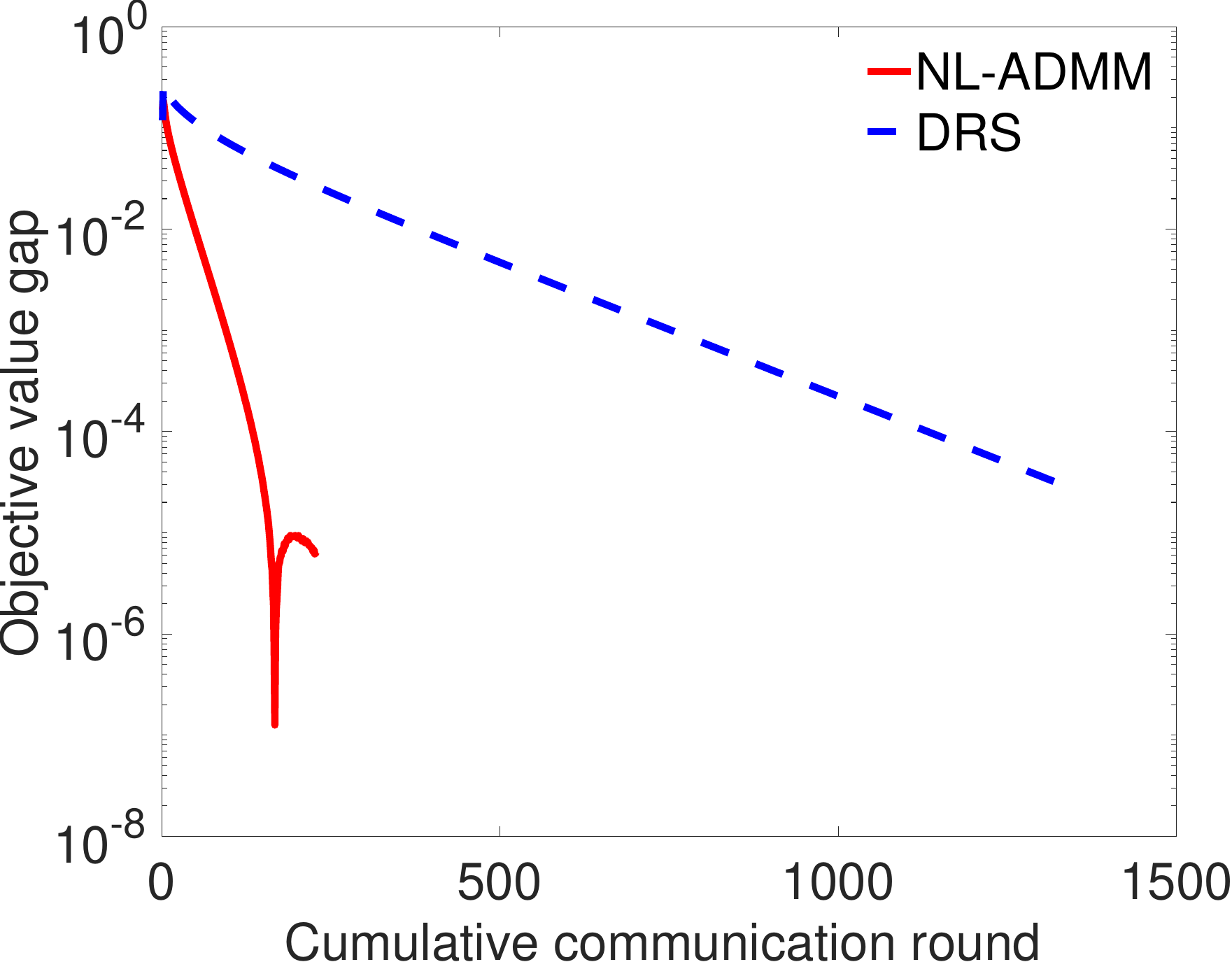}
    \label{fig:a9a-gap-p10}
  \end{subfigure}

  \vspace{0.4em}

  \centering
  \begin{subfigure}{0.23\textwidth}
    \includegraphics[width=\linewidth]{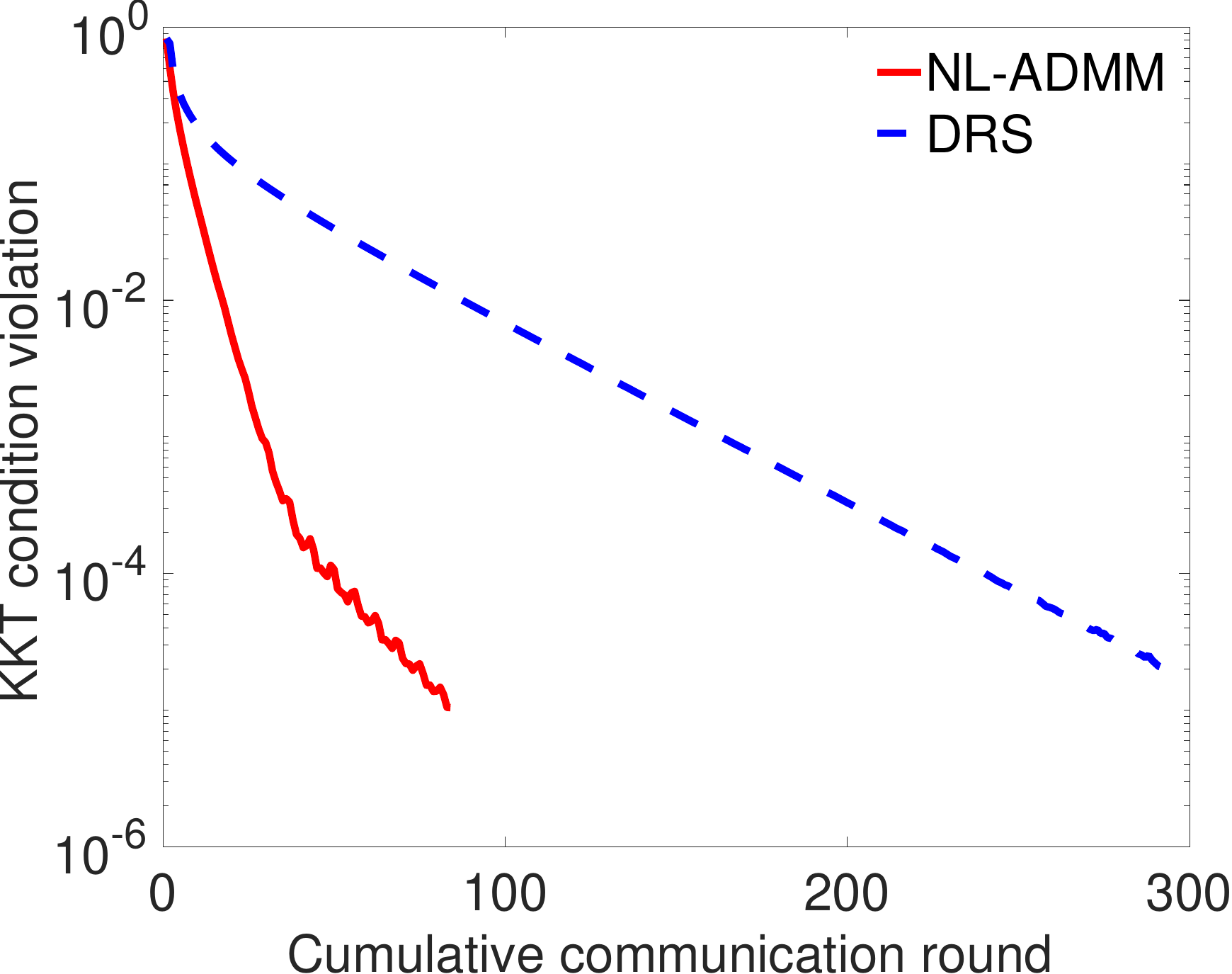}
    \caption{\(p=2\)}
    \label{fig:a9a-kkt-p2}
  \end{subfigure}
  \begin{subfigure}{0.23\textwidth}
    \centering
    \includegraphics[width=\linewidth]{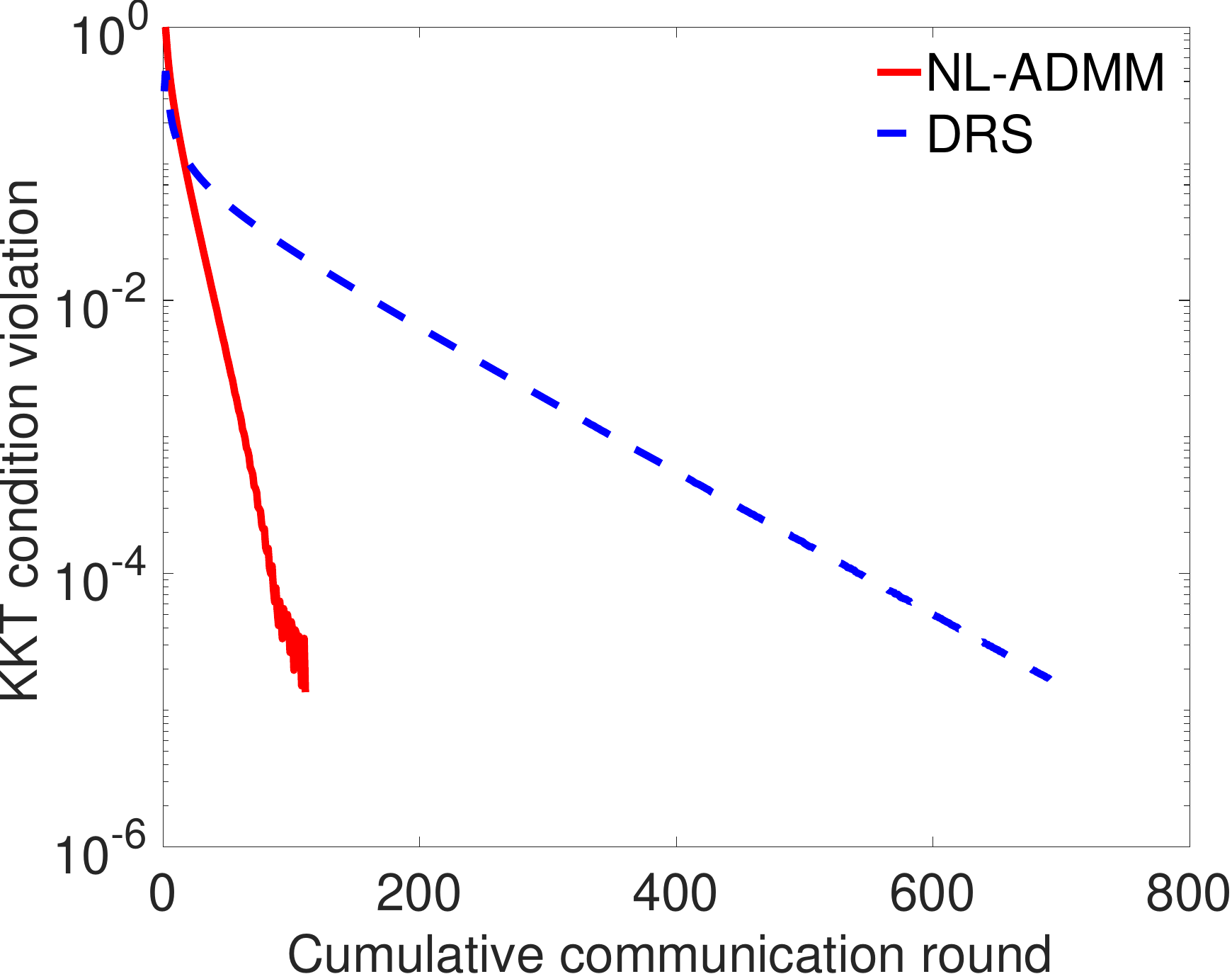}
    \caption{\(p=5\)}
    \label{fig:a9a-kkt-p5}
  \end{subfigure}
  \begin{subfigure}{0.23\textwidth}
    \includegraphics[width=\linewidth]{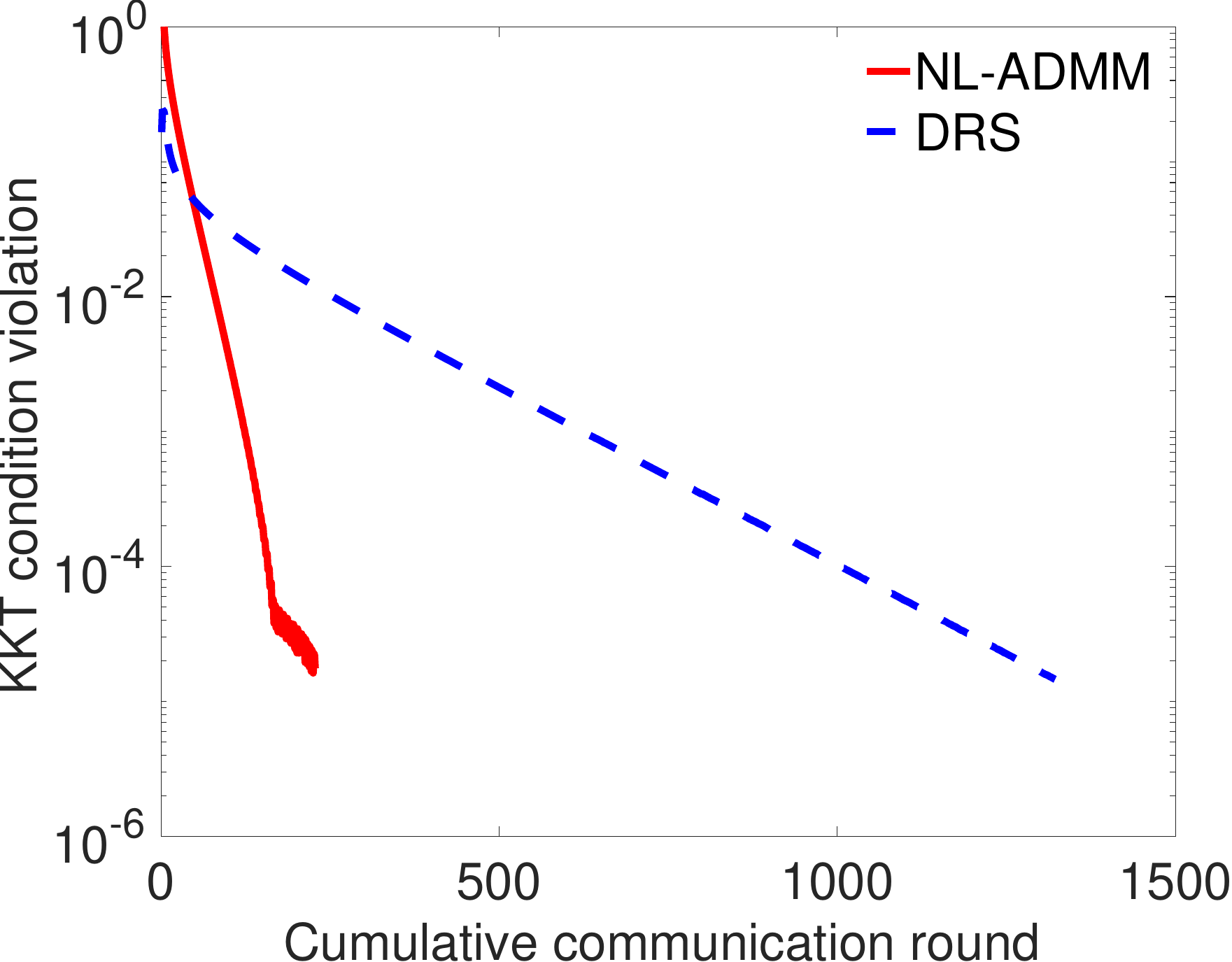}
    \caption{\(p=10\)}
    \label{fig:a9a-kkt-p10}
  \end{subfigure}

  \caption{Convergence behavior of the proposed nonlinear ADMM and the DRS baseline on solving instance of \eqref{eq:consensus} with the \texttt{a9a} dataset. 
  The top row shows objective value gaps $|f(\vx)+g(\tilde{\vy})-\big(f(\vx^{*})+g(\tilde{\vy}^{*})\big)|$, while the bottom row reports KKT violation defined as the sum of primal, dual, and complementarity residuals. Columns correspond to different number $p$ of workers. 
  Reference values $f(\vx^{*})+g(\tilde{\vy}^{*})$ are obtained from a high-accuracy ALM solution. The moderate-accuracy ALM run uses $56$ outer iterations and $10863$ communication rounds.}
  \label{fig:a9a_scale-2x3}
\end{figure}

\section{Concluding Remarks} 

Motivated by large-scale resource allocation and constrained learning tasks, we developed a nonlinear alternating direction method of multipliers (NL-ADMM) for two-block convex optimization problems with nonlinear inequality and affine equality constraints. Under a standard KKT existence assumption, we established global convergence of NL-ADMM and derived an ergodic $\mathcal{O}(1/k)$ convergence rate in both feasibility violation and objective optimality, thereby extending classical ADMM theory beyond linear constraint settings. 

From a practical perspective, NL-ADMM is particularly well suited for distributed environments in which communication cost is a primary concern. Numerical experiments on representative distributed resource allocation and constrained machine learning problems show that NL-ADMM consistently achieves higher communication efficiency than the classical augmented Lagrangian method and Douglas–Rachford Splitting method across a range of problem instances. 

To the best of our knowledge, this work represents 
the first systematic study of ADMM-type methods for convex optimization problems with general nonlinear functional constraints. Several directions merit further investigation, including stochastic (sub)gradient variants for data-intensive constrained learning and decentralized implementations for resource allocation over networked agents. Overall, this work positions NL-ADMM as a communication-efficient framework for large-scale distributed constrained optimization and provides a foundation for future stochastic and decentralized extensions.

\appendix

\section{Derivation of DRS update in Section \ref{sec:numerical}}
We consider the following problem:
\begin{equation}
\label{eq:DRS}
\min_{\vx_{0}, \vx,\vy} \sum_{j=1}^{p}f_{j}(\vx_{j})+\iota_{\mathcal{Y}}(\vy), 
\quad \st \vh(\vx) \leq \vy, \quad M_{0}\vx_{0}+M_{1}\vx=\mathbf{0},
\end{equation}
where $\mathcal{Y}=\{\vy: \sum_{j=1}^{p} \vy_j = 0\}$.
When both $M_{0}$ and $M_{1}$ are zero matrices, problem~\eqref{eq:DRS} reduces to the resource allocation problem~\eqref{eq:resource_allocation}.  
When
$M_{0}=-(\mathbf{1}\otimes I)$, and $M_{1}=I$,
it reduces to the consensus formulation~\eqref{eq:consensus}. 
Thus we only need to derive the DRS update for problem \eqref{eq:DRS} to recover the updates in \eqref{DRS_algorithm_1} and \eqref{DRS_algorithm_2}.

Introducing dual variables $\vu\geq 0$ and $\vv$, we define the  Lagrangian function of \eqref{eq:DRS} as follows
\begin{equation*}
    \mathcal{L}(\vx_{0},\vx,\vy,\vu,\vv):=\sum_{j=1}^{p}f_{j}(\vx_{j})+\iota_{\mathcal{Y}}(\vy)+\langle \vu,\vh(\vx)-\vy\rangle+\langle \vv, M_{0}\vx_{0}+M_{1}\vx\rangle.
\end{equation*}
Then the 
Lagrangian dual problem is 
\begin{equation*}
\max_{\vu\geq 0,\vv}\inf_{\vx_{0}, \vx,\vy}\mathcal{L}(\vx_{0},\vx,\vy,\vu,\vv).
\end{equation*}
Note that for the $\vy$-dependent part in the above formulation, we have 
\begin{equation*}
    \inf_{\vy}\iota_{\mathcal{Y}}(\vy)-\langle \vu,\vy\rangle
=\begin{cases}
0, & \text{if }\vu\in \text{span}\{\mathbf{1}\},\\
-\infty, & \text{otherwise}.
\end{cases} 
\end{equation*}
Also, for the $\vx_{0}$-dependent part, it holds 
\begin{equation*}
    \inf_{\vx_{0}}\langle \vv, M_{0}\vx_{0}\rangle
=\begin{cases}
0, & \text{if } M_{0}^{\top}\vv=\vzero,\\
-\infty, & \text{otherwise}.
\end{cases} 
\end{equation*}

Based on these observations, we define the following two functions:
\begin{equation*}
\begin{aligned}
d_{g}(\vu,\vv)&:=\iota_{\mathcal{U}}(\vu)+\iota_{\mathcal{V}}(\vv), \ \text{with } \mathcal{U}:={\{\vu:\vu=\alpha \mathbf{1},\alpha\in \mathbb{R}\}} \text{ and } \mathcal{V}:=\{\vv:M_{0}^{\top}\vv=\vzero\},\\
d_{f}(\vu,\vv)&:=-\inf_{\vx}\left(\sum_{j=1}^{p}f_{j}(\vx_{j})+\langle \vu,\vh(\vx)\rangle+\langle \vv, M_{1}\vx\rangle\right)+\iota_{+}(\vu).
\end{aligned}
\end{equation*}
Therefore, the Lagrangian dual problem becomes 
\begin{equation*}
    \max_{\vu,\vv}-d_{f}(\vu,\vv)-d_{g}(\vu,\vv),
\end{equation*}
which is equivalent to the following zero-inclusion problem
\begin{equation}
\label{DRS_optimality}
    \mathbf{0}\in \partial d_{f}(\vu,\vv)+\partial d_{g}(\vu,\vv).
\end{equation}
Then for a fixed $\eta\in (0,1]$, the DRS iteration takes the form
\begin{subequations}
\begin{align}
(\vu^{k},\vv^{k})&=\vJ_{\beta\partial d_{g}}(\vw^{k}),\label{DRS_update_1}\\
\vw^{k+1}&=(1-\eta)\vw^{k}+\eta(2\vJ_{\beta \partial d_{f}}-I)\left(2(\vu^{k},\vv^{k})-(\vw_{\vu}^{k}, \vw_{\vv}^{k})\right),\label{DRS_update_2}
\end{align}
\end{subequations}
where $\vw^{k}=(\vw_{\vu}^{k}, \vw_{\vv}^{k})$ and the resolvent operators are defined as
\begin{equation*}
    \mathbf{J}_{\beta \partial d_{g}}:=(I+\beta\partial d_{g})^{-1}, \quad \mathbf{J}_{\beta \partial d_{f}}:=(I+\beta\partial d_{f})^{-1}.
\end{equation*}

Below, we first find the closed-form expression for \eqref{DRS_update_1}. By definition of $\mathbf{J}_{\beta \partial d_{g}}$, we have
\begin{equation*}
\label{DRS_uv}
    (\vu^{k},\vv^{k})=\argmin_{\vu,\vv}\tfrac{1}{2\beta}\lVert \vu-\vw_{\vu}^{k}\rVert^{2}+\tfrac{1}{2\beta}\lVert \vv-\vw_{\vv}^{k}\rVert^{2}+d_{g}(\vu,\vv).
\end{equation*}
Then by the definition of $d_g$, the closed-form solution follows 
\begin{equation}
\label{DRS_v}
\vu^{k}=\tfrac{1}{p}\mathbf{1}\mathbf{1}^{\top}\vw_{\vu}^{k}, \quad    \vv^{k}=\Pi_{\mathcal{V}}(\vw_{\vv}^{k}).
\end{equation}
Second, we find the closed-form expression for \eqref{DRS_update_2}. For convenience, we define 
\begin{equation*}
\label{DRS_p}
    \vq^{k}=(\vq_{\vu}^{k},\vq_{\vv}^{k}):=2(\vu^{k},\vv^{k})-(\vw_{\vu}^{k},\vw_{\vv}^{k}), \quad 
    \vt^{k}=(\vt_{\vu}^{k},\vt_{\vv}^{k}):=\vJ_{\beta\partial d_{f}}(\vq_{\vu}^{k},\vq_{\vv}^{k}).
\end{equation*}
By the definition of $\mathbf{J}_{\beta\partial d_{f}}$ and $d_{f}(\vu,\vv)$, we have
\begin{equation*}
    (\vt_{\vu}^{k},\vt_{\vv}^{k})=\argmin_{\vt_{\vu}\geq \mathbf{0},\vt_{\vv}}\tfrac{1}{2\beta}\lVert \vt_{\vu}-\vq_{\vu}^{k}\rVert^{2}+\tfrac{1}{2\beta}\lVert \vt_{\vv}-\vq_{\vv}^{k}\rVert^{2}-\inf_{\vx}\left(\sum_{j=1}^{p}f_{j}(\vx_{j})+\langle \vt_{\vu},\vh(\vx)\rangle+\langle \vt_{\vv}, M_{1}\vx\rangle\right).
\end{equation*}
Note that 
\begin{equation}
\label{DRS_equivalence}
\begin{aligned}
    &\min_{\vt_{\vu}\geq \mathbf{0},\vt_{\vv}}\tfrac{1}{2\beta}\lVert \vt_{\vu}-\vq_{\vu}^{k}\rVert^{2}+\tfrac{1}{2\beta}\lVert \vt_{\vv}-\vq_{\vv}^{k}\rVert^{2}-\inf_{\vx}\left(\sum_{j=1}^{p}f_{j}(\vx_{j})+\langle \vt_{\vu},\vh(\vx)\rangle+\langle \vt_{\vv}, M_{1}\vx\rangle\right)\\
    \iff &\min_{\vt_{\vu}\geq \mathbf{0},\vt_{\vv}}\sup_{\vx}\left(-\sum_{j=1}^{p}f_{j}(\vx_{j})-\langle \vt_{\vu},\vh(\vx)\rangle-\langle \vt_{\vv}, M_{1}\vx\rangle\right)+\tfrac{1}{2\beta}\lVert \vt_{\vu}-\vq_{\vu}^{k}\rVert^{2}+\tfrac{1}{2\beta}\lVert \vt_{\vv}-\vq_{\vv}^{k}\rVert^{2}\\
    \iff &\sup_{\vx}-\sum_{j=1}^{p}f_{j}(\vx_{j})+\min_{\vt_{\vu}\geq \mathbf{0},\vt_{\vv}}\Big(-\langle \vt_{\vu},\vh(\vx)\rangle-\langle \vt_{\vv}, M_{1}\vx\rangle+\tfrac{1}{2\beta}\lVert \vt_{\vu}-\vq_{\vu}^{k}\rVert^{2}+\tfrac{1}{2\beta}\lVert \vt_{\vv}-\vq_{\vv}^{k}\rVert^{2}\Big).
\end{aligned}    
\end{equation}
Then we derive
\begin{equation*}
\begin{aligned}
    (\vt_{\vu}^{*}, \vt_{\vv}^{*})&=\argmin_{\vt_{\vu}\geq 0,\vt_{\vv}}-\langle \vt_{\vu},\vh(\vx)\rangle-\langle \vt_{\vv}, M_{1}\vx\rangle+\tfrac{1}{2\beta}\lVert \vt_{\vu}-\vq_{\vu}^{k}\rVert^{2}+\tfrac{1}{2\beta}\lVert \vt_{\vv}-\vq_{\vv}^{k}\rVert^{2}\\
    &=\left([\vq_{\vu}^{k}+\beta \vh(\vx)]_{+},\vq_{\vv}^{k}+\beta M_{1}\vx\right),
\end{aligned}
\end{equation*}
for which notice that the solution $(\vt_{\vu}^{*}, \vt_{\vv}^{*})$ depends on $\vx$. 
Plug the formulation of $ (\vt_{\vu}^{*}, \vt_{\vv}^{*})$ back into the last line of \eqref{DRS_equivalence} and simplify terms. We then obtain the following problem and denote $\vx^{k}$ as the optimal solution: 
\begin{equation}
\label{DRS_x}
\vx^{k}=\argmin_{\vx}\sum_{j=1}^{p}f_{j}(\vx_{j})-\tfrac{1}{2\beta}\lVert [-\vq_{\vu}^{k}-\beta \vh(\vx)]_{+}\rVert^{2}+\tfrac{\beta}{2}\lVert \vh(\vx)\rVert^{2}+\langle\vq_{\vu}^{k},\vh(\vx)\rangle+\tfrac{\beta}{2}\lVert M_{1}\vx\rVert^{2}+\langle \vq_{\vv}^{k},M_{1}\vx\rangle. 
\end{equation}
With the computation of $\vx^{k}$, we then let 
\begin{equation}
\label{DRS_w}
    (\vt_{\vu}^{k},\vt_{\vv}^{k})=\left ([\vq_{\vu}^{k}+\beta \vh(\vx^{k})]_{+},\vq_{\vv}^{k}+\beta M_{1}\vx^{k}\right ).
\end{equation}
Therefore, \eqref{DRS_update_2} becomes
\begin{equation*}
\label{DRS_omega}
\begin{aligned}
    \vw^{k+1}&=(1-\eta)\vw^{k}+\eta(2\vJ_{\beta \partial d_{f}}-I)\big(2(\vu^{k},\vv^{k})-(\vw_{\vu}^{k}, \vw_{\vv}^{k})\big)\\
    &=(1-\eta)\vw^{k}+\eta(2\vt^{k}-\vq^{k})\\
    &=(1-\eta)\vw^{k}+\eta\Big(2\vt^{k}-\left(2(\vu^{k},\vv^{k})-\vw^{k}\right)\Big)\\
    &=\vw^{k}+2\eta\left(\vt^{k}-(\vu^{k},\vv^{k})\right).
\end{aligned}
\end{equation*}

\subsection*{Recovering the updates in \eqref{DRS_algorithm_1} and \eqref{DRS_algorithm_2}} 
When $M_{0}$ and $M_{1}$ are zero matrices, \yx{the iterates corresponding to $\vv$ block will remain constant, and thus} we do not need to consider 
the variables and terms with respect to the equality constraint and the $\vx_{0}$ update. Due to the \yx{separable} block structure of $\vh(\vx)$, the solution $\vx^k$ in
\eqref{DRS_x} can be written as follows after simplification
\begin{equation*}
\vx_{j}^{k}=\argmin_{\vx_{j}}f_{j}(\vx_{j})+\tfrac{\beta}{2}\left ( \Big[h_{j}(\vx_{j})+\tfrac{\vq_{\vu,j}^{k}}{\beta}\Big]_{+}\right )^{2} ,\quad j=1,\dots,p.
\end{equation*}
Hence, for problem \eqref{eq:resource_allocation}, we have the following DRS update:  
\begin{subequations}
\begin{align}
\vu^{k}&= \frac{1}{p}\mathbf{1}\mathbf{1}^{\top}\vw^{k},\label{eq:drs-1-u}\\
\vq^{k}&=2\vu^{k}-\vw^{k},\\
\vx_{j}^{k}&=\argmin_{\vx_{j}}f_{j}(\vx_{j})+\tfrac{\beta}{2}\left ( \Big[h_{j}(\vx_{j})+\tfrac{\vq_{j}^{k}}{\beta}\Big]_{+}\right )^{2} ,\quad j=1,\dots,p,\\
\vt^{k}&=[\vq^{k}+\beta \vh(\vx^{k})]_{+},\\
\vw^{k+1} &= \vw^{k}+2\eta(\vt^{k}-\vu^{k}).\label{eq:drs-1-w}
\end{align}
\end{subequations} 
Now define $\tilde{\vw}^{k}:=\frac{\vw^{k}}{\beta}$, $\tilde{\vu}^{k}:=\frac{\vu^{k}}{\beta}$, $\tilde{\vq}^{k}:=\frac{\vq^{k}}{\beta}$, and $\tilde{\vt}^{k}:=\frac{\vt^{k}}{\beta}$. Then the update to $(\tilde{\vw}^{k}, \tilde{\vu}^{k}, \tilde{\vq}^{k}, \tilde{\vt}^{k})$ will reduce exactly to the updates in \eqref{eq:drs-u}-\eqref{eq:drs-w}. 

When $M_{0}=-(\mathbf{1}\otimes I)$ and $M_{1}=I$, according to \eqref{DRS_v}, we derive 
\begin{equation*}
    \vv^{k}=\vw_{\vv}^{k}-(\mathbf{1}\otimes I)\bar{\vw}_{\vv}^{k}, \ \text{with} \
\bar{\vw}_{\vv}^{k}=\frac{1}{p}\sum_{j=1}^{p}\vw_{\vv,j}^{k}.
\end{equation*}
Moreover, because of the separable block structure of $\vh(\vx)$, \eqref{DRS_x} becomes
\begin{equation*}
\vx_{j}^{k}=\argmin_{\vx_{j}}f_{j}(\vx_{j})+\tfrac{\beta}{2}\left ( \Big[h_{j}(\vx_{j})+\tfrac{\vq_{\vu,j}^{k}}{\beta}\Big]_{+}\right )^{2}
+\tfrac{\beta}{2}\Big\| \vx_j+\tfrac{\vq_{\vv,j}^{k}}{\beta} \Big\|^{2}, \quad j=1,\dots,p,
\end{equation*}
and \eqref{DRS_w} becomes
\begin{equation*}
(\vt_{\vu}^{k},\vt_{\vv}^{k})=\left ([\vq_{\vu}^{k}+\beta \vh(\vx^{k})]_{+},\vq_{\vv}^{k}+\beta \vx^{k}\right).    
\end{equation*}
Hence, for problem \eqref{eq:consensus}, we have the following DRS update:  
\begin{subequations}
\begin{align}
(\vu^{k}, \vv^{k}) &= \left(\tfrac{1}{p}\mathbf{1}\mathbf{1}^{\top}\vw_{\vu}^{k},\vw_{\vv}^{k}-(\mathbf{1}\otimes I)\bar{\vw}_{\vv}^{k}\right),\  \text{with} \
\bar{\vw}_{\vv}^{k}=\frac{1}{p}\sum_{j=1}^{p}\vw_{\vv,j}^{k},\\
\vq^{k}&=(\vq_{\vu}^{k},\vq_{\vv}^{k}) = \big(2\vu^{k}-\vw_{\vu}^{k}, 2\vv^{k}-\vw_{\vv}^{k}\big),\\
\vx_{j}^{k}&=\argmin_{\vx_{j}}f_{j}(\vx_{j})+\tfrac{\beta}{2}\left ( \Big[h_{j}(\vx_{j})+\tfrac{\vq_{\vu,j}^{k}}{\beta}\Big]_{+}\right )^{2}+\tfrac{\beta}{2}
\Big\|\vx_{j}+\tfrac{\vq_{\vv,j}^{k}}{\beta}\Big\|^{2},\quad j=1,\dots,p,\\
\vt^{k}&=(\vt_{\vu}^{k},\vt_{\vv}^{k})=\big([\vq_{\vu}^{k}+\beta \vh(\vx^{k})]_{+},\vq_{\vv}^{k}+\beta\vx^{k}\big),\\
\vw^{k+1} &= \vw^{k}+2\eta\big(\vt_{\vu}^{k}-\vu^{k},\vt_{\vv}^{k}-\vv^{k}\big).
\end{align}
\end{subequations}
Finally, similar to the updates in \eqref{eq:drs-1-u}-\eqref{eq:drs-1-w}, we scale the above iterates except $\vx^k$ by $\frac{1}{\beta}$ and recover those in \eqref{eq:DRS-uv}-\eqref{eq:DRS-w-p2}. 

\bibliographystyle{abbrv}
			\bibliography{optim, Pub_YXu}

@article{gao2019rpd-bcu,
  title={Randomized Primal-Dual Proximal Block Coordinate Updates},
  author={Gao, Xiang and Xu, Yangyang and Zhang, Shuzhong},
  journal={Journal of the Operations Research Society of China},
  volume={7},
  number={2},
  pages={205--250},
  year={2019}
}

@article{xu2019asynchronous-pdbcu,
  title={Asynchronous parallel primal-dual block coordinate update methods for affinely constrained convex programs},
  author={Xu, Yangyang},
  journal={Computational Optimization and Applications},
  pages={87--113},
  volume={72},
  number={1},
  year={2019}
}

@article{xu2018hybrid-Jac-GS,
  title={Hybrid {J}acobian and {G}auss--{S}eidel Proximal Block Coordinate Update Methods for Linearly Constrained Convex Programming},
  author={Xu, Yangyang},
  journal={SIAM Journal on Optimization},
  volume={28},
  number={1},
  pages={646--670},
  year={2018},
  publisher={SIAM}
}

@article{xu2015alternating,
  title={Alternating proximal gradient method for sparse nonnegative Tucker decomposition},
  author={Xu, Yangyang},
  journal={Mathematical Programming Computation},
  volume={7},
  pages={39--70},
  year={2015},
  publisher={Springer}
}

@article{xu2012alternating,
  title={An alternating direction algorithm for matrix completion with nonnegative factors},
  author={Xu, Yangyang and Yin, Wotao and Wen, Zaiwen and Zhang, Yin},
  journal={Frontiers of Mathematics in China},
  volume={7},
  pages={365--384},
  year={2012},
  publisher={Springer}
}

@article{hong2016convergence,
	title={Convergence analysis of alternating direction method of multipliers for a family of nonconvex problems},
	author={M. Hong and Z. Luo and M. Razaviyayn},
	journal={SIAM Journal on Optimization},
	volume={26},
	number={1},
	pages={337--364},
	year={2016},
	publisher={SIAM}
}

@article{monteiro2013iteration,
  title={Iteration-complexity of block-decomposition algorithms and the alternating direction method of multipliers},
  author={Monteiro, Renato DC and Svaiter, Benar F},
  journal={SIAM Journal on Optimization},
  volume={23},
  number={1},
  pages={475--507},
  year={2013},
  publisher={SIAM}
}

@article{lin2015global,
  title={On the global linear convergence of the ADMM with multiblock variables},
  author={Lin, Tianyi and Ma, Shiqian and Zhang, Shuzhong},
  journal={SIAM Journal on Optimization},
  volume={25},
  number={3},
  pages={1478--1497},
  year={2015},
  publisher={SIAM}
}

@article{davis2017three,
  title={A three-operator splitting scheme and its optimization applications},
  author={Davis, Damek and Yin, Wotao},
  journal={Set-valued and variational analysis},
  volume={25},
  number={4},
  pages={829--858},
  year={2017},
  publisher={Springer}
}

@article{he2012alternating,
  title={Alternating direction method with Gaussian back substitution for separable convex programming},
  author={He, Bingsheng and Tao, Min and Yuan, Xiaoming},
  journal={SIAM Journal on Optimization},
  volume={22},
  number={2},
  pages={313--340},
  year={2012},
  publisher={SIAM}
}

@article{yang2011alternating,
  title={Alternating direction algorithms for $\ell_1$-problems in compressive sensing},
  author={Yang, Junfeng and Zhang, Yin},
  journal={SIAM journal on scientific computing},
  volume={33},
  number={1},
  pages={250--278},
  year={2011},
  publisher={SIAM}
}

@article{themelis2020douglas,
  title={Douglas--Rachford splitting and ADMM for nonconvex optimization: Tight convergence results},
  author={Themelis, Andreas and Patrinos, Panagiotis},
  journal={SIAM Journal on Optimization},
  volume={30},
  number={1},
  pages={149--181},
  year={2020},
  publisher={SIAM}
}

@article{hong2017distributed,
  title={A distributed, asynchronous, and incremental algorithm for nonconvex optimization: An ADMM approach},
  author={Hong, Mingyi},
  journal={IEEE Transactions on Control of Network Systems},
  volume={5},
  number={3},
  pages={935--945},
  year={2017},
  publisher={IEEE}
}

@article{barber2024convergence,
  title={Convergence for nonconvex ADMM, with applications to CT imaging},
  author={Barber, Rina Foygel and Sidky, Emil Y},
  journal={Journal of Machine Learning Research},
  volume={25},
  number={38},
  pages={1--46},
  year={2024}
}

@article{xu2021iteration,
  title={Iteration complexity of inexact augmented {Lagrangian} methods for constrained convex programming},
  author={Y. Xu},
  journal={Mathematical Programming},
  volume={185},
  number={1},
  pages={199--244},
  year={2021},
  publisher={Springer}
}

@article{he2012-rate-drs,
  title={On the ${O}(1/n)$ Convergence Rate of the Douglas--{Rachford} Alternating Direction Method},
  author={He, Bingsheng and Yuan, Xiaoming},
  journal={SIAM Journal on Numerical Analysis},
  volume={50},
  number={2},
  pages={700--709},
  year={2012},
  publisher={SIAM}
}

@article{rigollet2011neyman,
  title={Neyman-pearson classification, convexity and stochastic constraints},
  author={Rigollet, Philippe and Tong, Xin},
  journal={Journal of Machine Learning Research},
  volume={12},
  number={Oct},
  pages={2831--2855},
  year={2011}
}

@article{melo2017iteration,
	title={Iteration-complexity of a {Jacobi}-type non-{Euclidean} {ADMM} for multi-block linearly constrained nonconvex programs},
	author={J. G. Melo and R. D. C. Monteiro},
	journal={Preprint,  arXiv:1705.07229},
	year={2017}
}

@article{gabay1976dual,
  title={A dual algorithm for the solution of nonlinear variational problems via finite element approximation},
  author={Gabay, Daniel and Mercier, Bertrand},
  journal={Computers $\&$ Mathematics with Applications},
  volume={2},
  number={1},
  pages={17--40},
  year={1976},
  publisher={Elsevier}
}

@article{deng2017parallel,
  title={Parallel multi-block {ADMM} with $O(1/k)$ convergence},
  author={Deng, Wei and Lai, Ming-Jun and Peng, Zhimin and Yin, Wotao},
  journal={Journal of Scientific Computing},
  volume={71},
  number={2},
  pages={712--736},
  year={2017},
  publisher={Springer}
}

@article{rockafellar1976augmented,
  title={Augmented {Lagrangian}s and applications of the proximal point algorithm in convex programming},
  author={Rockafellar, R Tyrrell},
  journal={Mathematics of operations research},
  volume={1},
  number={2},
  pages={97--116},
  year={1976},
  publisher={INFORMS}
}

@article{boyd2011distributed,
  title={Distributed optimization and statistical learning via the alternating direction method of multipliers},
  author={Boyd, Stephen and Parikh, Neal and Chu, Eric and Peleato, Borja and Eckstein, Jonathan},
  journal={Foundations and Trends{\textregistered} in Machine Learning},
  volume={3},
  number={1},
  pages={1--122},
  year={2011},
  publisher={Now Publishers Inc.}
}

@incollection{robbins1971convergence,
  title={A convergence theorem for non negative almost supermartingales and some applications},
  author={Robbins, Herbert and Siegmund, David},
  booktitle={Optimizing methods in statistics},
  pages={233--257},
  year={1971},
  publisher={Elsevier}
}

@inproceedings{falsone2016distributed,
  title={Distributed constrained convex optimization and consensus via dual decomposition and proximal minimization},
  author={Falsone, Alessandro and Margellos, Kostas and Garatti, Simone and Prandini, Maria},
  booktitle={2016 IEEE 55th Conference on Decision and Control (CDC)},
  pages={1889--1894},
  year={2016},
  organization={IEEE}
}

@article{chen2016distributed,
  title={Distributed finite-time economic dispatch of a network of energy resources},
  author={Chen, Gang and Ren, Jianghong and Feng, E Ning},
  journal={IEEE Transactions on Smart Grid},
  volume={8},
  number={2},
  pages={822--832},
  year={2016},
  publisher={IEEE}
}

@article{halabian2019distributed,
  title={Distributed resource allocation optimization in 5G virtualized networks},
  author={Halabian, Hassan},
  journal={IEEE Journal on Selected Areas in Communications},
  volume={37},
  number={3},
  pages={627--642},
  year={2019},
  publisher={IEEE}
}

@article{donini2018empirical,
  title={Empirical risk minimization under fairness constraints},
  author={Donini, Michele and Oneto, Luca and Ben-David, Shai and Shawe-Taylor, John S and Pontil, Massimiliano},
  journal={Advances in neural information processing systems},
  volume={31},
  year={2018}
}

@article{eckstein1992douglas,
  title={On the Douglas—Rachford splitting method and the proximal point algorithm for maximal monotone operators},
  author={Eckstein, Jonathan and Bertsekas, Dimitri P},
  journal={Mathematical programming},
  volume={55},
  number={1},
  pages={293--318},
  year={1992},
  publisher={Springer}
}

@article{chen2016direct,
  title={The direct extension of ADMM for multi-block convex minimization problems is not necessarily convergent},
  author={Chen, Caihua and He, Bingsheng and Ye, Yinyu and Yuan, Xiaoming},
  journal={Mathematical Programming},
  volume={155},
  number={1},
  pages={57--79},
  year={2016},
  publisher={Springer}
}

@article{sun2024dual,
  title={Dual descent augmented Lagrangian method and alternating direction method of multipliers},
  author={Sun, Kaizhao and Sun, Xu Andy},
  journal={SIAM Journal on Optimization},
  volume={34},
  number={2},
  pages={1679--1707},
  year={2024},
  publisher={SIAM}
}

@article{hien2024multiblock,
  title={Multiblock ADMM for nonsmooth nonconvex optimization with nonlinear coupling constraints},
  author={Hien, Le Thi Khanh and Papadimitriou, Dimitri},
  journal={Optimization},
  pages={1--26},
  year={2024},
  publisher={Taylor \& Francis}
}

@article{el2025convergence,
  title={Convergence rates for an inexact linearized ADMM for nonsmooth nonconvex optimization with nonlinear equality constraints},
  author={El Bourkhissi, Lahcen and Necoara, Ion},
  journal={Computational Optimization and Applications},
  pages={1--39},
  year={2025},
  publisher={Springer}
}

@article{cohen2022dynamic,
  title={A dynamic alternating direction of multipliers for nonconvex minimization with nonlinear functional equality constraints},
  author={Cohen, Eyal and Hallak, Nadav and Teboulle, Marc},
  journal={Journal of Optimization Theory and Applications},
  volume={193},
  number={1},
  pages={324--353},
  year={2022},
  publisher={Springer}
}

@article{zhu2024first,
  title={A first-order primal-dual method for nonconvex constrained optimization based on the augmented Lagrangian},
  author={Zhu, Daoli and Zhao, Lei and Zhang, Shuzhong},
  journal={Mathematics of Operations Research},
  volume={49},
  number={1},
  pages={125--150},
  year={2024},
  publisher={INFORMS}
}

@article{glowinski1975approximation,
  title={Sur l'approximation, par {\'e}l{\'e}ments finis d'ordre un, et la r{\'e}solution, par p{\'e}nalisation-dualit{\'e} d'une classe de probl{\`e}mes de Dirichlet non lin{\'e}aires},
  author={Glowinski, Roland and Marroco, Americo},
  journal={Revue fran{\c{c}}aise d'automatique, informatique, recherche op{\'e}rationnelle. Analyse num{\'e}rique},
  volume={9},
  number={R2},
  pages={41--76},
  year={1975},
  publisher={EDP Sciences}
}

@article{he2018class,
  title={A class of ADMM-based algorithms for three-block separable convex programming},
  author={He, Bingsheng and Yuan, Xiaoming},
  journal={Computational Optimization and Applications},
  volume={70},
  number={3},
  pages={791--826},
  year={2018},
  publisher={Springer}
}

@article{lin2018global,
  title={Global convergence of unmodified 3-block ADMM for a class of convex minimization problems},
  author={Lin, Tianyi and Ma, Shiqian and Zhang, Shuzhong},
  journal={Journal of Scientific Computing},
  volume={76},
  number={1},
  pages={69--88},
  year={2018},
  publisher={Springer}
}

@article{chen2019unified,
  title={A unified algorithmic framework of symmetric Gauss-Seidel decomposition based proximal ADMMs for convex composite programming},
  author={Chen, Liang and Sun, Defeng and Toh, Kim-Chuan and Zhang, Ning},
  journal={Journal of Computational Mathematics},
  pages={739--757},
  year={2019},
  publisher={JSTOR}
}

@article{wang2019global,
  title={Global convergence of ADMM in nonconvex nonsmooth optimization},
  author={Wang, Yu and Yin, Wotao and Zeng, Jinshan},
  journal={Journal of Scientific Computing},
  volume={78},
  number={1},
  pages={29--63},
  year={2019},
  publisher={Springer}
}

@article{chang2011libsvm,
  author    = {Chih-Chung Chang and Chih-Jen Lin},
  title     = {{LIBSVM}: A library for support vector machines},
  journal   = {ACM Transactions on Intelligent Systems and Technology},
  volume    = {2},
  year      = {2011},
  pages     = {27:1--27:27},
  note      = {Software available at http://www.csie.ntu.edu.tw/~cjlin/libsvm}
}

@article{aybat2019distributed,
  title={A distributed ADMM-like method for resource sharing over time-varying networks},
  author={Aybat, Necdet Serhat and Hamedani, Erfan Yazdandoost},
  journal={SIAM Journal on Optimization},
  volume={29},
  number={4},
  pages={3036--3068},
  year={2019},
  publisher={SIAM}
}
\end{document}